\documentclass[11pt,reqno]{amsart}
\usepackage{amsthm,amssymb,amsmath}
\usepackage{textcomp}
\usepackage{xcolor}
\usepackage{enumitem}
\usepackage[colorlinks=true, pdfstartview=FitV, linkcolor=blue, citecolor=blue, urlcolor=blue,pagebackref=false]{hyperref}
\usepackage{ulem}
\usepackage{esint}
\usepackage{makecell}

\usepackage[T1]{fontenc}
\usepackage{mlmodern}

\usepackage{fix-cm}

\newtheorem{theorem}{Theorem}
\newtheorem*{theorem*}{Theorem}
\newtheorem{lemma}{Lemma}[section]

\newtheorem{proposition}{Proposition}
\newtheorem{claim}{\it Claim}

\newtheorem*{definition*}{\bf Definition}

\theoremstyle{definition}

\newtheorem{example}{\bf Example}
\newtheorem*{examples}{\bf Examples}
\newtheorem{remark}{Remark}

\newtheorem*{remark*}{Remark}

\newtheorem*{example*}{\bf Example}

\newcommand{\loc}{{\rm loc}}

\numberwithin{equation}{section}

\makeatletter
\renewcommand{\normalsize}{%
  \@xsetfontsize\normalsize 6%
  \fontsize{10.35pt}{4.4mm}\selectfont
  \@adjustvertspacing\let\@listi\@listI}
\makeatother
\AtBeginDocument{\normalsize}

\begin{document}

\title{Non-local SDEs, critical drifts and local blow-ups}

\author{Damir Kinzebulatov and Chengjun Yue}

\address{Universit\'{e} Laval, D\'{e}partement de math\'{e}matiques et de statistique, Qu\'{e}bec, QC, Canada}

\email{damir.kinzebulatov@mat.ulaval.ca}

\email{chengjun.yue.1@ulaval.ca}

\thanks{D.K. is supported by NSERC Discovery grant RGPIN-2024-04236, C.\,Y.\,is supported by the CRM-Laval postdoctoral fellowship.}

\begin{abstract}
The paper is concerned with  $\alpha$-stable SDEs with singular time-inhomogeneous general drift. Our drifts satisfy e.g.\,a close to minimal possible critical Morrey class condition that covers local blow-ups of the type arising in some particle systems with strong attracting interactions.
We develop detailed Sobolev regularity theory
for the forward and backward equations, construct Feller propagator and prove weak existence and conditional uniqueness for the  $\alpha$-stable SDE.
\end{abstract}



\maketitle

\setcounter{tocdepth}{1}

\tableofcontents

\section{Introduction} This paper is concerned with weak solution theory for SDEs with singular drift:
\begin{equation}
\label{sde1}
X_t = x - \int_0^t b(s,X_s) ds+ Z_t, \quad t\ge 0,
\end{equation}
where $Z_t$ is an isotropic $\alpha$-stable process in $\mathbb R^d$ ($d \geq 2$) generated by the fractional Laplacian $(-\Delta)^{\frac{\alpha}{2}}$, $1<\alpha<2$. The drift $b:\mathbb R^{1+d} \rightarrow \mathbb R^d$ is a Borel measurable vector field satisfying a close to the minimal possible condition of invariance with respect to the diffusive scaling, cf.\,\eqref{b_s_cond}.
We consider drifts that can be singular enough to introduce local in space blow-up effects due to a strong attraction to a point or a submanifold, and can have critical singularities in time. Our drifts are general, i.e.\,they are not assumed to have any particular structure such as a gradient form or divergence of appropriate sign. 

SDEs with attracting singularities in the drift arise in a number of physical models, including the SDE behind the non-classical Feynman-Kac formula for the 2D $\delta$-Bose gas \cite{C} and the nonlinear SDE for the Keller-Segel model of chemotaxis, see \cite{FJ} for $\alpha=2$, and \cite{E} for $\alpha<2$. Another class of models where SDEs with singular (non-attracting) drift appear is related to the passive tracer model in hydrodynamics where the drift is the velocity field obtained by solving 3D Navier-Stokes equations \cite{ZZ}. 
It is well known that the non-locality of the diffusion operator $(-\Delta)^{\frac{\alpha}{2}} + b \cdot \nabla$ behind SDE \eqref{sde1} makes the analysis of \eqref{sde1} and of the corresponding Kolmogorov equations easier in some aspects (e.g.\,the heat kernel of the fractional Laplacian vanishes at infinity only polynomially, which allows for some scaling arguments that do not exist when $\alpha=2$, cf.\,\cite{JW}) and considerably more difficult in some other aspects. We refer e.g.\,to \cite{CV,KSS, MM,S} for extensions of De Giorgi's and Nash's methods to non-local PDEs.
The literature on non-local PDEs and stable-like diffusion processes \cite{CF} is motivated, in particular, by physical models involving anomalous diffusions \cite{KRS}. A part of it is the literature on non-local ($=$ stable-driven) SDEs, which we will discuss in a moment.

Our main assumption on $b$ is the following Rellich-type condition:
\vspace*{2mm}
\begin{align}
&\text{$b \cdot \nabla$ is strongly bounded by $\lambda + \partial_t+(-\Delta)^{\frac{\alpha}{2}}$} \notag \\[2mm]
& \qquad \; \text{in $L^p\big(\mathbb R^{1+d},(1+|b|)^{-p+1}dtdx\big)$, \quad $p>d+1$,} \notag \\[2mm]
& \qquad \; \text{with sufficiently small bound, for some constant $\lambda>0$}, \label{hyp1}
\end{align}

\noindent where the role of the weight $(1+|b|)^{-p+1}$ is to desingularize the drift; see \eqref{main_est}-\eqref{half2} for the complete form of \eqref{hyp1}.  
This condition seems to be rather well suited for handling particle systems with attracting interactions that are strong enough to introduce local blow-up effects: as the strength of attraction between the particles crosses a certain critical threshold, particles start to collide in finite time a.s.\,and the particle system ceases to have a global-in-time weak solution (hence the assumption of smallness in \eqref{hyp1}).

An elementary sufficient condition for \eqref{hyp1} (i.e.\,\eqref{main_est}-\eqref{half2}) is the following parabolic Morrey class condition:
\begin{equation}
\label{b_s_cond}
\| |b|^{\frac{1}{\alpha-1}} \|_{E_{1+\epsilon}} \text{ is sufficiently small}.
\end{equation}
Let us emphasize that replacing the drift $b$ by $cb$, where $c$ is a positive constant, rescales the Morrey norm $\||b|^{\frac{1}{\alpha-1}}\|_{E_{1+\epsilon}}$ by $c^{\frac{1}{\alpha-1}}$. In particular, taking $c$ small, we can make the Morrey norm of the new drift as small as needed.

\begin{definition*}
    A function $v\in L_{loc}^{1+\epsilon}(\mathbb{R}^{1+d})$, $\epsilon>0$, is said to be in the Morrey class $E_{1+\epsilon}=E_{1+\epsilon}^\alpha$ if it has finite Morrey norm
    $$
    \|v\|_{E_{1+\epsilon}} :=
    \sup_{\rho>0, (t,x)\in \mathbb{R}^{1+d}}
    \rho\left(
    \frac{1}{\rho^{d+\alpha} }\int_{C^\alpha_\rho(t,x)} |v(s,y)|^{1+\epsilon} dsdy
    \right)^{\frac{1}{1+\epsilon}}<\infty,
    $$
    where the supremum is taken over all $\alpha$-parabolic cylinders $C^\alpha_\rho(t,x):=\{(s,y) \in \mathbb{R}^{1+d} \mid t \leq s \leq t+\rho^{\alpha}, \,|y-x| \leq \rho\}.$
    \end{definition*}

	 See Proposition \ref{prop1} for details. It refines some results of Adams and  Krylov, and is of interest on its own.

\begin{examples}

1.~Condition \eqref{b_s_cond} includes e.g.\,$$b(x)=\pm \kappa \frac{x}{|x|^\alpha}, \quad \kappa>0$$ in which case $\| |b|^{\frac{1}{\alpha-1}} \|^{\alpha-1}_{E_{1+\epsilon}}=c(d,\epsilon)\kappa$, so $\kappa$ has to be sufficiently small. This drift introduces strong attraction to the origin in \eqref{sde1} ($+$) or strong repulsion ($-$).

2.~More generally, \eqref{b_s_cond} includes $|b| \in L^{\frac{d}{\alpha-1},\infty}$, i.e.\,weak Lebesgue class, with sufficiently small weak Lebesgue norm, in particular, a drift with an everywhere dense singular locus
\begin{equation}
\label{b_ser}
b(t,x)=\sum_{z_i \in \mathbb Q^d}\kappa_i a_i(t,x) \frac{x-z_i}{|x-z_i|^\alpha}, \quad \|a_i\|_{L^\infty(\mathbb R^{1+d})} \leq 1,
\end{equation}
where $\kappa_i>0$ and $\sum_i \kappa_i$ is sufficiently small.

3.~Another example is the following many-particle drift: $b:\mathbb R^{dN} \rightarrow \mathbb R^{dN}$, $b=(b^i)_{i=1}^N$,
$$
b^i(x):=\frac{\kappa}{N}\sum_{j=1, j \neq i}^N \frac{x^i-x^j}{|x^i-x^j|^\alpha}, \quad \kappa>0, \quad x=(x^1,\dots,x^N),
$$
(see the proof of Theorem \ref{thm3} for details).
This drift has singularities along the collision hyperplanes $x^i=x^j$ and is not anywhere close to $L^{\frac{dN}{\alpha-1}}(\mathbb R^{dN})$.

4.~Another example is given by a time-inhomogeneous Borel measurable vector field $b$ satisfying, for some fixed $T>0$ and $ t_0\in[-T,T]$
$$|b(t,x)|\leq \kappa\mathbf{1}_{[-T,T]}|t-t_0|^{-\frac{\alpha-1}{\alpha}},
\quad \kappa>0,  \quad \text{ for all $(t,x) \in \mathbb R^{1+d}$. }
$$

The previous examples can be combined, i.e.\,we can consider $b=\sum_{i=1}^\infty b_i$ and still obtain an admissible drift, as long as  $\sum_{i=1}^\infty \kappa_i$ is sufficiently small.

\end{examples}

We can also verify \eqref{hyp1} directly by means of various fractional Hardy inequalities \cite{FLS, KPS}, which provide  control over what ``sufficiently small'' in \eqref{hyp1}, \eqref{b_s_cond} is.

\medskip

Our result, stated briefly, is as follows (for detailed statements see Theorems \ref{thm1}, \ref{thm2}). 

\begin{theorem*}
If the unbounded part of the drift $b$ satisfies \eqref{b_s_cond}, then the following are true:
\begin{enumerate}
\item[{\rm (\textit{i})}]
For every $x \in \mathbb R^d$, SDE \eqref{sde1} has a conditionally unique weak solution in a large class. These weak solutions constitute a strong Markov family. In fact, they determine a strongly continuous Feller propagator.

\medskip 

\item[{\rm (\textit{ii})}]
There is a detailed Sobolev regularity theory of the  Kolmogorov backward and forward equations corresponding to SDE \eqref{sde1} that, in particular, yields the uniqueness of the law of $X_t$  among laws having a sufficiently regular density satisfying the Kolmogorov forward equation in the weak sense. 
\end{enumerate}

\end{theorem*}

\begin{example}[Counterexample to weak existence]
\label{counter_ex}
The smallness condition in \eqref{b_s_cond} is necessary. Namely, consider SDE \eqref{sde1} with the following drift that pushes the solution $X_t$ towards the origin:
	$$
	b(x)=\kappa \frac{x}{|x|^{\alpha}},\qquad \kappa>0, \qquad b(0)=0.
	$$
	Fix some $\beta\in(1,\alpha)$ and put $$c_{d,\alpha,\beta}:=-2^\alpha \frac{\Gamma(\frac{\alpha-\beta}{2})\Gamma(\frac{d+\beta}{2})}{\Gamma(-\frac{\beta}{2})\Gamma(\frac{d+\beta-\alpha}{2})}>0.$$ This constant is fixed by the identity
	$$
	-(-\Delta)^{\frac{\alpha}{2}} |x|^\beta = c_{d,\alpha,\beta}|x|^{\beta-\alpha}
$$
that follows e.g.\,from the L\'{e}vy representation of the fractional Laplacian \eqref{levy} and the rotational invariance of the latter and of $|x|^\beta$.
	If the strength of attraction towards the origin satisfies
\begin{equation}
\label{kappa_too_large}
	\kappa>\frac{c_{d,\alpha,\beta}}{\beta},
\end{equation}
	then the SDE \eqref{sde1}
	has no weak solution started from $x=0$. Indeed, supposing that there is a weak solution, and applying formally It\^{o}'s formula \eqref{ito} to $|X_t|^\beta$, one obtains
	\begin{align}
\mathbb E |X_{t\wedge\tau_R\wedge\sigma_m}|^\beta
&=
\mathbb E\int_0^{t\wedge\tau_R\wedge\sigma_m}
(-(-\Delta)^{\frac{\alpha}{2}} |X_s|^\beta-b(X_s)\cdot \nabla |X_s|^\beta)\,ds\notag\\
&=\mathbb E\int_0^{t\wedge\tau_R\wedge\sigma_m}
(c_{d,\alpha,\beta}-\kappa\beta\big)|X_s|^{\beta-\alpha}\,ds, 	\label{ito_cor}
\end{align}
where $\tau_R$, $\sigma_m$ are stopping times needed to address the growth of $|x|^\beta$ at infinity, see details in Appendix \ref{counter_app}. 
Now, if \eqref{kappa_too_large} holds, the left-hand side \eqref{ito_cor} is positive and the right-hand side is negative, which leads to a contradiction. 
This argument, as it is written, is formal since $|x|^\beta$ is continuous but is not smooth at the origin, so a justification of \eqref{ito_cor} is required (Appendix \ref{counter_app}).
\end{example}

SDE \eqref{sde1} can be viewed as a model (``two-particle'') case of the interacting system
\begin{equation}
\label{syst1}
X_t^i=X_0^i-\frac{1}{N}\sum_{j=1, j \neq i}^N \int_0^t K(X_s^i-X_s^j)ds+Z_t^i,
\end{equation}
where $Z_t^i$ are independent $\alpha$-stable processes. We can reverse this relation and embed \eqref{syst1} in the class of our SDEs in the case $\alpha=2$, and also, to some extent, in the case $1<\alpha<2$ (Section \ref{particle_sect}). In the Brownian case $\alpha=2$, $d=2$, $K(x)=\kappa x|x|^{-2}\mathbf{1}_{x \neq 0}$, particle system \eqref{syst1} becomes the finite-particle approximation of the celebrated Keller-Segel model of chemotaxis. Some sharp and  detailed results are obtained in \cite{FJ}, see also \cite{FT}, by exploiting the special form of the interaction kernel. Our condition on the drift provides us with flexible means for modifying the former, i.e.\,we can multiply the interaction kernel by a function having $L^\infty$ norm less than or equal to one and still have \eqref{b_s_cond} satisfied and, more generally, \eqref{hyp1}. Regarding time-inhomogeneous drifts, the dynamics of a typical particle in \eqref{syst1} as  $N \rightarrow \infty$ is described by the McKean-Vlasov SDE
 \begin{equation}
\label{mv}
X_t=X_0-\int_0^t \big(K \star \eta(s)\big)(X_s)ds+Z_t,
\end{equation}
where $\eta(t,\cdot)$ is the density of the law of $X_t$ and $\star$ is the convolution in the spatial variables \cite{CJM}. Despite the regularizing effect of the convolution, this time-inhomogeneous drift is still highly singular near $t=0$ if e.g.\,${\rm Law}X_0=\delta_x$. However, if the  time-homogeneous interaction kernel $K$ satisfies \eqref{b_s_cond}, then the convolution with $\eta$ also satisfies \eqref{b_s_cond} (Section \ref{particle_sect}).

To our knowledge, the present paper and the elliptic weakly form-bounded theory  \cite{KM}  are the first to treat $\alpha$-stable SDEs with \textit{general}\footnote{That is, without necessarily  profiting from the special structure of the drift, e.g.\,the regularizing effect of the convolution in distribution-dependent drifts and the regularity of the density of the initial distribution.} drifts that can be singular enough to reach  local blow-up phenomena.

In the Brownian case $\alpha=2$, more is known; see discussion of the existing literature below.

\subsection{Literature}
The following is a brief survey of the literature on non-local (stable-driven) SDEs with irregular drift.

\smallskip

1.~For bounded drifts $|b| \in L^\infty$,  the weak existence and uniqueness in law for $1<\alpha<2$ are due to  Komatsu \cite{K}.

2. Portenko \cite{P} and Podolynny-Portenko \cite{PP} attained a sharp condition on the Lebesgue scale up to the strict inequality:
\begin{equation}
\label{PP_cond}
|b| \in L^q(\mathbb R^d)+L^\infty(\mathbb R^d), \quad q>\frac{d}{\alpha-1}.
\end{equation}
The time-inhomogeneous case,
\begin{align*}
	|b|\in L^r(\mathbb R, L^q(\mathbb{R}^d)) + L^\infty(\mathbb R^{1+d}), \quad
	\frac{d}{q}+\frac{\alpha}{r}<\alpha-1,
\end{align*}
was handled in \cite{J2}, see also the  recent developments in \cite{FJM}.

3.~Kim-Song \cite{KS} and Chen-Wang \cite{CW} considered Kato class drifts, i.e.
$$
\lim_{r\rightarrow 0}\sup_{x\in \mathbb{R}^d}
\int_{B_r(x)} \frac{|b(y)|}{|x-y|^{d+1-\alpha}} dy=0 \quad (\text{in many results it suffices to require $\leq \delta$ for small $\delta$}),
$$
proved the existence of a weak solution and, importantly, its uniqueness in law, and constructed a Feller semigroup, strengthening \cite{P,PP}; \cite{KS} also consider measure-valued Kato class drifts.
We emphasize that, for every $\varepsilon>0$, there exist Kato class drifts $b$ such that $|b|^{\frac{1}{\alpha-1}} \not \in L_{\loc}^{1+\varepsilon}(\mathbb R^d)$, i.e.\,situated far from \eqref{PP_cond}. 
The Kato class condition can be restated in an equivalent form: 
\begin{equation}
\label{dual_form}
\|(\lambda+(-\Delta)^{\frac{\alpha}{2}})^{-\frac{\alpha-1}{\alpha}}|b|\|_\infty \leq \delta \quad \bigg(\text{ or, by duality, } \quad \||b|(\lambda+(-\Delta)^{\frac{\alpha}{2}})^{-\frac{\alpha-1}{\alpha}}\|_{L^1(\mathbb R^d) \rightarrow L^1(\mathbb R^d)} \leq \delta\bigg),
\end{equation}
see \cite{Si} for details.
Kato class drifts arise in physical applications that are different from the ones that motivate the present work, such as diffusions on fractals; in particular, they do not introduce local blow-ups.
The Kato class is essentially the largest class of general drifts for which one has the standard two-sided heat kernel bounds\footnote{Given $a(z)$, $b(z) \geq 0$, we write $a(z) \approx b(z)$ if $c^{-1}b(z)\leq a(z) \leq cb(z)$ for some constant $c>1$ and all admissible $z$.
} of $\Lambda=(-\Delta)^{\frac{\alpha}{2}} + b \cdot \nabla$:
\begin{equation}
\label{heat_ker}
e^{-t\Lambda}(x,y) \approx e^{-t(-\Delta)^{\frac{\alpha}{2}}}(x-y),
\end{equation}
see Bogdan-Jakubowski \cite{BJ}; these two-sided bounds are used in the proofs in \cite{KS, CW}. In our setting,  \eqref{heat_ker} is no longer valid, cf.\,\eqref{nash}. Still, we use in our proof of Theorem \ref{thm1} some of the probabilistic arguments of \cite{CW, P,PP}.

In \cite{KM}, the authors considered SDE \eqref{sde1} with weakly form-bounded drift $b=b(x)$:
\begin{equation}
\label{wfb}
\||b|^{\frac{1}{2}}(\lambda+(-\Delta)^{\frac{\alpha}{2}})^{-\frac{1}{2}\frac{\alpha-1}{\alpha}}\|_{L^2(\mathbb R^d) \rightarrow L^2(\mathbb R^d)} \leq \sqrt{\delta}.
\end{equation}
 This class contains the Kato class, the weak Lebesgue class $|b| \in L^{\frac{d}{\alpha-1},\infty}(\mathbb R^d)$ and, more generally, $|b|^{\frac{1}{\alpha-1}}$ having sufficiently small elliptic Morrey norm
\begin{equation}
\label{morrey_e}
\||b|^{\frac{1}{\alpha-1}}\|_{M_{1+\epsilon}}:=\sup_{x \in \mathbb R^d, r>0}r \biggl(\frac{1}{r^d}\int_{B_r(x)}|b(y)|^{\frac{1+\epsilon}{\alpha-1}} dy\biggr)^{\frac{1}{1+\epsilon}},
\end{equation}
plus a bounded part of the drift. This class thus reaches drifts introducing local blow-up effects and destroying the standard heat kernel bounds \eqref{heat_ker}. Namely, \cite{KSS} extended some aspects of Nash's method  to handle $\Lambda=(-\Delta)^{\frac{\alpha}{2}} + b \cdot \nabla$ with a critical attracting (weakly form-bounded) drift $b(x)= \kappa x|x|^{-\alpha}$, $\kappa \in ]0,\kappa_\star[$, and established the following two-sided heat kernel bounds:
\begin{align}
\label{nash}
e^{-t\Lambda}(x,y) \approx e^{-t(-\Delta)^{\frac{\alpha}{2}}}(x-y)( t^{-\frac{1}{\alpha}}|y| \wedge 1)^{-d+\gamma}
\end{align}
for all $t>0, x, y \in \mathbb R^d, y \neq 0$. The exponent $\gamma \in ]\alpha,d]$ is uniquely determined by 
$$\Lambda^*|\cdot|^{-d+\gamma}=0 \quad \text{($\Lambda^*$ is the formal adjoint of $\Lambda$)}.$$
It is seen that already for this drift the heat kernel of $\Lambda$ is singular at $y=0$. For more general drifts \eqref{b_ser} such explicit bounds are out of reach since the constants already in a bound of type \eqref{nash} would depend on the distance between the singularities of the drift.

Concerning non-Morrey weakly form-bounded drifts, we refer to Chang-Wilson-Wolff \cite{CWW} who improve the Morrey condition by some logarithmic-type factors (strictly speaking, here we would have to speak of a closely related class of form-bounded drifts for $\alpha=2$ (Table \ref{tab1}), but it seems like the appropriate extension of \cite{CWW} is possible with a few additional efforts). At the same time, the $L^2$ setting of \eqref{wfb}  makes it convenient to e.g.\,apply Fourier transform to extend this class to other operators arising in some particle systems, see Theorem \ref{thm3}. 

The approach of \cite{KM} is, however, limited  to the elliptic setting and therefore to time-homogeneous drifts $b=b(x)$. In the present paper we pursue an alternative to the argument in \cite{KM}, strengthen rather substantially the results  in \cite{KM} by establishing weak uniqueness under reasonable occupation-time estimates, develop detailed parabolic Sobolev regularity theory of forward and backward equations that, in particular, yields a priori regularity for McKean-Vlasov equation. We also simplify substantially the localization argument in \cite{KM} (Section \ref{loc_sect}).

4.~There is a rich literature on the  stable-driven SDE \eqref{sde1} with distributional drifts, in particular, divergence-free drifts in the Bessel or Besov scale. We refer to Athreya-Butkovsky-Mytnik \cite{ABM}, Chaudru de Raynal-Menozzi \cite{CM}, Fitoussi \cite{F} and Kremp-Perkowski \cite{KrP}. These results are somewhat orthogonal to what we do in the present paper, i.e.\,handling attraction phenomena.  One interesting question that appears in relation to distributional drifts is how to define the product of two distributions $b \cdot \nabla u$ in the Kolmogorov PDE or how to interpret $\int_0^t b(s,X_s) ds$ in the SDE, see cited papers.

It is also interesting to note that one of the entry points to our analysis of the Kolmogorov PDE for \eqref{sde1}, i.e.\,the fractional Hardy inequality 
\begin{equation}
\label{hardy_frac}
H_{d,\alpha}\int_{\mathbb R^d} |x|^{-\alpha}\varphi^2 dx \leq \int_{\mathbb R^d}|(-\Delta)^{\frac{\alpha}{4}} \varphi|^2 dx \quad  \text{$\forall\,\varphi \in C_c^\infty(\mathbb R^d)$},
\end{equation}
 can be obtained using the machinery of Besov spaces,  see \cite{BCG}; moreover, they obtained a refined Hardy inequality having the additional feature of invariance under oscillations. However, the other Besov class distributions that satisfy the Hardy inequality obtained in this way must have sufficiently regular distributional derivatives, while in the present paper we are interested in conditions that involve only the drift itself (e.g.\,to be free to modify the interaction kernels). This observation is not a rigorous impossibility result for the SDE \eqref{sde1} since it is conditional on the use of our approach. That being said, to our knowledge, at the moment there are no well-posedness results for SDEs with Besov drifts that deal with the part of the Besov scale that allows one to handle the Hardy inequality or, more generally, deal with strong attracting singularities of the drift.

5.~Strong/pathwise well-posedness for stable-driven SDEs with irregular (but more regular than \eqref{hyp1} or \eqref{b_s_cond}) drifts has been studied by Priola \cite{Pr}, who treated bounded H\"older drifts, and by X. Zhang \cite{Z}, who treated time-dependent Sobolev drifts. Related stable-like and supercritical results for H\"older drifts were obtained by Chen-Song-Zhang \cite{CSZ}, Chen-Zhang-Zhao \cite{CZZ}, and Knopova-Kulik \cite{KnK}, see also the recent result of Brze\'{z}niak-Priola-Zhai-Zhu \cite{BPZZ} regarding well-posedness of the stochastic transport equation with L\'{e}vy noise.

6.~Surveying weak solutions to SDEs in the Brownian case $\alpha=2$ would require another paper, so we  mention only a few results dealing with general singular drifts that can introduce local blow-ups: by Sem\"enov and one of us \cite{D0,KiS}, R\"ockner-Zhao \cite{RZ}, Krylov's series of papers where he also handles VMO diffusion coefficients, including \cite{Kr3,Kr4,Kr5}, and \cite{D2}. The last paper is a Brownian  counterpart of the present paper; it was a continuation of \cite{D0, KiS}. However, the probabilistic arguments in \cite{D2} are quite different from the present work. Also, a number of results below are new even in the case $\alpha=2$, such as Theorem \ref{thm2}.

\subsection{Proofs}The proof of Theorem \ref{thm1} (SDE theory) has two components:
\begin{enumerate}
\item[--] The regularity theory of the Kolmogorov backward PDE. This is based on the following operator norm estimate (an equivalent form of \eqref{hyp1}) 
\begin{equation}
\label{main_est}
\|b^{\frac{1}{p}} \cdot  \nabla (\lambda + \partial_t+(-\Delta)^{\frac{\alpha}{2}})^{-1}|b|^{\frac{1}{p'}} \|_{L^p(\mathbb R^{1+d}) \rightarrow L^p(\mathbb R^{1+d})} <c_p \delta<1,
\end{equation} 
where
$b^{\frac{1}{p}}:=b|b|^{-1+\frac{1}{p}}$. It allows us to construct the Duhamel series for the inhomogeneous equation
$$
(\lambda+\partial_t+(-\Delta)^{\frac{\alpha}{2}} + b \cdot \nabla)u=f,
$$
where \eqref{main_est} is the estimate on the denominator of the geometric series. 

To extract useful information from the Duhamel series, i.e.\,to handle irregular $f$ and to apply the parabolic Sobolev embedding to characterize the regularity of $u$, we will need two ``almost-halves'' of \eqref{main_est}:
\begin{align}
\label{half1}
\|b^{\frac{1}{p}} \cdot \nabla(\lambda+\partial_t  + & (- \Delta)^{\frac{\alpha}{2}} )^{-\frac{1}{\alpha}+(-1+\frac{1}{\alpha})\frac{1}{p}-\varepsilon_0}\|_{L^p(\mathbb{R}^{1+d}) \rightarrow L^p(\mathbb{R}^{1+d})} <\infty
\end{align}
\begin{align}
\label{half2}
\|(\lambda+\partial_t  + & (- \Delta)^{\frac{\alpha}{2}} )^{(-1+\frac{1}{\alpha})\frac{1}{p'}-\varepsilon_0}|b|^{\frac{1}{p'}}\|_{L^p(\mathbb{R}^{1+d}) \rightarrow L^p(\mathbb{R}^{1+d})} <\infty,
\end{align}
where ``almost'' refers to $\varepsilon_0>0$; if $\varepsilon_0=0$, then the composition of these two estimates yields \eqref{main_est}.
 The Sobolev embedding, when $p>d+1$, provides us with the Feller propagator for SDE \eqref{sde1}. It determines a family of probability measures $\mathbb P_x$ on the canonical space of c\`{a}dl\`{a}g paths parametrized by the initial point $x \in \mathbb R^d$ -- these are the candidates for the weak solutions to  SDE \eqref{sde1}.

\medskip

\item[--] Armed with these PDE regularity results and the probability measures $\{\mathbb P_x\}_{x \in \mathbb R^d}$, we apply the argument of Chen-Wang \cite{CW}, Portenko \cite{P}, Podolynny-Portenko \cite{PP}, i.e.\,we inspect the characteristic function and apply Banach's contraction principle to show that
\begin{equation*}
\mathbb E_{0,x} \biggl[e^{i \xi \cdot (\omega_t-x+\int_0^t b(s,\omega_s)ds)} \biggr]=e^{-t|\xi|^\alpha}, \quad \forall\,x,\xi \in \mathbb R^d, \;t \geq 0,
\end{equation*}
which is the characteristic function of the isotropic $\alpha$-stable process.

\end{enumerate}

The proof of Theorem \ref{thm2} on the Kolmogorov forward (Fokker-Planck) equation is based on a dual variant of the PDE results used in the proof of Theorem \ref{thm1}. We emphasize that even if we take $\alpha=2$ and $p=2$, we do not end 
up with the standard triple of Hilbert spaces $W^{1,2} \rightarrow L^2 \rightarrow W^{-1,2}$; instead, we work in an asymmetric triple of Bessel potential spaces (cf.\,Theorem \ref{thm2}), which allows us to handle a more irregular right-hand side in the Fokker-Planck equation, and therefore more irregular initial data. By working in the asymmetric triple we lose the possibility to handle irregular diffusion coefficients, but this is not among our objectives in this work.

\medskip

Condition \eqref{main_est} goes back to the work of Sem\"enov on the regularity of eigenfunctions of Schr\"odinger operators, see \cite{BS,LS}, and one of us \cite{D0} on the realization of an elliptic Kolmogorov operator with singular weakly form-bounded drift $b=b(x)$ as a Feller generator. 
One important difference between these works is that a Schr\"odinger operator -- having a singular form-bounded potential -- is well defined in $L^2$ via the representation theorem \cite[Ch.\,VI, \,Sect.\,2]{Ka}. One thus ventures into $L^p$ to extract additional information about its domain. In contrast, in \cite{D0} none of the classical methods of the perturbation theory would apply in $L^p$, including $p=2$, so some rather deep  semigroup theory, including Hille's theory of pseudoresolvents and Trotter's approximation theorem,  had to be used. By working in the parabolic setting directly, we will use estimate \eqref{main_est} on \textit{Sem\"enov's denominator} to postulate the Duhamel series for the Kolmogorov equations, sacrificing \textit{a few} spatial singularities compared to the weakly form-bounded drifts, but avoiding all the difficulties arising in the elliptic setting and also gaining critical singularities of the drift  in time.

At the time of writing, we are aware of three ways to verify \eqref{main_est}  or  its elliptic counterpart, see Table \ref{tab1}.

\begin{table}[h]
\centering
\caption{Ways to verify estimates \eqref{main_est}-\eqref{half2}}
\label{tab1}
\begin{tabular}{|c|c|c|}
\hline
Condition on drift $b$ & $\alpha$ & Features \\
\hline
\makecell{\\ 1.~Form-bounded $b=b(x)$ \\ $\||b|(\lambda-\Delta)^{-\frac{1}{2}}\|_{L^2(\mathbb R^d) \rightarrow L^2(\mathbb R^d)} \leq \delta$ \\ \text{ }} & $\alpha=2$ & \makecell{Allows for explicit control of \\ admissible $\delta$ for SDE \cite{D4} \\ Has a parabolic variant}  \\
\hline
\makecell{ \\ 2.~Weakly form-bounded $b=b(x)$ \\ [4mm] } & $1<\alpha \leq 2$ & \makecell{Largest class of spatial singularities \\ $\delta$ is sufficiently small \cite{KM,D0}} \\
\hline
\makecell{3.~Morrey class condition \eqref{b_s_cond} on $b=b(t,x)$} & $1<\alpha \leq 2$ & \makecell{Has parabolic variant \\  Morrey norm ($\approx \delta$) is small \\ Spatial singularities in-between 1 \& 2\\ Can be split into ``halves''} \\
\hline
\end{tabular}
\end{table}

In this paper we pursue the third condition, i.e.\,the Morrey class \eqref{b_s_cond}. In fact, for the Morrey class it is possible to obtain \eqref{main_est}, using the boundedness of the parabolic Riesz transform on $L^p(\mathbb R^{1+d})$, from the following two ``halves'': given $0<\varepsilon ~(<d+\alpha-1)$, for every $1<p<\infty$,
\begin{align}
\label{adams_est1_}
\||b|^{\frac{1}{p}}(\pm \partial_t  + & (- \Delta)^{\frac{\alpha}{2}} )^{(-1+\frac{1}{\alpha})\frac{1}{p}}\|_{L^p(\mathbb{R}^{1+d}) \rightarrow L^p(\mathbb{R}^{1+d})} \leq c_{d,\alpha,p,\varepsilon}\||b|^{\frac{1}{\alpha-1}}\|^{\frac{\alpha-1}{p}}_{E_{1+\epsilon}},
\end{align}
and, by duality,
\begin{align}
\label{adams_est2_}
\|(\pm \partial_t  + & (- \Delta)^{\frac{\alpha}{2}} )^{(-1+\frac{1}{\alpha})\frac{1}{p'}}|b|^{\frac{1}{p'}}\|_{L^p(\mathbb{R}^{1+d}) \rightarrow L^p(\mathbb{R}^{1+d})} \leq c_{d,\alpha,p',\varepsilon}\||b|^{\frac{1}{\alpha-1}}\|^{\frac{\alpha-1}{p'}}_{E_{1+\epsilon}}
\end{align}
(so one actually can take $\varepsilon_0=0$ in \eqref{half1}, \eqref{half2} -- but this is not possible for weakly form-bounded drifts, for which \eqref{half1}, \eqref{half2} require $\varepsilon_0>0$). See Section \ref{adams_est_sect} for a precise statement and the history of such estimates.

Let us also add that for the class of form-bounded drifts one can manage the gradient instead of getting rid of it via Riesz transform or even using pointwise heat kernel estimates, arriving at better conditions on $\delta$, but this requires some rather delicate and strong gradient bounds \cite{D4}; it is not clear at the moment how to obtain them for the other classes.

Our class of drifts is large enough to destroy $L^p$ estimates, for $p$ large, on the order $\alpha$ spatial derivatives of solutions of the Kolmogorov backward equation. Even in the case $\alpha=2$, this excludes the possibility to also incorporate discontinuous diffusion coefficients, such as VMO; the reason is the size of our class of singular drifts.

Theorems \ref{thm1}, \ref{thm2} allow us to describe quantitatively how the theory of SDE \eqref{sde1} and of the corresponding Kolmogorov equations changes if we change the drift $b$ to $cb$ for a constant $c>0$ or, in other words, adjust the Morrey norm (form-bound) of the drift. Obtaining such a quantitative description is, essentially, our ultimate goal in the present work.

\subsection{Notations}

\label{notations_sect}
\begin{enumerate}[label=(\alph*)]

\item Put
$$
\langle h \rangle :=\int_{\mathbb{R}^{1+d}} h dz, \quad \langle h,g\rangle:=\langle hg\rangle.
$$

\item 
Denote by $\mathcal B(X,Y)$ the space of bounded linear operators between Banach spaces $X \rightarrow Y$, endowed with the operator norm $\|\cdot\|_{X \rightarrow Y}$, and set $\mathcal B(X):=\mathcal B(X,X)$.

\item We write $T=s\mbox{-} X \mbox{-}\lim_n T_n$ for $T$, $T_n \in \mathcal B(X)$ if $Tf=\lim_n T_nf$ in $X$ for every $f \in X$. 

\item 
Put $$\|h\|_p \equiv \|h\|_{L^p(\mathbb{R}^{1+d})}:=\langle |h|^p\rangle^{\frac{1}{p}}.$$ 
Denote by $\|\cdot\|_{p \rightarrow q}:=\|\cdot\|_{L^p(\mathbb{R}^{1+d}) \rightarrow L^q(\mathbb{R}^{1+d})}$ the corresponding operator norm.

\item Let $$C_\infty(\mathbb R^{d}):=\{f \in C(\mathbb R^{d}) \mid \lim_{|x| \rightarrow \infty}f(x)=0\},$$ 
endowed with the $\sup$-norm. 

\item We denote by $C_b(\mathbb{R}^{d})$ the space of bounded continuous functions on $\mathbb{R}^{d}$, and by $\mathcal{S}(\mathbb{R}^{d})$ the Schwartz space of smooth functions rapidly decaying at infinity.

\item We will need the polynomial weight
\begin{equation}
\label{rho_def0}
\rho(x)=(1+\kappa|x|^2)^{-\frac{\theta}{2}}, \quad \theta>d, \quad \kappa>0.
\end{equation}
The derivatives of both $\rho$ and $\rho^{-1}$ are bounded from above by $\rho$ and $\rho^{-1}$, respectively, times a constant that can be made as small as needed by selecting $\kappa$ sufficiently small.

For each $x\in\mathbb R^d$, we define the translated weight
$$\rho_x(y):=\rho(y-x)=(1+\kappa|y-x|^2)^{-\frac{\theta}{2}}.$$
The same derivative bounds hold for $\rho_x$ and $\rho_x^{-1}$,
with constants independent of $x$. We also write $\rho=\rho_0$.

\item
Define the weighted Lebesgue space $$L^p_{\rho_x}(\mathbb R^{1+d}):=L^p(\mathbb R^{1+d},\rho_x(y)dtdy).$$
We also write $L^p_{\rho_x}([s,t] \times \mathbb R^d):=L^p([s,t] \times \mathbb R^d,\rho_x(y)\,dr\,dy)$.

The dual space of $L^p_{\rho_x}(\mathbb R^{1+d})$ is
$L^{p'}_{\rho_x^{1-p'}}(\mathbb R^{1+d})$, where
$\frac1p+\frac1{p'}=1$.

\item Recall that the integral kernel $q_\gamma(t,x)$, $0<\gamma \leq 2$, of the operators ($\lambda \geq 0$)
\begin{align*}
(\lambda+\partial_t+(-\Delta)^{\frac{\alpha}{2}})^{-\frac{\gamma}{2}}f(t,x) &=\int_{\mathbb{R}^{1+d}} e^{-\lambda(t-s)}q_\gamma(t-s,x-y)f(s,y)dyds \\[2mm]
(\lambda-\partial_t+(-\Delta)^{\frac{\alpha}{2}})^{-\frac{\gamma}{2}}f(t,x)&=\int_{\mathbb{R}^{1+d}} e^{-\lambda(s-t)}q_\gamma(s-t,x-y)f(s,y)dyds
\end{align*}
satisfies the pointwise bounds
\begin{align}
\label{ul_bounds}
c\, p_\gamma(t,x) \leq q_\gamma(t,x)  \leq C\, p_\gamma(t,x),
\end{align}
for
$$
p_\gamma(t,x):=\mathbf{1}_{\{t>0\}} t^{\frac{\gamma}{2}-1}\bigg(t^{-\frac{d}{\alpha}}\wedge\frac{t}{|x|^{d+\alpha}}\bigg) = \mathbf{1}_{\{t>0\}} \left\{
\begin{array}{ll}
t^{\frac{\gamma}{2}}|x|^{-d-\alpha}, & |x| \geq t^{\frac{1}{\alpha}}, \\
t^{\frac{\gamma}{2}-\frac{d+\alpha}{\alpha}}, & |x| < t^{\frac{1}{\alpha}},
\end{array}
\right.
$$
with constants $0<c$, $C<\infty$ depending only on $d$, $\alpha$, and $\gamma$.
Furthermore, 
\begin{equation}
\label{grad_bd}
|\nabla_x q_\gamma(t,x)| \leq C_1 p_{\gamma-\frac{2}{\alpha}}(t,x).
\end{equation}

\item L\'evy representation for the fractional Laplacian: for $0<\alpha<2$,
\begin{equation}
\label{levy}
-(-\Delta)^{\frac{\alpha}{2}} f(x)
=
c_{d,\alpha}\int_{\mathbb R^d}
\Big(f(x+z)-f(x)-\mathbf 1_{\{|z|\le1\}}\nabla f(x)\cdot z\Big)\frac{dz}{|z|^{d+\alpha}}.
\end{equation}

\item Let  $\mathbb{W}_\alpha^{\gamma,p}(\mathbb{R}^{1+d})$ 
denote the fractional parabolic Bessel potential space:
$$
\mathbb{W}_\alpha^{\gamma,p}(\mathbb{R}^{1+d})
:=
(\lambda+\partial_t+(-\Delta)^{\frac{\alpha}{2}})^{-\frac{\gamma}{2}}
L^p (\mathbb{R}^{1+d})
$$
with norm
$$
\|g\|_{\mathbb{W}_\alpha^{\gamma,p}}
:=
\| (\lambda+\partial_t+(-\Delta)^{\frac{\alpha}{2}})^{\frac{\gamma}{2}} g\|_p.
$$

\item The parabolic Sobolev embedding theorem: for all
 $1 \leq p \leq q \leq \infty$ such that
\begin{equation}
\label{sob_emb0}
\frac{\gamma}{2}>\biggl(1+\frac{d}{\alpha} \biggr)\bigg(\frac{1}{p}-\frac{1}{q}\bigg),
\end{equation}
one has
\begin{equation}
\label{sob_emb}
\big(\lambda+\partial_t + (-\Delta)^{\frac{\alpha}{2}}\big)^{-\frac{\gamma}{2}} \in \mathcal B(L^p,L^q), \qquad \text{ so }\mathbb{W}_\alpha^{\gamma,p} \subset L^q(\mathbb{R}^{1+d}).
\end{equation}
See \cite{AS}.

We refer to the recent comprehensive treatment  of the theory of elliptic and parabolic Morrey spaces by Krylov \cite{Kr2}. See also Adams-Xiao \cite{AX}.

\item 
Let $I$ be $\mathbb R_+:=[0,\infty[$, a closed interval in $\mathbb R_+$, or an interval $[s,T[$ with $0\le s<T<\infty$. We denote by $D=D(I,\mathbb R^d)$  the space of c\`{a}dl\`{a}g paths $I \rightarrow \mathbb R^d$.
We endow $D(I,\mathbb R^d)$ with the Skorokhod $J_1$ topology
(see \cite[Chapter~3, Section~5]{EK}). Recall that
$\omega_n\rightarrow\omega$ in $D([s,T],\mathbb R^d)$ if and only if
there are increasing homeomorphisms $\tau_n:[s,T]\rightarrow[s,T]$ such that
\[
 \sup_{r\in[s,T]}|\tau_n(r)-r|
 +\sup_{r\in[s,T]}|\omega_n(\tau_n(r))-\omega(r)|\rightarrow0.
\]
Thus convergence is uniform after small changes of time, allowing
nearby jump times to be aligned. 

On $D([s,T[,\mathbb R^d)$,
convergence means the $J_1$ convergence on the restrictions to $[s,t]$
for every $s<t<T$ such that $\omega$ is continuous at $t$: $\omega(t-)=\omega(t)$. The same 
convention applies to $\mathbb R_+$.

\item 
A weak solution to the SDE \eqref{sde1} is a pair of c\`{a}dl\`{a}g processes $(X_t,Z_t)$ defined on a complete probability space  $(\Omega,\{\mathcal F_t\}_{t \geq s},\mathcal F,\mathbb P)$ with a right-continuous filtration, such that $X_0=x$ $\mathbb P$-a.s., $$\int_0^T |b(s,X_s)|ds<\infty \quad \text{$\mathbb P$-a.s.},$$ $Z_t$ is an isotropic $\alpha$-stable process and 
the identity \eqref{sde1} holds $\mathbb P$-a.s. 

\medskip

In Theorem \ref{thm1}, we construct a weak solution to the SDE \eqref{sde1} by taking  $\Omega:=D([s,T],\mathbb R^d)$ (for $0 \leq s<T<\infty$ fixed), i.e.\,on the canonical space of c\`{a}dl\`{a}g paths 
$$\mathbf{D}_{s,T}:=(D([s,T],\mathbb R^d),\{\mathcal F_t\}_{t \geq s},\mathcal F),$$  where $\mathcal F_t$ is the canonical filtration (complete, right-continuous). 
We will define $X_t(\omega):=\omega(t)$ to be the coordinate process and will  define $Z_t$ via identity \eqref{sde1}. Our goal thus will be to find a probability measure on $\mathbf{D}_{s,T}$ that makes $Z_t$ an isotropic $\alpha$-stable process.

\medskip

\item  The It\^{o} formula for $X_t$, a weak solution of SDE \eqref{sde1}, implies, in particular, that for every bounded $f \in C_t^1C^2_x(\mathbb R \times \mathbb R^d)$ having bounded derivatives
\begin{equation}
\label{ito}
f(t,X_t)=f(0,x)+\int_0^t \bigl(\partial_s f  - b \cdot \nabla f - (-\Delta)^{\frac{\alpha}{2}}f\bigr)(s,X_s)ds + M_t,
\end{equation}
where $M_t$ is the compensated jump martingale part.

\end{enumerate}

\subsection*{Acknowledgements} We express our gratitude to N. V. Krylov for many useful discussions, in particular concerning the localization argument in Proposition~\ref{prop1_weight}; we use some of his arguments in the proof. 
We are also deeply grateful to M.-A. Orsoni for a number of insightful comments. 

C. Yue would like to thank A. Girouard for his attention to this work and support.

Finally, we thank the anonymous referees for careful reading of the manuscript and useful comments.

\bigskip

\section{Main results}

Our condition on the drift is as follows\footnote{In complete generality, we need to control operator norms in \eqref{adams_est_weight}, \eqref{adams_est2_weight}  uniformly in the weight centre, with a common choice of $\kappa,\lambda$. This also must hold for the chosen regularization $\{b_n\}$.}:

\begin{enumerate}
\item[($\mathbf{H}$)] Let $b:\mathbb R^{1+d} \rightarrow \mathbb R^d$ be a Borel measurable vector field of the form
\begin{equation}
\label{b_decomp}
b=b_{\mathfrak{s}} + b_{\mathfrak{b}},
\end{equation}
where $|b_{\mathfrak{b}}| \in L^\infty(\mathbb R^{1+d})$ (the bounded part) and the singular part is in the Morrey class, i.e.
\begin{equation*}
\| |b_{\mathfrak s} |^{\frac{1}{\alpha-1}} \|_{E_{1+\epsilon}} <\infty
\end{equation*}
where $\epsilon>0$ is fixed small. As discussed in the introduction, below we will also need smallness of this Morrey norm, otherwise we run into a counterexample to weak existence for the SDE.
\end{enumerate}

Fix some $T > 0$. We will use weight $\rho_x(y)=(1+\kappa|y-x|^2)^{-\frac{\theta}{2}}.$

\begin{theorem}[SDE theory]
\label{thm1}
Let $d \geq 2$. Assume that the Borel measurable drift $b:\mathbb R^{1+d} \rightarrow \mathbb R^d$ satisfies \rm($\mathbf{H}$). Let $p>d+1$. Then there exists constant $c_{d,\alpha,p,\varepsilon}>0$ such that 
if 
$$
\| |b_{\mathfrak s} |^{\frac{1}{\alpha-1}} \|_{E_{1+\epsilon}} \leq c_{d,\alpha,p,\varepsilon},
$$
then\footnote{The constant $c_{d,\alpha, p, \varepsilon}>0$ depends on the constants in some fundamental inequalities of harmonic analysis.} the following are true:
 
 \smallskip

\begin{enumerate}[label=(\roman*)] 

\item {\rm (Weak existence)} For each $x \in \mathbb R^d$ and $0 \leq s < T$, there exists a  probability measure $\mathbb P_{s,x}$ on $\mathbf{D}_{s,T}$ such that
\begin{enumerate}[label=(\alph*)]

\item 
$$
\mathbb P_{s,x}[\omega_s=x]=1,
$$

\item
$$
\mathbb E_{s,x}\int_s^t |b(r,\omega_r)|dr<\infty, \quad t \geq s,
$$

\item 
$$
\omega_t = x - \int_s^t b(r, \omega_r) dr + Z_t-Z_s,
$$
for an $\alpha$-stable process $Z_t$ defined on the completion of  $\mathbf{D}_{s,T}$ with respect to $\mathbb P_{s,x}$. 
\end{enumerate}

That is, the pair (coordinate process $X_t$, $Z_t$ defined via (c)) is a weak solution to the SDE \eqref{sde1}.

Also, assume that in the definition of weight $\rho_x$ we fix $d<\theta<d+\alpha$ and $\kappa$ 
sufficiently small  (that is, depending only on $d$, $\alpha$ and the norms of $b_{\mathfrak b}$ and $b_\mathfrak{s}$). Let $\mathsf f:\mathbb R^{1+d}\rightarrow\mathbb R^d$
be a test vector field satisfying ($\mathbf{H}$). There exists constant $K$, independent of $s,x,h$, such that
\begin{equation}
\label{krylov_bd}
 \mathbb E_{s,x}\int_s^T|\mathsf f(r,\omega_r)h(r,\omega_r)|\,dr
 \le K\||\mathsf f|^{1/p}h\|_{L^p_{\rho_x}([s,T] \times \mathbb R^d)},
 \qquad 0\le s<T,\quad x\in\mathbb R^d,
\end{equation}
for all $h\in C_b(\mathbb R^{1+d})$. In particular,
\[
 \sup_{x\in\mathbb R^d}\mathbb E_{s,x}
       \int_s^T|\mathsf f(r,\omega_r)h(r,\omega_r)|\,dr
 \le K\sup_{x\in\mathbb R^d}
       \||\mathsf f|^{1/p}h\|_{L^p_{\rho_x}([s,T] \times \mathbb R^d)}.
\]
Taking $\mathsf f=b$ controls the time spent by $X_t$ near the singular set of the drift.

\medskip

\item {\rm (Feller propagator)}
The operators
\begin{equation}
\label{P}
P^{t,r}f(x):=\mathbb E_{t,x}[f(X_r)], \quad 0 \leq t \leq r, \quad x \in \mathbb R^d
\end{equation}
constitute a backward Feller propagator on $C_\infty$.

\medskip

\item  {\rm (Conditional weak uniqueness)} Let $x \in \mathbb R^d$. Any probability measure $\mathbb Q_{s,x}$ on $\mathbf{D}_{s,T}$ that satisfies (a)-(c) in (\textit{i}) and
\begin{equation*}
\mathbb E_{\mathbb Q_{s,x}}\int_s^T |h(r,\omega_r)|dr \leq K\|h\|_{L^p([s,T] \times \mathbb R^d)},
\end{equation*}
and
\begin{equation*}
\mathbb E_{\mathbb Q_{s,x}}\int_s^T |b(r,\omega_r)h(r,\omega_r)|dr \leq K\||b|^{\frac{1}{p}}h\|_{L^p([s,T] \times \mathbb R^d)}
\end{equation*}
for all $h \in (C_b \cap L^p)(\mathbb R^{1+d})$ also satisfies $$\mathbb Q_{s,x}=\mathbb P_{s,x}$$
(it is seen that the weak solution constructed in {\rm (\textit{i})} satisfies the previous two inequalities).

\end{enumerate}

\end{theorem}

\begin{remark}[Convergence on the path space]
\label{path_law_convergence_rem}
Let $\mathbb P^n_{s,x}$ denote the laws of the SDEs with
regularized drifts $b_n$ defined in \eqref{b_n}. The strong convergence of the
conservative Feller propagators, locally uniform in the two time
parameters, established in \ref{feller_subsect} and used in the
proof of Theorem~\ref{thm1}(\textit{ii}), implies
\[
 \mathbb P^n_{s,x}\rightarrow\mathbb P_{s,x}
 \quad\text{ weakly on }D([s,T],\mathbb R^d)
\]
for every $0\le s<T$ and $x\in\mathbb R^d$.
This follows from the analogous result for Feller semigroups, see
Ethier-Kurtz \cite[Chapter~4, Theorem~2.5]{EK}, applied to the
space-time semigroups obtained from the Feller propagators by B\"ottcher's construction
\cite[Theorem~3.2]{Bo}. This convergence, however, is not used in the proof
of Theorem~\ref{thm1}.
\end{remark}

	We consider the uniqueness class of (\textit{iii}) -- rather than the standard uniqueness class in terms of the Krylov bound -- since this allows us to capture the following effect: the uniqueness class becomes larger as $b$ becomes more regular. Namely, a nice $b$ (that is, one having small Morrey norm) allows us to take a large  $p$. So, in principle, in the limit, $p \rightarrow \infty$, e.g.\,bounded $|b|$, we recover the  uniqueness in law.

Another form of conditional weak uniqueness of $X_t$ is established in the next theorem.

\medskip

We fix the following regularization of $b$:
\begin{equation}
\label{b_n}
b_{n}:=b_{n,\mathfrak{b}}+ b_{n,\mathfrak{s}}, \quad n=1,2,\dots,
\end{equation}
where
\begin{equation}
\label{b_n2}
b_{n,\mathfrak{b}}=\gamma_{\varepsilon_n} \star \mathbf{1}_n b_{\mathfrak{b}}, \quad b_{n,\mathfrak{s}}=\gamma_{\varepsilon_n} \star \mathbf{1}_n b_{\mathfrak{s}},
\end{equation}
and 
$$\mathbf{1}_n:=\mathbf{1}_{\{(t,x) \in \mathbb{R}^{1+d} \mid |t| \leq n, |x| \leq n, |b(t,x)| \leq n\}},
$$ 
 $\{\gamma_\epsilon\}_{\epsilon>0}$ is the standard Friedrichs mollifier on $\mathbb{R}^{1+d}$ with compact support. 
We can  select
 $\epsilon_n \downarrow 0$ sufficiently rapidly to have
$
\sup_n\| |b_{n,\mathfrak{s}} |^{\frac{1}{\alpha-1}} \|_{E_{1+\epsilon}}$ as close to $\| |b_{\mathfrak{s}} |^{\frac{1}{\alpha-1}} \|_{E_{1+\epsilon}}
$ as needed, say,
$$
\| |b_{n,\mathfrak{s}} |^{\frac{1}{\alpha-1}} \|_{E_{1+\epsilon}} \leq (1+2^{-n})\| |b_{\mathfrak{s}} |^{\frac{1}{\alpha-1}} \|_{E_{1+\epsilon}}.
$$
Alternatively, we can multiply $b_n$ defined as above by 
 $0<c_n \uparrow 1$ (sufficiently slowly) to have the inequality $\| |b_{n,\mathfrak{s}} |^{\frac{1}{\alpha-1}} \|_{E_{1+\epsilon}} \leq \| |b_{\mathfrak{s}} |^{\frac{1}{\alpha-1}} \|_{E_{1+\epsilon}}$ for all $n=1,2,\dots$.

\begin{remark} By selecting $\varepsilon_n \downarrow 0$ sufficiently rapidly, we  can make $b_n$ as close to $\mathbf{1}_n b$ as needed in any $L^s(\mathbb{R}^{1+d})$, $1 \leq s<\infty$, so that we can treat $\{|b_n|\}_{n \geq 1}$ as essentially a monotone approximation of $|b|$, in the same way as $\{|\mathbf{1}_n b|\}_{n \geq 1}$ is.
\end{remark}

Let us  add that the role of $\lambda>0$ below is to handle the bounded part of the drift, and also to let us work with nicer Sobolev embeddings. We also recall that $\lambda$ enters the definition of the parabolic Bessel potential spaces $\mathbb{W}_\alpha^{\gamma,p}$.

\begin{theorem}[PDE theory]
	\label{thm2} Assume that the Borel measurable drift $b:\mathbb R^{1+d} \rightarrow \mathbb R^d$ satisfies  {\rm($\mathbf{H}$)}. For every $1<p<\infty$, there exist constants $c_{d,\alpha,p,\epsilon}>0$ and $\lambda_{d,\alpha,p,\epsilon} \geq 0$ such that if 
		$$\||b_{\mathfrak s}|^{\frac{1}{\alpha-1}}\|_{E_{1+\epsilon}} < c_{d,\alpha,p,\epsilon},$$ 
		then
	the following are true:

	\begin{enumerate}[label=(\roman*)]

		\item
		Given a right-hand side 
		\begin{equation}
		\label{f_cond}
		f \in \mathbb{W}_\alpha^{-\frac{2}{\alpha}+(-1+\frac{1}{\alpha})\frac{2}{p'},p}(\mathbb{R}^{1+d})
		\end{equation}
		and a sequence of functions
		$\{f_n\} \subset L^\infty(\mathbb{R}^{1+d})$ such that $$
		f_n \rightarrow f \quad \text{ in } \mathbb{W}_\alpha^{-\frac{2}{\alpha}+(-1+\frac{1}{\alpha})\frac{2}{p'},p},
		$$
		the solutions $\eta_n$ of the approximating inhomogeneous Kolmogorov-Fokker-Planck equations
		\begin{equation*}
			\lambda \eta_n + \partial_t 
			\eta_n
			 +(- \Delta)^{\frac{\alpha}{2}} \eta_n - {\rm div\,} (b_n   \eta_n)=f_n, \qquad \lambda > \lambda_{d,\alpha,p,\epsilon},
		\end{equation*}
		where $\{b_n\}$ is defined in \eqref{b_n},
		converge as $n \rightarrow \infty$ in $\mathbb{W}_\alpha^{\frac{\alpha-1}{\alpha}\frac{2}{p},p}$ to the function (formal Duhamel series) 
		\begin{equation}
			\label{FPu_repr31}
			\eta=(\lambda+\partial_t  +(- \Delta)^{\frac{\alpha}{2}})^{(-1+\frac{1}{\alpha})\frac{1}{p}}(1-G_pH_p)^{-1} (\lambda+\partial_t  +(- \Delta)^{\frac{\alpha}{2}})^{-\frac{1}{\alpha}+(-1+\frac{1}{\alpha})\frac{1}{p'}} f
		\end{equation}
		where
		\begin{equation*}
			G_p:=\nabla \cdot (\lambda+\partial_t  +(- \Delta)^{\frac{\alpha}{2}} )^{-\frac{1}{\alpha}+(-1+\frac{1}{\alpha})\frac{1}{p^\prime}}  b^{\frac{1}{p^\prime}}, 
		\end{equation*}
		\begin{equation*}
			H_p:=|b|^{\frac{1}{p}}(\lambda+\partial_t  +(- \Delta)^{\frac{\alpha}{2}})^{(-1+\frac{1}{\alpha})\frac{1}{p}},
		\end{equation*}
		are bounded operators on $L^p(\mathbb{R}^{1+d})$, and the $L^p \rightarrow L^p$ operator norm of the denominator of the geometric series in \eqref{FPu_repr31} is dominated by
		\begin{equation*}
			\|G_p\|_{p \rightarrow p}\|H_p\|_{p \rightarrow p} <1,
		\end{equation*}
		so the series converges in $L^p(\mathbb{R}^{1+d})$.
		It follows that $\eta \in \mathbb{W}_\alpha^{(1-\frac{1}{\alpha})\frac{2}{p},p}.$
		
	Another representation for $\eta$:
		\begin{align}\label{FPu_repr32}
			& \eta =  (\lambda+\partial_t  +(- \Delta)^{\frac{\alpha}{2}})^{-1}f \notag \\
			&\quad + (\lambda+\partial_t  +(- \Delta)^{\frac{\alpha}{2}})^{(-1+\frac{1}{\alpha})\frac{1}{p}}G_p (1-H_pG_p)^{-1}H_p (\lambda+\partial_t  +(- \Delta)^{\frac{\alpha}{2}})^{-\frac{1}{\alpha}+(-1+\frac{1}{\alpha})\frac{1}{p'}} f. 
		\end{align}

		\smallskip

		\item Let $f$ satisfy \eqref{f_cond}. Then the function $\eta$ given by \eqref{FPu_repr31} is the unique $L^p$ weak solution (see the definition after the theorem) to the inhomogeneous Kolmogorov-Fokker-Planck equation 
	\begin{align}\label{FP3}
		\lambda \eta + \partial_t \eta +(-\Delta)^{\frac{\alpha}{2}}\eta- {\rm div\,}( b  \eta)=f.
	\end{align}

		\item In particular, whenever $1<p<\frac{d+1}{d}$ in {\rm(\textit{ii})}, the following Cauchy problem for the Kolmogorov-Fokker-Planck  equation
		\begin{equation}\label{FP_initial}
		\left\{
		\begin{array}{l}
					\lambda  \eta + \partial_t \eta +(-\Delta)^{\frac{\alpha}{2}}\eta- {\rm div\,}( b  \eta)=0, \\
					\eta|_{t=0}=\delta_x,
					\end{array}
					\right.
		\end{equation} 
		has a unique weak solution $\eta(t,\cdot)$. This weak solution corresponds to selecting the right-hand side $f:=\delta_{t=0}\delta_x$ in (\textit{ii}) -- it satisfies \eqref{f_cond}.
	\end{enumerate}
\end{theorem}

The following definition of the weak solution to PDE \eqref{FP3} is obtained from \eqref{FP3} by formally multiplying it by a test function $\varphi$ and integrating by parts:

	\begin{definition*}
	\label{def_weak_sol}
			Let $1<p<\infty$. An $L^p$ weak solution to \eqref{FP3} is a Borel measurable function  $\eta:\mathbb R^{1+d} \rightarrow \mathbb R$ having regularity 
			\begin{equation}
			\label{sob_hyp}
			\eta \in \mathbb{W}_\alpha^{(1-\frac{1}{\alpha})\frac{2}{p},p}
\end{equation}
			such that the following identity holds:
	    \begin{align}
		&\langle
		(\lambda+\partial_t+(-\Delta)^{\frac{\alpha}{2}})^{(1-\frac{1}{\alpha})\frac{1}{p}} \eta,
		(\lambda-\partial_t+(-\Delta)^{\frac{\alpha}{2}})^{1-(1-\frac{1}{\alpha})\frac{1}{p}} \varphi
		\rangle \notag \\
		&\quad+
		\langle
		H_p	(\lambda+\partial_t+(-\Delta)^{\frac{\alpha}{2}})^{(1-\frac{1}{\alpha})\frac{1}{p}}  
		\eta, 
		-G^\ast_p
		(\lambda-\partial_t+(-\Delta)^{\frac{\alpha}{2}})^{1-(1-\frac{1}{\alpha})\frac{1}{p}}   \varphi
		\rangle=
		\langle
 f,
 \varphi
		\rangle \quad \forall\, \varphi \in \mathcal S(\mathbb R^{1+d}), \label{weak_def3_}
	\end{align}
	where the operators $H_p$, $G_p$ are from Theorem \ref{thm2}.

	\end{definition*}

	By dealing with operators $\pm \partial_t + (-\Delta)^{\frac{\alpha}{2}}$, we consider time as another variable: in the analysis of parabolic equations one encounters other arguably non-physical constructions involving the time variable, such as complex time when dealing with holomorphic semigroups. We also consider Morrey norm $\|\cdot\|_{E_{1+\epsilon}}$ where $t$ and $x$ play equal roles in  scaling. This, however, appears to be restrictive for further development of Theorems \ref{thm1}, \ref{thm2}, e.g.\,applications to McKean-Vlasov equations. It would be reasonable to consider more general anisotropic Morrey norms; such norms are discussed in \cite{Kr2}.

	\bigskip
	
\section{Some applications to particle systems}
\label{particle_sect}

Let us make a few comments, rather for illustration purposes, regarding applications of our results to the finite particle system \eqref{syst1} and to the McKean-Vlasov PDE behind \eqref{mv}. Regarding particle system \eqref{syst1}, Theorems \ref{thm1} and \ref{thm2} require a modification -- unless $\alpha=2$ -- namely, replacing the unperturbed operator by a quasi-isotropic $\alpha$-stable generator
$$
\sum_{i=1}^N (-\Delta)^{\frac{\alpha}{2}}_{x^i},
$$
acting on  $x=(x^1,\dots,x^N) \in \mathbb R^{dN}$. To this end, we can either use our proofs but employing a variant of Proposition \ref{prop1} for the heat kernel $$e^{-t\sum_{i=1}^N (-\Delta)^{\frac{\alpha}{2}}_{x^i}}(x-y)=\prod_{i=1}^N e^{-t (-\Delta)^{\frac{\alpha}{2}}_{x^i}}(x^i-y^i),$$ or follow the approach of \cite{KM}, i.e.\,the elliptic argument using estimates for weakly form-bounded drifts obtained via Stroock-Varopoulos inequalities. We will pursue the latter, taking advantage of its somewhat more abstract form. 

\begin{theorem}
\label{thm3}

Let the interaction kernel $K:\mathbb R^d \rightarrow \mathbb R^d$ be Borel measurable.
The following are true:

\begin{enumerate}

\item[{\rm (\textit{i})}] {\rm (Finite particle system)} Assume additionally that $K$ is weakly form-bounded with respect to the fractional Laplacian on $\mathbb R^d$, i.e. $|K| \in L^1_{\loc}(\mathbb R^d)$ and
$$
\||K|^{\frac{1}{2}}(\lambda+(-\Delta)^{\frac{\alpha}{2}})^{-\frac{1}{2}\frac{\alpha-1}{\alpha}}\|_{L^2(\mathbb R^d) \rightarrow L^2(\mathbb R^d)} \leq \sqrt{\delta}
$$
for some $\lambda>0$. For instance, $K$ satisfies the elliptic Morrey class condition $$|K|^{\frac{1}{\alpha-1}} \in M_{1+\epsilon}$$
with sufficiently small elliptic Morrey  norm (see \eqref{morrey_e}). Then there exists a constant $C_{d,N,\alpha}$ such that if $\delta<C_{d,N,\alpha}$, then there exist (unique, see below) probability measures $\{\mathbb P_{x}\}_{x \in \mathbb R^{dN}}$ on the canonical space of c\`{a}dl\`{a}g paths $\mathbf{D}=\mathbf{D}_{[0,\infty[}$ that deliver a weak solution to the particle system
\begin{equation}
\label{syst1_}
X_t^i=x^i-\frac{1}{N}\sum_{j=1, j \neq i}^N \int_0^t K(X_s^i-X_s^j)ds+Z_t^i,
\end{equation}
where $Z_t^i$ are independent $\alpha$-stable processes, and determine a strongly continuous Feller semigroup via
$$
e^{-t\Lambda}f(x):=\mathbb E_{\mathbb P_x} f(X_t^1,\dots,X_t^N), \quad f \in C_\infty(\mathbb R^{dN}),
$$
where $\Lambda$ is an appropriate operator realization of $\sum_{i=1}^N (-\Delta)^{\frac{\alpha}{2}}_{x^i}+\frac{1}{N}\sum_{i=1}^N \sum_{j=1, j \neq i}^N K(x^i-x^j) \cdot \nabla_{x^i}$ on $C_\infty$.

\smallskip

\item[{\rm (\textit{ii})}] {\rm (A priori Sobolev regularity for McKean-Vlasov equation)} Assume that $K$ satisfies the elliptic Morrey class condition $$|K|^{\frac{1}{\alpha-1}} \in M_{1+\epsilon}$$
with sufficiently small Morrey norm. Then, whenever $1<p<\frac{d+1}{d}$, for every $x \in \mathbb R^d$, $\lambda>0$, the solution to the Cauchy problem for the McKean-Vlasov equation
\begin{align*}
\left\{
\begin{array}{l}
\partial_t w +(-\Delta)^{\frac{\alpha}{2}} w - {\rm div\,}\big((K \star w)w\big)=0, \\
w|_{t=0}=\delta_x,
\end{array}
\right.
\end{align*}
satisfies $w\in L^\infty(\mathbb R_+,L^1(\mathbb R^d))$, $ w_\lambda :=\mathbf{1}_{t>0} e^{-\lambda t} w \in  \mathbb{W}_\alpha^{\frac{\alpha-1}{\alpha}\frac{2}{p},p} $ and
\begin{align}
\label{mild}
&(\lambda+\partial_t +(- \Delta)^{\frac{\alpha}{2}})^{\frac{\alpha-1}{\alpha}\frac{1}{p}} w_\lambda\notag\\
&\quad = (\lambda+\partial_t +(- \Delta)^{\frac{\alpha}{2}})^{-1+\frac{\alpha-1}{\alpha}\frac{1}{p}}\delta_{t=0}\delta_x + \nabla \cdot (\lambda+\partial_t +(- \Delta)^{\frac{\alpha}{2}})^{-1+\frac{\alpha-1}{\alpha}\frac{1}{p}}[(K \star w)w_\lambda].
\end{align}
 Here we assume additionally that $K$ is smooth to ensure the existence and uniqueness of $w$, but the Sobolev norm $\|w_\lambda\|_{\mathbb{W}_\alpha^{\frac{\alpha-1}{\alpha}\frac{2}{p},p} }$ is independent of the smoothness of $K$.
 \end{enumerate}
 
\end{theorem}

Theorem \ref{thm3}(\textit{i}) expands the class of admissible singularities of the interaction kernel in \eqref{syst1_}, but it does not yet settle the question of the weak well-posedness of \eqref{syst1_} since the condition on the interaction strength $\delta$ depends on $N$. 

Theorem \ref{thm3}(\textit{ii}) stays at the a priori level and exploits the fact 
that
\begin{equation}
\label{conv_incl}
|K \star w|^{\frac{1}{\alpha-1}} \in E_{1+\epsilon}
\end{equation}
(via, basically,  Young's inequality for Morrey spaces),
so Theorem \ref{thm2} applies to the equation for $w_\lambda$. In fact, this was done in \cite{D1}, but we include the details for the reader's convenience. Theorem \ref{thm2} and \eqref{conv_incl} open up a way for launching the contraction principle to establish existence and uniqueness of the mild solution to the a posteriori ($K$ no longer smooth) McKean-Vlasov equation, i.e.\,such that identity \eqref{mild} holds.
However, here we run into some limitations of our choice of Morrey class that treats time $t$ and the spatial variable $x$ in the same way, modulo the adjustment of scaling. We plan to address this elsewhere.

\begin{remark}[On the critical strength of attraction] 
\label{attr_strength_rem}

1.~In the Brownian case $\alpha=2$, $d \geq 3$, one can prove weak existence, as well as some uniqueness results and the strong Markov property, for the particle system \eqref{syst1_} with form-bounded interaction kernel (Table \ref{tab1}) for the strength of interaction $\delta$ going up to (but excluding) the first blow-up threshold, i.e.\,before the particles start to collide in finite time with positive probability and the PDE theory fails, see \cite{D5,KiS2}. This is done there using De Giorgi's method carried out in $L^p$ where $p$ is chosen as a function of the proximity of $\delta$ to the critical threshold. 

\smallskip

2.~When $d=2$ and $K(x)=\kappa x|x|^{-2}$, some PDE results for the  particle system \eqref{syst1_}, including heat kernel bounds, can still be proved, but, at the moment,  the admissible strength of attraction degenerates, although very slowly, to zero as $N \rightarrow \infty$ \cite{BK}. The probabilistic methods cover all strengths of attraction \cite{FJ}; of course, they focus on well-posedness results of different character.

\smallskip

3.~In the case $1<\alpha \leq 2$ and $N=2$, one has a PDE theory for the strength of attraction \textit{equal} to the first blow-up threshold  -- this requires working in the hyperbolic Orlicz space \cite{D6}.

\end{remark}

\bigskip

\section{Rellich from Morrey}
\label{adams_est_sect}

The proofs of Theorems \ref{thm1}, \ref{thm2} use the following proposition as a means to verify the Rellich-type perturbation estimate \eqref{main_est} -- it is needed to construct the Duhamel series in Theorem \ref{thm2} and in the proof of Theorem \ref{thm1}.

\begin{proposition}
\label{prop1}
Assume that the Borel measurable vector field $b:\mathbb R^{1+d } \rightarrow \mathbb R^d$ satisfies  {\rm($\mathbf{H}$)}.
Then, for every $1<p<\infty$, for all $\lambda>0$,
\begin{align}
\||b|^{\frac{1}{p}}(\lambda  \pm \partial_t  + & (- \Delta)^{\frac{\alpha}{2}} )^{(-1+\frac{1}{\alpha})\frac{1}{p}}\|_{L^p(\mathbb{R}^{1+d}) \rightarrow L^p(\mathbb{R}^{1+d})} \notag \\
& \leq c_{d,\alpha,p,\epsilon}\||b_{\mathfrak s}|^{\frac{1}{\alpha-1}}\|^{\frac{\alpha-1}{p}}_{E_{1+\epsilon}} + c \lambda^{(-1+\frac{1}{\alpha})\frac{1}{p}}\|b_{\mathfrak {b}} \|_\infty^{\frac{1}{p}}, \label{adams_est1}
\end{align}
and, by duality,
\begin{align}
\|(\lambda \pm \partial_t  + & (- \Delta)^{\frac{\alpha}{2}} )^{(-1+\frac{1}{\alpha})\frac{1}{p'}}|b|^{\frac{1}{p'}}\|_{L^p(\mathbb{R}^{1+d}) \rightarrow L^p(\mathbb{R}^{1+d})} \notag \\
& \leq c_{d,\alpha,p',\epsilon}\||b_{\mathfrak s}|^{\frac{1}{\alpha-1}}\|^{\frac{\alpha-1}{p'}}_{E_{1+\epsilon}} + c \lambda^{(-1+\frac{1}{\alpha})\frac{1}{p'}}\|b_{\mathfrak {b}} \|_\infty^{\frac{1}{p'}}. \label{adams_est2}
\end{align}
\end{proposition}

The elliptic counterparts of estimates \eqref{adams_est1_}, \eqref{adams_est2_} are due to Adams \cite[Theorem 7.3]{A} who, in fact, worked in the more general setting of $A_p$ weights. 
Krylov obtained a parabolic analogue of Adams' result in \cite{Kr1} for $\alpha=2$ and for integer powers of $|b|$, as required for his approach to SDEs with singular drift and discontinuous diffusion coefficients. Estimates \eqref{adams_est1_}, \eqref{adams_est2_} were obtained in \cite{D1} by extending the Adams-Krylov argument to handle $|b|^{1/p}$ and also to include $\alpha<2$. 
We use these estimates differently from Krylov. That said, the proof of Proposition \ref{prop1} in \cite{D1} is not  far from the original arguments, so it is fair to refer to these estimates as \textit{non-local Adams-Krylov estimates}.

For the reader's convenience, we include the proof from \cite{D1}.

\begin{proof}[Proof of Proposition \ref{prop1}]
It suffices to show:

\begin{enumerate}
\item[]
\textit{Let
$
|b|^{\frac{1}{\alpha-1}} \in E_{q}
$ for some $1<q<d+\alpha$.
Then for every $1<p<\infty$ there exists constant $C_{d,p,q,\alpha}$ such that
\begin{align}
\label{req_ineq}
|\langle |b|(P_{(1-\frac{1}{\alpha})\frac{2}{p}}f)^p\rangle| & \leq C_{d,p,q,\alpha}^p\||b|^{\frac{1}{\alpha-1}}\|_{E_q}^{\alpha-1}\|f\|_p^p, \\
|\langle |b|(P^\ast_{(1-\frac{1}{\alpha})\frac{2}{p}}f)^p\rangle| & \leq C_{d,p,q,\alpha}^p\||b|^{\frac{1}{\alpha-1}}\|_{E_q}^{\alpha-1}\|f\|_p^p, \notag
\end{align}
where 
\begin{align}
\label{req_ineq2}
P_\gamma f(t,x) & :=\int_{\mathbb{R}^{1+d}}p_\gamma(s,y)f(t+s,x+y)dyds, \\
P_\gamma^\ast f(t,x) & :=\int_{\mathbb{R}^{1+d}}p_\gamma(s,y)f(t-s,x+y)dyds
\end{align}
on $f \in C_c(\mathbb{R}^{1+d})$. }
\end{enumerate}

Indeed, by \eqref{ul_bounds} we have 
\begin{align*}
P_\gamma f & \geq C^{-1}(-\partial_t+(-\Delta)^{\frac{\alpha}{2}})^{-\frac{\gamma}{2}} f, \\
P_\gamma^\ast f & \geq C^{-1}(\partial_t+(-\Delta)^{\frac{\alpha}{2}})^{-\frac{\gamma}{2}} f
\end{align*}
on positive $f$, and so the assertion of Proposition \ref{prop1} follows.

We will only prove \eqref{req_ineq}. Since our regularization $\{b_n\}$ of $b$ can be treated as essentially monotone (i.e.\,at the level of Euclidean norms), we will consider $b_n:=\mathbf{1}_n b$, where $\mathbf{1}_n$ is the indicator of $\{|t| \leq n, |x| \leq n, |b(t,x)| \leq n\}$ and then use the Dominated convergence theorem.

We will need the following notations and estimates. For brevity, below we write $b$ instead of $b_n$.
Set
$$
M_\beta f(t,x):=\sup_{\rho>0}\rho^\beta \frac{1}{|C_\rho(t,x)|}\int_{C_\rho(t,x)} |f(s,y)| dy ds, \quad 0 \leq \beta \leq d+\gamma,
$$
where we put for brevity $C_\rho:=C_\rho^\alpha$,
and define the maximal function $M:=M_0$.
Also, define 
$$
\hat{M} f(t,x):=\sup_{(t,x) \in C} \frac{1}{|C|}\int_{C} |f| dy ds,
$$
where the supremum is taken over all parabolic $\alpha$-cylinders $C$ containing $(t,x)$. This function is required for harmonic-analytic arguments that are transferred from the elliptic setting (i.e.\,from the cubes). The fact that our ``cubes'' $C$ are squeezed in one of the coordinates presents no problem.

\begin{lemma}
\label{lem1}
The following are true for every $\beta \in ]\frac{\alpha \gamma}{2},d+\gamma]$, for all $(t,x) \in \mathbb{R}^{1+d}$:

\smallskip

{\rm (\textit{i})}~For all $\rho>0$, we have
\begin{align*}
P_\gamma(\mathbf{1}_{C^c_\rho}f)(t,x)  \leq K \rho^{\frac{\alpha \gamma}{2}-\beta }M_\beta f(t,x), & \qquad K:=\frac{1}{\beta-\frac{\alpha\gamma}{2}}\bigl(-\frac{\alpha \gamma}{2}+d+\alpha\bigr), \\[2mm]
P_\gamma(\mathbf{1}_{C_\rho}f)(t,x)  \leq N \rho^{\frac{\alpha \gamma}{2}}Mf(t,x), & \qquad N:= \frac{2}{\alpha \gamma}\bigl(-\frac{\alpha\gamma}{2}+d+\alpha\bigr)
\end{align*}
for all $0 \leq f \in C_c(\mathbb{R}^{1+d})$.

\smallskip

{\rm (\textit{ii})}~$$
P_\gamma f(t,x) \leq C (M_\beta f(t,x))^{\frac{\alpha \gamma}{2}\frac{1}{\beta}}(M f(t,x))^{1-\frac{\alpha \gamma}{2}\frac{1}{\beta}}, \quad C:=K^{\frac{\alpha \gamma}{2}\frac{1}{\beta}}N^{1-\frac{\alpha \gamma}{2}\frac{1}{\beta}}. 
$$

\end{lemma}
\begin{proof}[Proof of Lemma]
We repeat Krylov \cite[proof of Lemma 2.2]{Kr1} with some straightforward modifications needed to accommodate the integral kernel $p_\gamma(s,y):=\mathbf{1}_{\{s>0\}} \, s^{\frac{\gamma}{2}}(|y|^{-d-\alpha} \wedge s^{-\frac{d+\alpha}{\alpha}})$ instead of the Gaussian density $(4\pi s)^{-\frac{d}{2}}e^{-|y|^2/s}$ times $s^{\frac{\gamma}{2}-1}$. 

\smallskip

(\textit{i}) It suffices to carry out the proof for $(t,x)=(0,0)$. Below $s \geq 0$, $y \in \mathbb R^d$. Set $$Q:=\{(s,y) \mid |y| \geq s^{\frac{1}{\alpha}}\}, \quad Q^c:=\{(s,y) \mid |y| < s^{\frac{1}{\alpha}}\}.$$ With a slight abuse of notation, we write $p_\gamma(s,r)$ instead of $p_\gamma(s,y)$, where $r=|y|$.

The following result is obtained using integration by parts. Let $\xi>0$ and $g:[0,\infty[ \rightarrow [0,\infty[$ be such that
$
t^{-\xi}\int_0^t g(s) ds \rightarrow 0$ as $t \uparrow \infty.
$
Then, for every $\tau \geq 0$,
\begin{equation}
\label{lem0}
\int_\tau^\infty t^{-\xi }g(t)dt \leq \xi \int_\tau^\infty t^{-\xi-1}\biggl(\int_\tau^t g(s)ds\biggr) dt.
\end{equation}

\medskip

1.~If $r \geq s^{\frac{1}{\alpha}}$, then $p_\gamma(s,r) \leq r^{\frac{\alpha \gamma}{2}-d-\alpha}$. Therefore,
\begin{align*}
P_\gamma(f\mathbf{1}_{Q \cap C_\rho^c})(0) & \leq \int_0^\infty\int_{\mathbb R^{d}} |y|^{\frac{\alpha \gamma}{2}-d-\alpha}f(s,y)\mathbf{1}_{Q \cap C_\rho^c}dy ds \\
& = \int_{\{0 \leq s \leq |y|^\alpha, |y| \geq \rho\}} |y|^{\frac{\alpha \gamma}{2}-d-\alpha}f(s,y)dy ds \\
& = \int_{\rho}^\infty r^{\frac{\alpha \gamma}{2}-d-\alpha} \int_0^{r^\alpha} \int_{\{|y|=r\}}f(s,y)d\sigma_r ds dr \\
& (\text{we apply \eqref{lem0} with $\xi \equiv C_1=-\frac{\alpha \gamma}{2}+d+\alpha$}) \\
& \leq C_1\int_{\rho}^\infty r^{\frac{\alpha \gamma}{2}-d-\alpha-1} \int_{\rho}^r \int_0^{\theta^\alpha} \int_{\{|y|=\theta\}} f(s,y)d\sigma_\theta ds d\theta dr\\
& \leq C_1\int_{\rho}^\infty r^{\frac{\alpha \gamma}{2}-d-\alpha-1} \int_{0}^r \int_0^{\theta^\alpha} \int_{\{|y|=\theta\}} f(s,y)d\sigma_\theta ds d\theta dr\\
& = C_1\int_{\rho}^\infty r^{\frac{\alpha \gamma}{2}-d-\alpha-1} I(r)dr,
\end{align*}
where  $$I(r):=\int_{C_r}f(s,y) ds dy.$$
Clearly,
$
I(r) \leq r^{d+\alpha-\beta}M_\beta f(0),
$
so
\begin{align*}
P_\gamma(f\mathbf{1}_{Q \cap C_\rho^c})(0) & \leq C_1\int_{\rho}^\infty r^{\frac{\alpha \gamma}{2}-d-\alpha-1} r^{d+\alpha-\beta} M_\beta f(0) dr\\
& \leq K \rho^{\frac{\alpha \gamma}{2}-\beta} M_\beta f(0), \qquad K=\frac{1}{\beta-\frac{\alpha\gamma}{2}}C_1 \equiv \frac{1}{\beta-\frac{\alpha\gamma}{2}}\bigl(-\frac{\alpha \gamma}{2}+d+\alpha\bigr),
\end{align*}
where we have used the hypothesis $\beta>\frac{\alpha \gamma}{2}$.

\medskip

2.\,If $r \leq s^{\frac{1}{\alpha}}$, then $p_\gamma(s,r) = s^{\frac{\gamma}{2}-\frac{d}{\alpha}-1}$. Therefore,
\begin{align*}
P_\gamma(f\mathbf{1}_{Q^c \cap C_\rho^c})(0) & = \int_{\rho^\alpha}^\infty s^{\frac{\gamma}{2}-\frac{d}{\alpha}-1}\int_{\{|y| \leq s^{\frac{1}{\alpha}}\}}f(s,y) dy ds \\
& (\text{we apply \eqref{lem0} with $\xi \equiv C_2:=-\frac{\gamma}{2}+\frac{d}{\alpha}+1$}) \\
& \leq C_2\int_{\rho^\alpha}^\infty s^{\frac{\gamma}{2}-\frac{d}{\alpha}-2} \int_{\rho^\alpha}^s\int_{|y| \leq \vartheta^{\frac{1}{\alpha}}} f(\vartheta,y) dy d\vartheta ds \\
& \leq C_2\int_{\rho^\alpha}^\infty s^{\frac{\gamma}{2}-\frac{d}{\alpha}-2} \int_0^s\int_{|y| \leq \vartheta^{\frac{1}{\alpha}}} f(\vartheta,y) dy d\vartheta ds \\
& = C_2\int_{\rho^\alpha}^\infty s^{\frac{\gamma}{2}-\frac{d}{\alpha}-2} I(s^{\frac{1}{\alpha}})ds  = \alpha C_2 \int_\rho^\infty r^{\frac{\alpha \gamma}{2} - d-\alpha -1 }I(r)dr \\
& \leq  \alpha C_2 \int_\rho^\infty r^{\frac{\alpha \gamma}{2} - d-\alpha -1 } r^{d+\alpha-\beta} dr M_\beta f(0) \\
& = K \rho^{\frac{\alpha \gamma}{2}-\beta} M_\beta f(0)
\end{align*}
(we have used $\beta>\frac{\alpha \gamma}{2}$).
Combined with Step 1, this yields the first inequality of Lemma \ref{lem1}(\textit{i}).

\medskip

Let us prove the second inequality in Lemma \ref{lem1}(\textit{i}):

\medskip

3.~We have, using integration by parts,
\begin{align*}
P_\gamma(f\mathbf{1}_{Q \cap C_\rho})(0) & \leq \int_0^\rho r^{\frac{\alpha\gamma}{2}-d-\alpha} \int_0^{r^\alpha} \int_{\{|y|=r\}} f(s,y)d\sigma_r ds dr \\
& \leq \int_0^\rho r^{\frac{\alpha\gamma}{2}-d-\alpha} \frac{\partial}{\partial r}\int_0^r \int_0^{\theta^\alpha} \int_{\{|y|=\theta\}} f(s,y) d\sigma_\theta ds d\theta dr \\
& = J_1 + C_3\int_0^{\rho} r^{\frac{\alpha\gamma}{2}-d-\alpha-1}\int_0^r \int_0^{\theta^\alpha}\int_{\{|y|=\theta\}} f(s,y) d\sigma_\theta ds d\theta dr \\
& \leq J_1 + C_3\int_0^\rho  r^{\frac{\alpha\gamma}{2}-d-\alpha-1} I(r)dr,
\end{align*}
where $C_3:=-\frac{\alpha\gamma}{2}+d+\alpha$ and
\begin{align*}
J_1 & :=\rho^{\frac{\alpha\gamma}{2}-d-\alpha}\int_0^\rho \int_0^{\theta^\alpha}\int_{\{|y|=\theta\}} f(s,y)d\sigma_\theta ds d\theta\\
&  \leq \rho^{\frac{\alpha\gamma}{2}-d-\alpha}I(\rho).
\end{align*}
Since $I(\rho) \leq \rho^{d+\alpha}Mf(0)$, we have
\begin{align*}
P_\gamma(f\mathbf{1}_{Q \cap C_\rho})(0) & \leq \rho^{\frac{\alpha\gamma}{2}-d-\alpha}\rho^{d+\alpha}Mf(0) + C_3  \int_0^\rho  r^{\frac{\alpha\gamma}{2}-d-\alpha-1} r^{d+\alpha} dr Mf(0) \\
& = \rho^{\frac{\alpha\gamma}{2}}Mf(0) + \frac{2C_3}{\alpha \gamma}\rho^{\frac{\alpha\gamma}{2}}Mf(0) \\
& = N \rho^{\frac{\alpha\gamma}{2}}Mf(0), \qquad N=1+\frac{2C_3}{\alpha \gamma} \equiv \frac{2}{\alpha \gamma}\bigl(-\frac{\alpha\gamma}{2}+d+\alpha\bigr).
\end{align*}

4.~Next,
\begin{align*}
P_\gamma(f\mathbf{1}_{Q^c \cap C_\rho})(0) & = \int_0^{\rho^\alpha} s^{\frac{\gamma}{2}-\frac{d}{\alpha}-1}\int_{\{|y| \leq s^{\frac{1}{\alpha}}\}}f(s,y) dy ds \\
& \leq J_2+ C_4 \int_0^{\rho^\alpha} s^{\frac{\gamma}{2}-\frac{d}{\alpha}-2} I(s^{\frac{1}{\alpha}}) ds \\
& = J_2 + \alpha C_4\int_0^\rho r^{\frac{\alpha \gamma}{2}-d-\alpha-1} I(r)dr,
\end{align*}
where $C_4=-\frac{\gamma}{2}+\frac{d}{\alpha}+1$ and
\begin{align*}
J_2 & :=\rho^{\frac{\alpha\gamma}{2}-d-\alpha}\int_0^{\rho^\alpha}\int_{\{|y| \leq \vartheta^{\frac{1}{\alpha}}\}} f(\vartheta,y)dy d\vartheta  \\
& \leq \rho^{\frac{\alpha\gamma}{2}-d-\alpha}I(\rho).
\end{align*}
Thus, applying again inequality $I(\rho) \leq \rho^{d+\alpha}Mf(0)$, we obtain 
$P_\gamma(f\mathbf{1}_{Q^c \cap C_\rho})(0) \leq N\rho^{\frac{\alpha\gamma}{2}}Mf(0)$. Together with the result of step 3,
this completes the proof of the second inequality in Lemma \ref{lem1}(\textit{i}). 

\medskip

(\textit{ii}) is obtained from (\textit{i}) by adding up both inequalities and minimizing in $\rho$. 
\end{proof}

Armed with Lemma \ref{lem1}(\textit{ii}), we now prove \eqref{req_ineq}. Put $u:=P_{(1-\frac{1}{\alpha})\frac{2}{p}}f$. Then 
\begin{align}
|\langle |b|(P_{(1-\frac{1}{\alpha})\frac{2}{p}}f)^p\rangle| &= |\langle |b||u|^{p-1},P_{(1-\frac{1}{\alpha})\frac{2}{p}}f\rangle| \notag \\
& \leq |\langle P^\ast_{(1-\frac{1}{\alpha})\frac{2}{p}}(|b||u|^{p-1}),f\rangle| \leq \|P^\ast_{(1-\frac{1}{\alpha})\frac{2}{p}}(|b||u|^{p-1})\|_{p'}\|f\|_p. \label{ineq2}
\end{align}
Thus, to obtain \eqref{req_ineq}, we need to bound the coefficient $\|P^\ast_{(1-\frac{1}{\alpha})\frac{2}{p}}(|b||u|^{p-1})\|_{p'}$. 
To this end, we apply the pointwise estimate
\begin{align*}
P_{(1-\frac{1}{\alpha})\frac{2}{p}}^\ast(|b||u|^{p-1}) = P_{(1-\frac{1}{\alpha})\frac{2}{p}}^\ast(|b|^{\frac{1}{p}+\nu} |b|^{\frac{1}{p'}-\nu}|u|^{p-1}) \leq P_{(1-\frac{1}{\alpha})\frac{2}{p}}^\ast(|b|^{1+\nu p})^{\frac{1}{p}} (P_{(1-\frac{1}{\alpha})\frac{2}{p}}^\ast(|b|^{1-\nu p'}|u|^p))^{\frac{1}{p'}}
\end{align*} 
for a small $\nu>0$ such that $1+\nu p < q_0$ for some fixed $q_0<q$. 
Hence 
\begin{align}
\|P^\ast_{(1-\frac{1}{\alpha})\frac{2}{p}}(|b||u|^{p-1})\|_{p'}^{p'} \leq \langle |b|^{1-\nu p'}|u|^p, P_{(1-\frac{1}{\alpha})\frac{2}{p}}[P^\ast_{(1-\frac{1}{\alpha})\frac{2}{p}}(|b|^{1+\nu p})]^{\frac{1}{p-1}}\rangle. \label{PP}
\end{align}

Step 1.\,Applying  Lemma \ref{lem1}(\textit{ii}) with $\gamma:=(1-\frac{1}{\alpha})\frac{2}{p}$, $\beta:=(\alpha-1)(1+\nu p)$  (clearly, $\beta>\frac{\alpha \gamma}{2}$) or, rather, its straightforward variant for $P_\gamma^\ast$ and noting that in this case
$$
\frac{\alpha \gamma}{2}\frac{1}{\beta}=\frac{1}{p}\frac{1}{1+\nu p},
$$ 
and (this will be needed a few lines later)
\begin{align*}
M_{(\alpha-1)(1+\nu p)}|b|^{1+\nu p}(t,x) & =\sup_{\rho>0}\rho^{(\alpha-1)(1+\nu p)}\frac{1}{|C_\rho(t,x)|}\int_{C_\rho(t,x)} |b(s,y)|^{1+\nu p} dy ds \\
& \leq \sup_{\rho>0} \rho^{(\alpha-1)(1+\nu p)} \biggl(\frac{1}{|C_\rho(t,x)|}\int_{C_\rho(t,x)} |b(s,y)|^{\frac{1+\nu p}{\alpha-1}} dy ds \biggr)^{\alpha-1} \\
& = \biggl[\sup_{\rho>0} \rho \biggl(\frac{1}{|C_\rho(t,x)|}\int_{C_\rho(t,x)} |b(s,y)|^{\frac{1+\nu p}{\alpha-1}} dy ds \biggr)^{\frac{1}{1+\nu p}}\biggr]^{(\alpha-1)(1+\nu p)} \\[3mm]
& \leq \||b|^{\frac{1}{\alpha-1}}\|_{E_{1+\nu p}}^{(\alpha-1)(1+\nu p)}.
\end{align*}
we obtain 
\begin{align*}
P^\ast_{(1-\frac{1}{\alpha})\frac{2}{p}}(|b|^{1+\nu p}) & \leq N (M_{(\alpha-1)(1+\nu p)}|b|^{1+\nu p})^{\frac{1}{p}\frac{1}{1+\nu p}}(M |b|^{1+\nu p})^{1-\frac{1}{p}\frac{1}{1+\nu p}}\\
& \leq C\||b|^{\frac{1}{\alpha-1}}\|_{E_{1+\nu p}}^{\frac{\alpha-1}{p}}(\hat{M}|b|^{1+\nu p})^{1-\frac{1}{p}\frac{1}{1+\nu p}} \\
& \leq C\||b|^{\frac{1}{\alpha-1}}\|_{E_{q_0}}^{\frac{\alpha-1}{p}}(\hat{M}|b|^{1+\nu p})^{1-\frac{1}{p}\frac{1}{1+\nu p}}.
\end{align*}
At this point, let us assume that $|b|^{1+\nu p}$ is an $A_1$-weight, i.e.
\begin{equation}
\label{a1}
\hat{M}|b|^{1+\nu p} \leq C_0 |b|^{1+\nu p}
\end{equation}
for some constant $C_0$
(we will get rid of this assumption later).

\begin{remark}
Up to this point all constants were explicit. Now we picked up constant $C_0$ which, as will be explained below (Step 4), depends on the constants in some classical inequalities of harmonic analysis.
\end{remark}

Then
$$
P^\ast_{(1-\frac{1}{\alpha})\frac{2}{p}}(|b|^{1+\nu p}) \leq C_2\|b\|_{E_{q_0}}^{\frac{1}{p}}|b|^{1+\nu p - \frac{1}{p}}.
$$

Step 2.\,After applying the previous estimate in \eqref{PP}, one sees that now we need to estimate
$$
P_{(1-\frac{1}{\alpha})\frac{2}{p}} (|b|^{(1+\nu p - \frac{1}{p})\frac{1}{p-1}})=P_{(1-\frac{1}{\alpha})\frac{2}{p}} (|b|^{\frac{1}{p}+\nu p'}).
$$
Selecting $\nu$ even smaller, if needed, one may assume that $\frac{1}{p}+\nu p'<q_0$.
By  Lemma \ref{lem1}(\textit{ii}) with $\gamma:=(1-\frac{1}{\alpha})\frac{2}{p}$, $\beta:=(\alpha-1)(\frac{1}{p}+\nu p')$ (in this case, again, $\beta>\frac{\alpha \gamma}{2}$),
\begin{align*}
P_{(1-\frac{1}{\alpha})\frac{2}{p}} (|b|^{\frac{1}{p}+\nu p'}) & \leq C (M_{(\alpha-1)(\frac{1}{p}+\nu p')}|b|^{\frac{1}{p}+\nu p'})^{\frac{1}{p}\frac{1}{\frac{1}{p}+\nu p'}}(M|b|^{\frac{1}{p}+\nu p'})^{1-\frac{1}{p}\frac{1}{\frac{1}{p}+\nu p'}} \\
& \leq C\||b|^{\frac{1}{\alpha-1}}\|_{E_{\frac{1}{p}+\nu p'}}^{\frac{\alpha-1}{p}} (\hat{M}|b|^{\frac{1}{p}+\nu p'})^{1-\frac{1}{p}\frac{1}{\frac{1}{p}+\nu p'}} \\
& \leq C\||b|^{\frac{1}{\alpha-1}}\|_{E_{q_0}}^{\frac{\alpha-1}{p}} (\hat{M}|b|^{\frac{1}{p}+\nu p'})^{1-\frac{1}{p}\frac{1}{\frac{1}{p}+\nu p'}}.
\end{align*}
In addition to \eqref{a1}, let us temporarily assume that
\begin{equation}
\label{a2}
\hat{M}|b|^{\frac{1}{p}+\nu p'} \leq C_0|b|^{\frac{1}{p}+\nu p'}.
\end{equation}
Then
$$
P_{(1-\frac{1}{\alpha})\frac{2}{p}} (|b|^{\frac{1}{p}+\nu p'})  \leq C_2 \||b|^{\frac{1}{\alpha-1}}\|_{E_{q_0}}^{\frac{\alpha-1}{p}} |b|^{\nu p'}.
$$

Step 3.\,Applying the results of Steps 1 and 2  in \eqref{PP}, we obtain
$$
\|P^\ast_{(1-\frac{1}{\alpha})\frac{2}{p}}(|b||u|^{p-1})\|_{p'} \leq C_3 \||b|^{\frac{1}{\alpha-1}}\|_{E_{q_0}}^{\frac{\alpha -1}{p}} \langle |b||u|^p\rangle^{\frac{1}{p'}}.
$$
Therefore, \eqref{ineq2} yields
$
\langle |b||u|^{p}\rangle \leq C_3 \||b|^{\frac{1}{\alpha-1}}\|_{E_{q_0}}^{\frac{\alpha-1}{p}} \langle |b||u|^p\rangle^{\frac{1}{p'}}\|f\|_p,
$
hence
\begin{equation}
\label{ineq4}
\langle |b||u|^{p}\rangle^{\frac{1}{p}} \leq C_3\||b|^{\frac{1}{\alpha-1}}\|_{E_{q_0}}^{\frac{\alpha-1}{p}} \|f\|_p.
\end{equation}

Step 4.\,Now, we get rid of the assumptions \eqref{a1} and \eqref{a2} at the expense of replacing $\|b\|_{E_{q_0}}$ in \eqref{ineq4} by $\|b\|_{E_q}$, where, recall, $q_0<q$. This, in turn, will give us \eqref{req_ineq}.
Fix $q_0<q_1<q$ and define $$\tilde{b}:=(\hat{M} |b|^{q_1})^{\frac{1}{q_1}}.$$ Then $\tilde{b} \geq |b|$ and $\tilde{b}$ satisfies 
\begin{equation}
\label{a3}
\hat{M}\tilde{b}^{q_0} \leq C_0\tilde{b}^{q_0},
\end{equation}
see \cite[p.158]{GR}.
Since $1+\nu p < q_0$, $\frac{1}{p}+\nu p'<q_0$, both inequalities \eqref{a1} and \eqref{a2} for $\tilde{b}$ follow from \eqref{a3}, and so we have
$$
\langle |b||u|^{p}\rangle^{\frac{1}{p}} \leq C_3 \|\tilde{b}^{\frac{1}{\alpha-1}}\|_{E_{q_0}}^{\frac{\alpha-1}{p}} \|f\|_p.
$$
It remains to apply the inequality $\||\tilde{b}|^{\frac{1}{\alpha-1}}\|_{E_{q_0}} \leq C \||b|^{\frac{1}{\alpha-1}}\|_{E_{q}}$, which was established in \cite[proof of Theorem 4.1]{Kr1} for $\alpha=2$. In fact, \cite{Kr1} gave two proofs of this fact; one is based on the Fefferman-Stein inequality, and the other one is more direct. 
In both cases, as is mentioned in \cite{Kr1}, transferring the corresponding harmonic-analytic results from cubes to $\hat{M}$ defined on the parabolic cylinders (with $\alpha=2$)  presents no problem. (Note that our parabolic cylinders with $1<\alpha<2$ get closer to the cubes as $\alpha \downarrow 1$. So, at this last step of the proof we find ourselves between \cite{Kr1} and the classical setting of cubes in $\mathbb{R}^{1+d}$.) 
\end{proof}

\bigskip

\section{Localization of Proposition \ref{prop1}}
\label{loc_sect}

The proof of Theorem \ref{thm1} requires us to control the convergence
of the Duhamel series in the weighted Lebesgue spaces
$L^p_{\rho_{x_0}}(\mathbb R^{1+d})=L^p(\mathbb R^{1+d},\rho_{x_0}(y)dydt)$,
uniformly in $x_0\in\mathbb R^d$, where
\begin{equation}
\label{rho_def}
\rho_{x_0}(y)=(1+\kappa|y-x_0|^2)^{-\frac{\theta}{2}},
\quad \theta>d,\quad \kappa>0.
\end{equation}
This will be achieved via
a localized ($\equiv$ weighted) variant of Proposition \ref{prop1}, i.e.\,Proposition \ref{prop1_weight} that we state and prove below.

\medskip

Let us comment first on the settings where such localization is easier:

\medskip

\noindent\emph{Case $\alpha=2$.} One can simply commute the weight $\rho_{x_0}$ with the Laplacian and establish a weighted variant of Proposition \ref{prop1} based on the fact that the derivatives of $\rho_{x_0}$ are dominated by the weight itself times a constant that becomes smaller as $\kappa \downarrow 0$, uniformly in $x_0$, see \cite{D2}. In the non-local case $1<\alpha<2$,  however, such direct calculations are problematic.

\medskip

\noindent\emph {Elliptic case.} Let $1<\alpha<2$. In \cite{KM}, the authors considered weakly form-bounded drift $b=b(x)$ with weak form-bound $\delta$ (see the definition in the introduction), and obtained the localized estimate
\begin{equation}
\label{loc_est2}
\||b|^{\frac{1}{p}} (\lambda+(-\Delta)^{\frac{\alpha}{2}})^{-1+\frac{1}{\alpha}}|b|^{\frac{1}{p'}}\|_{L^p(\mathbb R^d,\rho_{x_0}(y)dy) \rightarrow L^p(\mathbb R^d,\rho_{x_0}(y)dy)} \leq C\delta
\end{equation}  
by showing that the fractional Laplacian in the weighted space is still the generator of a symmetric Markov semigroup, so that the Stroock-Varopoulos inequalities for symmetric Markov generators that allow  to deduce the key estimate 
\begin{equation}
\label{loc_est}
\||b|^{\frac{1}{p}} (\lambda+(-\Delta)^{\frac{\alpha}{2}})^{-1+\frac{1}{\alpha}}|b|^{\frac{1}{p'}}\|_{L^p(\mathbb R^d) \rightarrow L^p(\mathbb R^d)} \leq \delta
\end{equation} 
from the weak form-boundedness of $b$ still apply. This approach leads to a rather lengthy proof of \eqref{loc_est2}. Before proceeding to the parabolic case, let us demonstrate that in the setting of \cite{KM} there is a simple way to deduce \eqref{loc_est2} from \eqref{loc_est} in an important special case $\lambda=0$ (e.g.\,covered by the fractional Hardy inequality \eqref{hardy_frac}).

\begin{proposition} 
\label{loc_prop1}
Assume that \eqref{loc_est} holds with $\lambda=0$.
Then \eqref{loc_est2} holds for $p\geq\frac{\theta}{2\alpha-1}$,
for all $\lambda>0$ and $x_0\in\mathbb R^d$, with a constant
$C=C(d,\alpha,p,\theta,\kappa,\lambda)$ independent of $x_0$.
\end{proposition}

\begin{proof}
	Fix $x_0\in\mathbb R^d$. We first take $\lambda=1$; the general
case follows by dilation, as explained at the end of the proof.
We first note that for $\gamma=2-\frac2\alpha>0$
 	\begin{align}
		(1+(-\Delta)^{\frac{\alpha}{2}})^{-\frac{\gamma}{2}} (x-y) \notag
		&=\frac1{\Gamma(\gamma/2)}\int_0^\infty e^{-t} t^{\frac{\gamma}{2}-1} e^{-t(-\Delta)^{\frac{\alpha}{2}}} (x-y) dt \notag \\
		& (\text{use the standard upper bound on the heat kernel of the fractional Laplacian}) \notag \\
		&\leq C_1 \int_0^\infty e^{-t} t^{\frac{\gamma}{2}-1} (t^{-\frac{d}{\alpha}} \wedge t |x-y|^{-d-\alpha}) dt \notag \\
		&\leq C_2 
		\begin{cases}
			|x-y|^{-d-\alpha}
			\quad  &\text{if} \quad |x-y|> 1\\
			|x-y|^{-d+\frac{\gamma \alpha}{2}}
			\quad  &\text{if} \quad |x-y|\leq 1 \label{frac_bd}
		\end{cases} .
	\end{align}
	Also, note that
	$$
	1+\kappa|y-x_0|^2
\leq 1+2\kappa|x-y|^2+2\kappa|x-x_0|^2
\leq 2(1+\kappa|x-y|^2)(1+\kappa|x-x_0|^2).
	$$
Therefore, we can estimate the integral kernel of the ``weighted'' operator $\rho_{x_0}^{\frac{1}{p}}(1+(-\Delta)^{\frac{\alpha}{2}})^{-1+\frac{1}{\alpha}}\rho_{x_0}^{-\frac{1}{p}}$:
		\begin{align}
		&\big[\rho_{x_0}^{\frac{1}{p}}(1+(-\Delta)^{\frac{\alpha}{2}})^{-1+\frac{1}{\alpha}}\rho_{x_0}^{-\frac{1}{p}}\big](x,y) \notag \\
		&\leq (1+\kappa|x-x_0|^2)^{-\frac{\theta}{2p}} (1+(-\Delta)^{\frac{\alpha}{2}})^{-1+\frac{1}{\alpha}} (x-y) (1+\kappa|y-x_0|^2)^{\frac{\theta}{2p}} \notag \\
		&\leq 2^{\frac{\theta}{2p}}(1+\kappa|x-y|^2)^{\frac{\theta}{2p}} (1+(-\Delta)^{\frac{\alpha}{2}})^{-1+\frac{1}{\alpha}} (x-y). \label{expr}
	\end{align}
We now apply the bound \eqref{frac_bd}. That is, if $|x-y| \leq 1$, then the right-hand side of \eqref{expr} is dominated by $|x-y|^{-d+\alpha-1}$ times a generic constant. If $|x-y|>1$, then \eqref{expr} is dominated by $|x-y|^{\frac{\theta}{p}-d-\alpha}$, in which case, if $p$ is chosen so that $\frac{\theta}{p}-\alpha \leq \alpha-1$, then \eqref{expr} $\leq |x-y|^{-d+\alpha-1}$ (times a constant). To summarize, if $p \geq \frac{\theta}{2\alpha-1}$, then
$$
\big[\rho_{x_0}^{\frac{1}{p}}(1+(-\Delta)^{\frac{\alpha}{2}})^{-1+\frac{1}{\alpha}}\rho_{x_0}^{-\frac{1}{p}}\big](x,y) \leq C(-\Delta)^{-\frac{\alpha-1}{2}}(x-y), \quad x,y \in \mathbb R^d,
$$ 
with $C$ independent of $x_0$. Multiplication of the kernel by
$|b(x)|^{1/p}|b(y)|^{1/p'}$ and \eqref{loc_est} with $\lambda=0$
give \eqref{loc_est2} for $\lambda=1$.

For general $\lambda>0$, the kernel scales as follows (set $ A=(-\Delta)^{\alpha/2}$):
$$
(\lambda+A)^{-1+\frac1\alpha}(z)
=\lambda^{\frac{d-\alpha+1}{\alpha}}
 (1+A)^{-1+\frac1\alpha}(\lambda^{1/\alpha}z).
$$
The same dilation replaces $\kappa$ by $\kappa\lambda^{-2/\alpha}$.
The homogeneity of the Riesz kernel therefore gives the same
domination with a constant depending on $\lambda,\kappa$, but not
on $x_0$. This proves \eqref{loc_est2} for every $\lambda>0$. 
\end{proof}

We will use some elements of the proof of  Proposition \ref{loc_prop1} in the proof of Theorem \ref{thm3}(\textit{i}).
 
 \medskip

\noindent \emph{Parabolic case.} In the proof above we have exploited the fact that the kernel $(\lambda+(-\Delta)^{\frac{\alpha}{2}})^{-\frac{\gamma}{2}}(x-y)$ tends to zero as $|x-y| \rightarrow \infty$  strictly faster  than the corresponding Riesz kernel. This does not work for the parabolic kernel $(\lambda \pm \partial_t + (-\Delta)^{\frac{\alpha}{2}})^{-\frac{\gamma}{2}}(t,x-y)$. 

However, in contrast to weakly form-bounded drifts, the Morrey class condition on $|b|$ provides some extra flexibility: one can exploit the fact that to localize parabolic Adams' estimates \eqref{adams_est1}, \eqref{adams_est2} one does not need to control the entire integral kernel $(\lambda \pm \partial_t + (-\Delta)^{\frac{\alpha}{2}})^{-\frac{\gamma}{2}}(t,x-y)$. 
Below we extend Krylov's argument from $\alpha=2$ by estimating the ``far'' term convolution before summing over the spatial cubes.

 Let $Q_r(x) \subset \mathbb R^d$ denote the cube centered at $x$ with edge length $r$. Set
$$A_{\lambda,p}:=\lambda^{-\frac{1}{p'}-(1-\frac{1}{\alpha})\frac{1}{p}} \downarrow 0 \quad \text{ as $\lambda \uparrow \infty$.}$$
The next Proposition \ref{prop1_weight} is used in the proof of Theorem \ref{thm1}. (The latter requires $p>d+1$, and so the condition $p>\frac{\theta}{\alpha}$ in Proposition \ref{prop1_weight} is automatically satisfied.)

\begin{proposition}
\label{prop1_weight}
\label{loc_prop2}
Assume that the Borel measurable vector field $b$ satisfies  {\rm($\mathbf{H}$)}.
Then, for every $p>\frac{\theta}{\alpha}$, provided that $\theta$  in the definition  \eqref{rho_def}  of weight $\rho_{x_0}$ is chosen so that $d<\theta<d+\alpha$ and $\kappa$ is fixed sufficiently small (that is, depending only on $d$, $\alpha$ and the norms of $b_{\mathfrak b}$ and $b_\mathfrak{s}$), the following are true for all $\lambda>1$ and
$x_0\in\mathbb R^d$, with the same choice of $\kappa$:

\begin{enumerate}
\item[{\rm (\textit{i})}]
\begin{align}
\||b|^{\frac{1}{p}}(\lambda \pm \partial_t  + & (- \Delta)^{\frac{\alpha}{2}} )^{(-1+\frac{1}{\alpha})\frac{1}{p}}\|_{L_{\rho_{x_0}}^p(\mathbb{R}^{1+d}) \rightarrow L_{\rho_{x_0}}^p(\mathbb{R}^{1+d})} \notag \\
& \leq K_1\sup_{x \in \mathbb R^d}\|\mathbf{1}_{Q_2(x)}|b_{\mathfrak s}|^{\frac{1}{\alpha-1}}\|^{\frac{\alpha-1}{p}}_{E_{1+\epsilon}}  + K_2 
A_{\lambda,p}
 \big\||b_{\mathfrak s}|^{\frac{1}{\alpha-1}}\big\|_{E_{1+\epsilon}}^{\frac{\alpha-1}{p}} + K_3 \lambda^{-\frac{1}{p}(1-\frac{1}{\alpha})}\|b_{\mathfrak b}\|^{\frac{1}{p}}_{L^\infty(\mathbb R^{1+d})} \label{adams_est_weight},
\end{align}

\item[{\rm (\textit{ii})}]
\begin{align}
\|(\lambda \pm \partial_t  + & (- \Delta)^{\frac{\alpha}{2}} )^{(-1+\frac{1}{\alpha})\frac{1}{p'}}|b|^{\frac{1}{p'}}\|_{L_{\rho_{x_0}}^p(\mathbb{R}^{1+d}) \rightarrow L_{\rho_{x_0}}^p(\mathbb{R}^{1+d})} \notag \\
& \leq K_1\sup_{x \in \mathbb R^d}\|\mathbf{1}_{Q_2(x)}|b_{\mathfrak s}|^{\frac{1}{\alpha-1}}\|^{\frac{\alpha-1}{p'}}_{E_{1+\epsilon}} + K_2 
A_{\lambda,p'}
 \big\||b_{\mathfrak s}|^{\frac{1}{\alpha-1}}\big\|_{E_{1+\epsilon}}^{\frac{\alpha-1}{p'}} + K_3 \lambda^{-\frac{1}{p'}(1-\frac{1}{\alpha})}\|b_{\mathfrak b}\|^{\frac{1}{p'}}_{L^\infty(\mathbb R^{1+d})}, \label{adams_est2_weight}
\end{align}
for constants $K_i=K_i(d,\alpha,p,\epsilon,\theta,\kappa)$
independent of $x_0$ and $\lambda$.

\end{enumerate}

\end{proposition}

Fix the weight centre $x_0\in\mathbb R^d$.
We may establish \eqref{adams_est_weight}, \eqref{adams_est2_weight} separately for the bounded part $b_{\mathfrak b}$ and the unbounded part $b_{\mathfrak s}$.
For brevity, set $A:=(-\Delta)^{\frac{\alpha}{2}}$ and $\partial:=\partial_t$.

\textit{Proof for the unbounded part $b_{\mathfrak{s}}$.} Put 
$$
M:=\big\||b_{\mathfrak s}|^{\frac1{\alpha-1}}\big\|_{E_{1+\epsilon}}
$$
and
$$
m_b:=\sup_{t\in\mathbb R,\,z\in\mathbb R^d}
\int_t^{t+1}\int_{Q_1(z)}|b_{\mathfrak s}(s,y)|\,dyds
\le C M^{\alpha-1}.
$$
The last inequality is a straightforward consequence of the Morrey class condition on $b_{\mathfrak{s}}$. 
The operator with $-\partial$ is treated by reversing time, which does not change the definition of $m_b$ and the estimates below.

\medskip

(\textit{i}) Fix $x\in\mathbb R^d$. Let $\langle,\rangle$ denote, as before, the integration over $\mathbb R^{1+d}$. We have
\begin{align*}
\|\mathbf1_{Q_1(x)}|b_{\mathfrak s}|^{\frac1p}(\lambda+\partial+A)^{(-1+\frac1\alpha)\frac1p}f\|_p^p
&=\langle\mathbf1_{Q_1(x)}|b_{\mathfrak s}|
 |(\lambda+\partial+A)^{(-1+\frac1\alpha)\frac1p}f|^p\rangle\\
&\le C\langle\mathbf1_{Q_1(x)}|b_{\mathfrak s}|
 |(\lambda+\partial+A)^{(-1+\frac1\alpha)\frac1p}\mathbf1_{Q_2(x)}f|^p\rangle\\
&\quad+C\langle\mathbf1_{Q_1(x)}|b_{\mathfrak s}|
 |(\lambda+\partial+A)^{(-1+\frac1\alpha)\frac1p}\mathbf1_{Q_2^c(x)}f|^p\rangle\\
&=:S(x)+R(x),
\end{align*}
i.e.\,``near'' and ``far'' terms.

Selecting $\kappa$ sufficiently small, independently of $x_0$, so that
$$
\big|\rho_{x_0}(x)-\sup_{y\in Q_2(x)}\rho_{x_0}(y)\big|
\le\frac1{10}\rho_{x_0}(x),
\qquad x,x_0\in\mathbb R^d,
$$
we have
\begin{align*}
	\|\rho_{x_0}^{1/p}|b_{\mathfrak s}|^{1/p}
	(\lambda+\partial+A)^{(-1+\frac1\alpha)\frac1p}f\|_p^p
	\le\frac{11}{10}\sum_{x\in\mathbb Z^d}
	\rho_{x_0}(x)(S(x)+R(x)).
\end{align*}

Step 1.~By Proposition \ref{prop1},
\begin{equation}
\label{est_S}
S(x)\le c_{d,\alpha,p,\epsilon}^p
\|\mathbf1_{Q_1(x)}|b_{\mathfrak s}|^{\frac1{\alpha-1}}\|_{E_{1+\epsilon}}^{\alpha-1}
\|\mathbf1_{Q_2(x)}f\|_p^p,
\end{equation}
which gives
\begin{align*}
	\sum_{x\in\mathbb Z^d}\rho_{x_0}(x)S(x)
	&\le c_{d,\alpha,p,\epsilon}^p
	\sup_{x\in\mathbb R^d}
	\|\mathbf1_{Q_2(x)}|b_{\mathfrak s}|^{1/(\alpha-1)}\|_{E_{1+\epsilon}}^{\alpha-1}
	\sum_{x\in\mathbb Z^d}\rho_{x_0}(x)\|\mathbf1_{Q_2(x)}f\|_p^p\\
	&\le C\sup_{x\in\mathbb R^d}
	\|\mathbf1_{Q_2(x)}|b_{\mathfrak s}|^{1/(\alpha-1)}\|_{E_{1+\epsilon}}^{\alpha-1}
	\|\rho_{x_0}^{1/p}f\|_p^p.
\end{align*}
In the last step we used the bounded overlap of the cubes $Q_2(x)$ and the comparability of the weight on each such cube. In what follows, the value of constant $C$ can change from line to line, but it remains independent of $x_0$.

Step 2.~Let us prove that 
\begin{equation}
\label{est_R}
\sum_{x\in\mathbb Z^d}\rho_{x_0}(x)R(x)
\le C A_{\lambda,p}^p M^{\alpha-1}
\|\rho_{x_0}^{1/p}f\|_p^p.
\end{equation}
With this and Step 1, we obtain the sought weighted bound \eqref{adams_est_weight} for $b_{\mathfrak s}$.

We first explain how the weight can be brought inside the integral kernel of
$(\lambda+\partial+A)^{-\frac1p(1-\frac1\alpha)}$. By the triangle inequality applied to $(1,\sqrt{\kappa}(y-x_0)) \in \mathbb R^{1+d}$, we have
\begin{align*}
(1+\kappa|z-x_0|^2)^{1/2}
&\le (1+\kappa|y-x_0|^2)^{1/2}+\sqrt\kappa|z-y|\\
&\le (1+\kappa|y-x_0|^2)^{1/2}(1+\sqrt\kappa|z-y|).
\end{align*}
Hence, by the definition of $\rho_{x_0}$,
$$
\frac{\rho_{x_0}(y)^{1/p}}{\rho_{x_0}(z)^{1/p}}
=\left(\frac{(1+\kappa|z-x_0|^2)^{1/2}}
{(1+\kappa|y-x_0|^2)^{1/2}}\right)^{\theta/p}
\le (1+\sqrt\kappa|y-z|)^{\theta/p}
\le C(1+|y-z|)^{\theta/p}.
$$
For $y\in Q_1(x)$, we also have
$\rho_{x_0}(x)^{1/p}\le C\rho_{x_0}(y)^{1/p}$. Thus
$$
\frac{\rho_{x_0}(x)^{1/p}}{\rho_{x_0}(z)^{1/p}}
\le C\frac{\rho_{x_0}(y)^{1/p}}{\rho_{x_0}(z)^{1/p}}
\le C(1+|y-z|)^{\theta/p},
\qquad y\in Q_1(x),
$$
with constants independent of $x_0$ and $x$.

Recall that $e^{-\lambda \tau}q_{\frac2p(1-\frac1\alpha)}(\tau,\xi)$ denotes the integral kernel of $(\lambda+\partial+A)^{-\frac1p(1-\frac1\alpha)}$. By \eqref{ul_bounds},
for $\tau>0$ and $\xi\ne0$,
$$
q_{\frac2p(1-\frac1\alpha)}(\tau,\xi)
\le C\tau^{\frac1p(1-\frac1\alpha)-1}
\left(\tau^{-d/\alpha}\wedge\frac{\tau}{|\xi|^{d+\alpha}}\right)
\le C\tau^{\frac1p(1-\frac1\alpha)}|\xi|^{-d-\alpha}.
$$
Fix $y\in Q_1(x)$. We now insert the factor
$\rho_{x_0}(z)^{-1/p}\rho_{x_0}(z)^{1/p}=1$ into the integral
representation for the parabolic resolvent above and use the preceding pointwise bounds for the ratios of $\rho_{x_0}$:
\begin{align*}
&\rho_{x_0}(x)^{1/p}
\big|(\lambda+\partial+A)^{-\frac1p(1-\frac1\alpha)}\mathbf1_{Q_2^c(x)}f(t,y)\big|\\
&\quad\le\int_0^\infty\!\int_{Q_2^c(x)}
 e^{-\lambda\tau}q_{\frac2p(1-\frac1\alpha)}(\tau,y-z)
 \frac{\rho_{x_0}(x)^{1/p}}{\rho_{x_0}(z)^{1/p}}
 \rho_{x_0}(z)^{1/p}|f(t-\tau,z)|\,dzd\tau\\
&\quad\le C\int_0^\infty\!\int_{Q_2^c(x)}
 e^{-\lambda\tau}\tau^{\frac1p(1-\frac1\alpha)}
 |y-z|^{-d-\alpha}(1+|y-z|)^{\theta/p}
 \rho_{x_0}(z)^{1/p}|f(t-\tau,z)|\,dzd\tau.
\end{align*}
For $y\in Q_1(x)$ and $z\in Q_2^c(x)$, we have $|y-z|\ge1/2$.
Consequently, $1+|y-z|\le3|y-z|$, and the last integral is bounded by
$$
C\int_0^\infty\!\int_{\{|y-z|\ge1/2\}}
 e^{-\lambda\tau}\tau^{\frac1p(1-\frac1\alpha)}
 |y-z|^{-d-\alpha+\theta/p}\rho_{x_0}(z)^{1/p}|f(t-\tau,z)|\,dzd\tau.
$$
Here we have also enlarged the domain of spatial integration. Thus, writing $*$ for convolution
on $\mathbb R^{1+d}$ and setting
$$
k_\lambda(\tau,z):=
C\mathbf1_{\{\tau>0,\ |z|\ge1/2\}}
 e^{-\lambda\tau}\tau^{\frac1p(1-\frac1\alpha)} |z|^{-d-\alpha+\theta/p},
$$
we can rewrite the previous estimate as
\begin{align*}
\rho_{x_0}(x)^{1/p}\mathbf1_{Q_1(x)}(y)
 \big|(\lambda+\partial+A)^{-\frac1p(1-\frac1\alpha)}\mathbf1_{Q_2^c(x)}f(t,y)\big|\le \mathbf1_{Q_1(x)}(y)(k_\lambda*(\rho_{x_0}^{1/p}|f|))(t,y).
\end{align*}

\begin{claim} 
\label{cl_aux}
The following are true:
$$
\|k_\lambda\|_1 \le C\lambda^{-\frac1p(1-\frac1\alpha)-1}
$$
and, defining $k_\lambda^\vee(\tau,z):=k_\lambda(-\tau,-z)$,
$$
\|k_\lambda*|b_{\mathfrak s}|\|_\infty
+\|k_\lambda^\vee*|b_{\mathfrak s}|\|_\infty
\le C\lambda^{-\frac1p(1-\frac1\alpha)}m_b.
$$
\end{claim}

\textit{Proof of Claim \ref{cl_aux}.}
Since $p>\theta/\alpha$, using polar coordinates yields
\begin{align*}
\|k_\lambda\|_1
&=C\left(\int_0^\infty e^{-\lambda\tau}\tau^{\frac1p(1-\frac1\alpha)}\,d\tau\right)
\left(\int_{\{|z|\ge1/2\}}|z|^{-d-\alpha+\theta/p}\,dz\right) \le C\lambda^{-\frac1p(1-\frac1\alpha)-1}.
\end{align*}
For the second estimate, we first note that, since 
$[0,\infty[ \ni a \mapsto a^{\frac1p(1-\frac1\alpha)} e^{-a/2}$ is a bounded function, we have
\begin{align*}
e^{-\lambda\tau}\tau^{\frac1p(1-\frac1\alpha)}
&=\lambda^{-\frac1p(1-\frac1\alpha)}(\lambda\tau)^{\frac1p(1-\frac1\alpha)}
 e^{-\lambda\tau/2}e^{-\lambda\tau/2}\\
&\le C\lambda^{-\frac1p(1-\frac1\alpha)}e^{-\tau/2},
\qquad \tau>0,\quad\lambda\ge1.
\end{align*}
Second, we note that 
$$
\sum_{j \in \mathbb Z, \ell \in \mathbb Z^d}\sup_{[j,j+1]\times Q_1(\ell)}\mathbf1_{\{\tau>0,\ |z|\ge1/2\}}
e^{-\tau/2}|z|^{-d-\alpha+\theta/p}<\infty.
$$
Indeed, the time factor decays exponentially, while the spatial factor is bounded on $|z|\ge1/2$, and $d+\alpha-\theta/p>d$. Armed with these two elementary observations, we estimate
\begin{align*}
(k_\lambda*|b_{\mathfrak s}|)(t,y)
&\le\sum_{j\in\mathbb Z,\,\ell\in\mathbb Z^d}
\left(\sup_{(\tau,z)\in[j,j+1]\times Q_1(\ell)}k_\lambda(\tau,z)\right)
\int_j^{j+1}\!\int_{Q_1(\ell)}
 |b_{\mathfrak s}(t-\tau,y-z)|\,dzd\tau\\
&\le m_b\sum_{j\in\mathbb Z,\,\ell\in\mathbb Z^d}
\sup_{(\tau,z)\in[j,j+1]\times Q_1(\ell)}k_\lambda(\tau,z)
\le C\lambda^{-\frac1p(1-\frac1\alpha)}m_b.
\end{align*}
For $k_\lambda^\vee$, the same argument uses
$|b_{\mathfrak s}(t+\tau,y+z)|$ and the boxes
$[t+j,t+j+1]\times Q_1(y+\ell)$. Both bounds hold uniformly
in $(t,y)$, so we can take the supremum. This ends the proof of Claim 1.
\medskip

Before we write down the proof of \eqref{est_R}, let us make two more observations:

1.~H\"older's inequality with respect to the measure
$k_\lambda(\tau,z)\,dzd\tau$ gives the pointwise estimate
\begin{align*}
&(k_\lambda*(\rho_{x_0}^{1/p}|f|))^p(t,y)\\
&\quad\le\left(\int_{\mathbb R^{1+d}}k_\lambda(\tau,z)\,dzd\tau\right)^{p-1}
\int_{\mathbb R^{1+d}}k_\lambda(\tau,z)\rho_{x_0}(y-z)|f(t-\tau,y-z)|^p\,dzd\tau\\
&\quad=\|k_\lambda\|_1^{p-1}(k_\lambda*(\rho_{x_0}|f|^p))(t,y).
\end{align*}

2.~We have
\begin{align*}
&\int_{\mathbb R^{1+d}}|b_{\mathfrak s}(t,y)|
 (k_\lambda*(\rho_{x_0}|f|^p))(t,y)\,dydt\\
&\quad=\int_{\mathbb R^{1+d}}\rho_{x_0}(z)|f(s,z)|^p
\left(\int_{\mathbb R^{1+d}}
 k_\lambda(t-s,y-z)|b_{\mathfrak s}(t,y)|\,dydt\right)\,dzds.
\end{align*}
The inner integral is $(k_\lambda^\vee*|b_{\mathfrak s}|)(s,z)$ by
the definition of the reflected kernel.

Armed with the last two observations and Claim \ref{cl_aux}, we prove \eqref{est_R}. The cubes $Q_1(x)$, $x\in\mathbb Z^d$, partition $\mathbb R^d$ up to sets of measure zero. Therefore,
\begin{align*}
\sum_{x\in\mathbb Z^d}\rho_{x_0}(x)R(x)
&\le C\int_{\mathbb R^{1+d}}
 |b_{\mathfrak s}(t,y)|(k_\lambda*(\rho_{x_0}^{1/p}|f|))^p(t,y)\,dydt\\
&\le C\|k_\lambda\|_1^{p-1}
 \int_{\mathbb R^{1+d}}|b_{\mathfrak s}(t,y)|
 (k_\lambda*(\rho_{x_0}|f|^p))(t,y)\,dydt\\
&=C\|k_\lambda\|_1^{p-1}
 \int_{\mathbb R^{1+d}}\rho_{x_0}(y)|f(t,y)|^p
 (k_\lambda^\vee*|b_{\mathfrak s}|)(t,y)\,dydt\\
 & (\text{we apply Claim \ref{cl_aux}}) \\
&\le C\lambda^{-p+1/\alpha}m_b\|\rho_{x_0}^{1/p}f\|_p^p \\
&\le C A_{\lambda,p}^p M^{\alpha-1}\|\rho_{x_0}^{1/p}f\|_p^p.
\end{align*}
Estimate \eqref{est_R} follows.

\medskip

(\textit{ii}) Upon applying the $L^p$-$L^{p'}$ duality, the result essentially follows from (\textit{i}), although now we will have $\rho_{x_0}$ to a negative power outside of the operator. In detail,
\begin{align*}
&\|(\lambda\pm\partial+A)^{(-1+\frac1\alpha)\frac1{p'}}
 |b_{\mathfrak s}|^{\frac1{p'}}\|_{L^p_{\rho_{x_0}}\rightarrow L^p_{\rho_{x_0}}}\\
&\quad=
\|\rho_{x_0}^{\frac1p}(\lambda\pm\partial+A)^{(-1+\frac1\alpha)\frac1{p'}}
 |b_{\mathfrak s}|^{\frac1{p'}}\rho_{x_0}^{-\frac1p}\|_{L^p\rightarrow L^p}\\
&\quad=
\|\rho_{x_0}^{-\frac1p}|b_{\mathfrak s}|^{\frac1{p'}}
 (\lambda\mp\partial+A)^{(-1+\frac1\alpha)\frac1{p'}}
 \rho_{x_0}^{\frac1p}\|_{L^{p'}\rightarrow L^{p'}}.
\end{align*}
Put $r:=p'$, so that $r'=p$.
We thus need to estimate
$$
\|\rho_{x_0}^{-1/r'}|b_{\mathfrak s}|^{1/r}
 (\lambda\mp\partial+A)^{(-1+\frac1\alpha)/r}
 \rho_{x_0}^{1/r'}\|_{L^r\rightarrow L^r}.
$$
Repeat the cube decomposition and Step 1 with $r$ in place of $p$
and $\rho_{x_0}^{-r/r'}$ in place of $\rho_{x_0}$.
The local comparison of weights and the bounded overlap used in Step 1 apply to $\rho_{x_0}^{-r/r'}$ as well, giving the corresponding estimate for the near term.

For the far part the relevant ratio is now
$$
\frac{\rho_{x_0}(y)^{-1/r'}}{\rho_{x_0}(z)^{-1/r'}}
\le C(1+|y-z|)^{\theta/r'}
=C(1+|y-z|)^{\theta/p}.
$$
In particular, its exponent is still $\theta/p<\alpha$.

Use the kernel in Step 2 with time factor $\tau^{\frac1r(1-\frac1\alpha)}$
and with the same spatial exponent $-d-\alpha+\theta/p$, and apply the
argument to $\rho_{x_0}^{-1/r'}f$.
The two kernel bounds proved there give, in place of \eqref{est_R},
\begin{align*}
\sum_{x\in\mathbb Z^d}\rho_{x_0}(x)^{-r/r'}R(x)
&\le C\|k_\lambda\|_1^{r-1}
 \|k_\lambda^\vee*|b_{\mathfrak s}|\|_\infty\|\rho_{x_0}^{-1/r'}f\|_r^r\\
&\le C\lambda^{-r+1/\alpha}m_b\|\rho_{x_0}^{-1/r'}f\|_r^r\\
&\le C A_{\lambda,r}^r M^{\alpha-1}\|\rho_{x_0}^{-1/r'}f\|_r^r.
\end{align*}
Here $R(x)$ denotes the far term in the decomposition with exponent $r$.
Combining this estimate with the local term and taking $r$th roots gives
\begin{align*}
&\|(\lambda\pm\partial+A)^{(-1+\frac1\alpha)\frac1{p'}}
 |b_{\mathfrak s}|^{1/p'}\|_{L^p_{\rho_{x_0}}\rightarrow L^p_{\rho_{x_0}}}\\
&\quad\le K_1\sup_{x\in\mathbb R^d}
 \|\mathbf1_{Q_2(x)}|b_{\mathfrak s}|^{1/(\alpha-1)}\|_{E_{1+\epsilon}}^{(\alpha-1)/p'}
 +K_2 A_{\lambda,p'}M^{(\alpha-1)/p'}.
\end{align*}
This proves \eqref{adams_est2_weight} for $b_{\mathfrak s}$ under the same condition $p>\theta/\alpha$ as in (\textit{i}).

\medskip

\textit{Proof for the bounded part $b_{\mathfrak b}$.} The proof is standard and is included below for the reader's convenience.
It gives the remaining terms in (\textit{i}), (\textit{ii}) with constant $K_3$ independent of $x_0$. Namely, let us show that, provided $\theta>d$, for every $\frac{\theta}{\alpha}<p<\infty$, for all $\lambda>1$ and $x_0\in\mathbb R^d$, 
\begin{equation}
\label{app_est1}
\| \rho_{x_0}^{\frac{1}{p}}(\lambda + \partial_t+
(-\Delta)^\frac{\alpha}{2}
)^{-\frac{1}{p}(1-\frac{1}{\alpha})}\rho_{x_0}^{-\frac{1}{p}}g\|_{L^p(\mathbb R^{1+d})} \leq C\lambda^{-\frac{1}{p}(1-\frac{1}{\alpha})} \|g\|_{L^p(\mathbb R^{1+d})},
\end{equation}
\begin{equation}
\label{app_est2}
\| \rho_{x_0}^{-\frac{1}{p}}(\lambda + \partial_t+
(-\Delta)^\frac{\alpha}{2}
)^{-\frac{1}{p'}(1-\frac{1}{\alpha})}\rho_{x_0}^{\frac{1}{p}}g\|_{L^p(\mathbb R^{1+d})} \leq C\lambda^{-\frac{1}{p'}(1-\frac{1}{\alpha})} \|g\|_{L^p(\mathbb R^{1+d})},
\end{equation}
where $C=C(d,\alpha,p,\theta,\kappa)$ is independent of
$x_0,\lambda$.

Fix $x_0\in\mathbb R^d$. Let us prove \eqref{app_est1}. Write for brevity $\beta:=\frac{1}{p}(1-\frac{1}{\alpha})$.  We have
\begin{align*}
\| \rho_{x_0}^{\frac{1}{p}}(\lambda + \partial+A)^{-\beta}\rho_{x_0}^{-\frac{1}{p}}g\|_{L^p(\mathbb R^{1+d})} & = \frac1{\Gamma(\beta)}\| \rho_{x_0}^{\frac{1}{p}}\int_0^\infty e^{-\lambda s}s^{\beta-1}e^{-sA}\rho_{x_0}^{-\frac{1}{p}}g(t-s,\cdot)ds\|_{L^p(\mathbb R^{1+d})}\\
& \leq \frac1{\Gamma(\beta)}\| \int_0^\infty e^{-\lambda s}s^{\beta-1}\|\rho_{x_0}^{\frac{1}{p}}e^{-sA}\rho_{x_0}^{-\frac{1}{p}}g(t-s,\cdot)\|_{L^p(\mathbb R^d)}ds\|_{L^p(\mathbb R,dt)}.
\end{align*}
For every $s \geq 0$ (denoting by $p_s(z)$ the integral kernel of $e^{-sA}$),
\begin{align}
\|\rho_{x_0}^{\frac{1}{p}}e^{-sA}\rho_{x_0}^{-\frac{1}{p}}g(t-s,\cdot)\|_{L^p(\mathbb R^d)} & =\big\langle \rho_{x_0}(x) \big|\langle p_s(z)\rho_{x_0}^{-\frac{1}{p}}(x-z)g(t-s,x-z)\rangle_z\big|^p\big\rangle^{\frac{1}{p}}_x \notag \\
& (\text{use $\rho_{x_0}^{-\frac{1}{p}}(x-z) \leq c\rho_{x_0}^{-\frac{1}{p}}(x)\rho^{-\frac{1}{p}}(z)$})\label{line1} \\
& \leq c\big\langle  \langle p_s(z)\rho^{-\frac{1}{p}}(z)|g(t-s,x-z)|\rangle_z^p\big\rangle^{\frac{1}{p}}_x\notag \\
& (\text{apply Young's inequality}) \notag  \\
& \leq c\|p_s(\cdot) (1+\kappa 
|\cdot|^2
)^{\frac{\theta}{2p}}\|_{L^1(\mathbb R^d)}\|g(t-s,\cdot)\|_{L^p(\mathbb R^d)} \qquad (\text{we use $\frac{\theta}{p}<\alpha$}) \notag \\
& \leq C (1+
s^{\frac{\theta}{\alpha p}
}
)\|g(t-s,\cdot)\|_{L^p(\mathbb R^d)}. \notag 
\end{align}
We now return to the first chain of inequalities, substitute there the last estimate, and apply Young's inequality in the time variable ($1+\frac{1}{p}=\frac{1}{1}+\frac{1}{p}$):
\begin{align*}
\| \rho_{x_0}^{\frac{1}{p}}(\lambda + \partial_t+A)^{-\frac{1}{p}(1-\frac{1}{\alpha})}\rho_{x_0}^{-\frac{1}{p}}g\|_{L^p(\mathbb R^{1+d})} & \leq C \int_0^\infty e^{-\lambda s}s^{\beta-1}(1+
s^{\frac{\theta}{\alpha p}
}
)ds\,\|g\|_{L^p(\mathbb R^{1+d})} \\
& \leq C_1 \lambda^{-\beta}\|g\|_{L^p(\mathbb R^{1+d})}, \quad \lambda>1.
\end{align*}
This ends the proof of \eqref{app_est1}.

The proof of \eqref{app_est2} is the same, except that in the
counterpart of line \eqref{line1} we use
$$
\rho_{x_0}^{\frac1p}(x-z)
\le c\rho_{x_0}^{\frac1p}(x)\rho^{-\frac1p}(z).
$$

\qed

\medskip

Proposition \ref{prop1_weight} imposes a more  restrictive bound on $p$ from below, due to the worse decay properties of the stable heat kernel at infinity compared to the stable Bessel potential.

\bigskip

\section{Proof of Theorem \ref{thm1}}

\subsection{Preliminaries}
Put for brevity $A:=(-\Delta)^{\frac{\alpha}{2}}$, $\partial:=\partial_t.$
The proof of Theorem \ref{thm1} requires regularity theory of the inhomogeneous Kolmogorov backward equation (here, for the ease of stating the results, with time reversed)
\begin{equation}
\label{eq1}
\lambda u + \partial u + Au + b \cdot \nabla u=f \quad \text{ on } \mathbb{R}^{1+d}.
\end{equation}
The approximating bounded vector fields $\{b_n\}$ are defined by \eqref{b_n}.

\subsubsection{ The dual variant of Theorem \ref{thm2}}
\label{dualsubsect}
Let $1<p<\infty$.
Then, for the same constants $c_{d,\alpha,p,\varepsilon}$ and $\lambda_{d,\alpha,p,\varepsilon}$ as in 
Theorem \ref{thm2}, for a right-hand side $$f \in \mathbb{W}_\alpha^{(-1+\frac{1}{\alpha})\frac{2}{p'},p}(\mathbb{R}^{1+d})$$ of \eqref{eq1} and a sequence of functions
$\{f_n\} \subset L^\infty(\mathbb{R}^{1+d})$ such that $$
f_n \rightarrow f \quad \text{ in } \mathbb{W}_\alpha^{(-1+\frac{1}{\alpha})\frac{2}{p'},p}
$$
solutions $u_n$ of the approximating Kolmogorov backward equations
\begin{equation*}
\lambda u_n + \partial u_n + A u_n + b_n \cdot \nabla u_n=f_n, \qquad \lambda > \lambda_{d,\alpha,p,\varepsilon}
\end{equation*}
converge as $n \rightarrow \infty$ in $\mathbb{W}_\alpha^{\frac{2}{\alpha}+\frac{\alpha-1}{\alpha}\frac{2}{p},p}$ to the function
\begin{equation}
\label{u_repr0}
u=(\lambda+\partial  +A)^{-\frac{1}{\alpha}+(-1+\frac{1}{\alpha})\frac{1}{p}}(1+Q_pR_p)^{-1} (\lambda+\partial  +A)^{(-1+\frac{1}{\alpha})\frac{1}{p'}} f
\end{equation}
where
\begin{equation}
R_p:=b^{\frac{1}{p}}\cdot \nabla (\lambda+\partial  +A )^{-\frac{1}{\alpha}+(-1+\frac{1}{\alpha})\frac{1}{p}}, 
\end{equation}
\begin{equation*}
Q_p:=(\lambda+\partial  + A)^{(-1+\frac{1}{\alpha})\frac{1}{p'}}|b|^{\frac{1}{p'}},
\end{equation*}
are bounded operators on $L^p(\mathbb{R}^{1+d})$, and
\begin{equation*}
\|R_p\|_{p \rightarrow p}\|Q_p\|_{p \rightarrow p} <1
\end{equation*}
so the geometric series in \eqref{u_repr0} converges in $L^p(\mathbb{R}^{1+d})$.
In particular, it follows that $$u \in \mathbb{W}_\alpha^{\frac{2}{\alpha}+\frac{\alpha-1}{\alpha}\frac{2}{p},p},$$
and if $p>d+1$
then, by the Sobolev embedding \eqref{sob_emb0}, \eqref{sob_emb} $u$ is bounded and continuous and the convergence of $u_n$ to $u$ is uniform on $\mathbb{R}^{1+d}$.

There is also the following alternative representation
\begin{align}
& u =  (\lambda+\partial  + A)^{-1}f \notag \\
& - (\lambda+\partial  + A)^{-\frac{1}{\alpha}+(-1+\frac{1}{\alpha})\frac{1}{p}}Q_p (1+R_pQ_p)^{-1}R_p (\lambda+\partial  + A)^{(-1+\frac{1}{\alpha})\frac{1}{p'}} f. \label{u_repr}
\end{align}

\subsubsection{Weighted counterpart of \ref{dualsubsect}} 
\label{weight_subsect}
The results in \ref{dualsubsect} are obtained by following
the argument in \cite[Section 5]{D1}, with the Adams-type estimate
Proposition~\ref{prop1} as the key point.
To obtain the weighted counterpart of \ref{dualsubsect},
we use the weighted Adams-type estimates in
Proposition~\ref{prop1_weight} to adapt the truncation and
mollification argument in \cite[Section 5]{D1} and obtain, with the dominated convergence theorem, 
for each fixed $x\in\mathbb R^d$,
$$
Q_p(b_n)\longrightarrow Q_p(b),
\qquad
R_p(b_n)\longrightarrow R_p(b)
\quad\text{strongly in }
L^p_{\rho_x}(\mathbb R^{1+d}).
$$
Together with the uniform contraction bound, this yields
convergence of both fractional representations for the
Duhamel series from \ref{dualsubsect} in
$L^p_{\rho_x}(\mathbb R^{1+d})$.

The bounds in that proposition are uniform in the weight center,
so $\kappa$ and a sufficiently large $\lambda$ may be chosen
independently of $x$. After translating $B_1(x)$ to $B_1(0)$,
the local Sobolev embedding yields, for $p>d+1$,
\begin{equation}
\label{weighted_emb}
 \sup_{t\in\mathbb R,\,y\in B_1(x)}|u(t,y)|
 \le K\|f\|_{L^p_{\rho_x}(\mathbb R^{1+d})},
 \qquad x\in\mathbb R^d,
\end{equation}
with $K$ independent of $x$. This relaxes the condition on the
behaviour of the right-hand side $f$ at infinity compared to
\ref{dualsubsect}.

Proposition~\ref{prop1_weight}(\textit{ii}) also bounds the source factor
$Q_p(\mathsf f):=(\lambda\pm\partial_t+A)^{(-1+\frac1\alpha)\frac1{p'}}|\mathsf f|^{1/p'}$
for the test field $\mathsf f$.

\subsubsection{A priori Krylov bound}
\label{smooth_krylov_subsect}
Recall that $\{b_n\}$ denotes the bounded smooth approximation of $b$
defined by \eqref{b_n}. We define $\{\mathsf f_n\}$ in the same way.
For each $n$, let $\mathbb P^n_{s,x}$ be the law of the solution to
\eqref{sde1} with the smooth drift $b_n$, starting from $x$ at time $s$,
and let $\mathbb E^n_{s,x}$ denote the corresponding expectation.

\begin{claim} Fix $p>d+1$. Then
\begin{equation}
\label{apr_bd}
 \mathbb E^n_{s,x}\int_s^T
       |\mathsf f_n(r,\omega_r)h(r,\omega_r)|\,dr
 \le K\||\mathsf f_n|^{1/p}h\|_{L^p_{\rho_x}([s,T] \times \mathbb R^d)},
\end{equation}
for all $h\in C_b(\mathbb R^{1+d})$, with $K$ independent of
$n,s,x,h$. 
\end{claim}

\begin{proof}
Since both sides of the inequality depend on $|h|$,
we may assume below that $h\ge0$.
We have
\begin{equation*}
\mathbb E^n_{s,x}\int_s^T |\mathsf{f}_n(r,\omega_r)h(r,\omega_r)|dr = u_n(s,x),
\end{equation*}
where $u_n=u_n(s,x)$ is the solution to the terminal-value problem in $s \leq T$
$$
(-\partial + A + b_n \cdot \nabla)u_n=|\mathsf{f}_n h|, \quad u_n|_{s=T}=0,
$$
which we put in the form
$$
(-\partial + A + b_n \cdot \nabla)u_n=\mathbf{1}_{s \leq T}|\mathsf{f}_n|h \quad \text{ on } \mathbb R^{1+d}
$$
and, furthermore, setting $$e_\lambda(s)=e^{\lambda s},$$
we write
$$
(\lambda-\partial + A + b_n \cdot \nabla)(e_\lambda u_n)=e_\lambda(s)\mathbf{1}_{s \leq T}|\mathsf{f}_n|h \quad \text{ on } \mathbb R^{1+d}.
$$
For a fixed initial time $s$, the value $u_n(s,x)$ depends only
on the source at times in $[s,T]$. We can therefore restrict the
right-hand side to that interval when applying the space-time
estimates. Thus, it remains to apply \eqref{weighted_emb}.
\end{proof}

In the right-hand side of \eqref{apr_bd} we can of course have $\mathsf f_m$ for $m \neq n$. (In fact, below we will be applying this estimate to $f_m=b_n-b_m$.)

\subsubsection{Feller propagator}
\label{feller_subsect}
Fix some $r \in \mathbb R$. Let $v_n$ denote the solution of the initial-value problem 
\begin{equation}
\label{cauchy}
\left\{
\begin{array}{l}
(\lambda + \partial+ A + b_n \cdot \nabla)v_n  =0, \quad (t,x) \in ]r,\infty[ \times \mathbb R^d, \\[2mm]
 v_n(r,\cdot)=g \in C_\infty(\mathbb R^d),
\end{array}
\right.
\end{equation}
where $b_n$ is defined by \eqref{b_n}.

By a classical result, for each $n \geq 1$, the operators $$U_n^{t,r}g:=v_n(t), \quad -\infty<r \leq t<\infty$$ constitute a (forward) Feller propagator on $C_\infty(\mathbb R^d)$. 
We have conservation of probability:
$$
\int_{\mathbb R^d} e^{\lambda (t-r)}U_n^{t,r}(x,y)dy=1, \quad x \in \mathbb R^d. 
$$

Let $\delta_r(t)$ denote the delta-function (in time) concentrated at $t=r$. Put
\begin{equation}
\label{delta_def}
(\lambda+\partial+ A)^{-\frac{\gamma}{2}}\delta_{r}\,g(t,x):=\mathbf{1}_{t > r}e^{-\lambda (t-r)}\int_{\mathbb R^d} q_{\gamma}(t-r,x-y) g(y)dy,
\end{equation}
i.e.\,by putting the delta-function in time in the right-hand side, the action of the parabolic operator becomes the action of the semigroup of the fractional Laplacian (times the exponential function in time) on $g=g(y)$.

\medskip

Assume now that we are in the setting of the previous subsection \ref{dualsubsect} with
$p>d+1$.
Then the limit
$$
U^{t,r}:=s\mbox{-}C_\infty(\mathbb R^d)\mbox{-}\lim_n U_n^{t,r}  \quad \text{locally uniformly in $(r,t) \in \mathbb R^2$, $r \leq t$}
$$
exists and determines a Feller propagator on $C_\infty(\mathbb R^d)$. Furthermore, for every initial function $g \in C_\infty(\mathbb R^d) \cap W^{1,p}(\mathbb R^d)$, the function $v(t):=U^{t,r}g$ has the representation
\begin{equation}
\label{v_repr}
v=(\lambda+\partial+A)^{-1}\delta_{r}g - (\lambda+\partial+A)^{-\frac{1}{\alpha}+(-1+\frac{1}{\alpha})\frac{1}{p}}Q_p (1+R_pQ_p)^{-1}N_p \cdot  S_p g,
\end{equation}
where
$$
N_p:=b^{\frac{1}{p}} (\lambda+\partial+A)^{(-1+\frac{1}{\alpha})\frac{1}{p}} \in \mathcal B(L^p(\mathbb{R}^{1+d}),[L^p(\mathbb{R}^{1+d})]^d)
$$
and
$$
S_p g:=\nabla (\lambda+\partial+A)^{-\frac{1}{p'}-\frac{1}{\alpha p}}\delta_{r}g, \qquad
\|S_pg\|_{L^p(\mathbb{R}^{1+d})} \leq C_{d,\alpha,p} \|\nabla g\|_{L^p(\mathbb R^d)}.
$$

\medskip

The results in \ref{dualsubsect} and \ref{feller_subsect} were proved earlier in \cite{D1}.

\medskip

\subsubsection{A $\rho_x$-weighted integral bound on $|b_n|$ and $|b|$}
\label{weighted_mass_subsect}
The drifts $b_n$ defined by \eqref{b_n}
satisfy the following bound:

\begin{claim}
\begin{equation}
\label{weighted_drift_mass}
 \sup_n\sup_{s\in\mathbb R,\,x\in\mathbb R^d}
 \int_s^{s+h}\int_{\mathbb R^d}|b_n(t,y)|\rho_x(y)\,dy\,dt
 \le C h^{1/\alpha},\qquad 0<h\le 1,
\end{equation}
where the constant $C$ is independent of  $n$, $s$, $x$.
The same estimate holds with $b$ in place of $b_n$.
\end{claim}

This integral appears in the right-hand side of the Krylov bound with $\mathsf{f}_n=b_n$.

\begin{proof}
1.~Put for brevity $M=\||b_{\mathfrak s}|^{1/(\alpha-1)}\|_{E_{1+\epsilon}}$. H\"older's inequality yields
\begin{align*}
 \int_{C_r^\alpha(s,z)}|b_{\mathfrak s}|
 &\le |C_r^\alpha(s,z)|^{1-\frac{\alpha-1}{1+\epsilon}}
 \left(\int_{C_r^\alpha(s,z)}
       |b_{\mathfrak s}|^{\frac{1+\epsilon}{\alpha-1}}
       \right)^{\frac{\alpha-1}{1+\epsilon}}\\
  & \text{(use $ \big(r^{-d-\alpha}
 \int_{C_r^\alpha(s,z)}
 |b_{\mathfrak s}|^{\frac{1+\epsilon}{\alpha-1}}
 \big)^{\frac1{1+\epsilon}}\le\frac Mr.$)} \\	   
 &\le C r^{d+\alpha}\left(\frac Mr\right)^{\alpha-1}
 = C M^{\alpha-1}r^{d+1}, \quad  0<r\le1.
\end{align*}
The bounded part satisfies
$$
 \int_{C_r^\alpha(s,z)}|b_{\mathfrak b}|
 \le C\|b_{\mathfrak b}\|_\infty r^{d+\alpha}
 \le C\|b_{\mathfrak b}\|_\infty r^{d+1}.
$$
The same estimates are valid for $\{b_n\}$.

2.~We next pass from these local estimates to the weighted integral in \eqref{weighted_drift_mass}. 

First, let us obtain the corresponding local estimates on cubes. Fix $0<h\le1$ and set $r=h^{1/\alpha}$.
We partition $\mathbb R^d$ into the union of disjoint cubes $Q$ of side $r/\sqrt d$, excluding a measure zero set, so that each cube has diameter $r$ and is contained in a ball of radius $r$. Therefore, the preceding cylinder estimate
yields 
$$
\int_s^{s+h}\int_Q|b_n| \leq C(M^{\alpha-1}+\|b_{\mathfrak b}\|_\infty)r^{d+1}.
$$
Now, we are going to use that $\rho_x$ is
comparable at any two points of the same cube, uniformly in $x$.
Indeed, for $y,z\in Q$, we have
$$
 \sqrt{1+\kappa|z-x|^2}
 \le \sqrt{1+\kappa|y-x|^2}+\sqrt\kappa|z-y|
 \le (1+\sqrt\kappa)\sqrt{1+\kappa|y-x|^2}
$$
using $|z-y|\le r\le1$ and $\sqrt{1+\kappa|y-x|^2}\ge1$  (cf.\,the proof of Proposition \ref{prop1_weight}).
Hence
$$
 \rho_x(y)\le(1+\sqrt\kappa)^\theta\rho_x(z),
 \qquad y,z\in Q.
$$
Next, for each fixed $y\in Q$, we integrate the previous inequality with respect to
$z$ over $Q$, divide by $|Q|$,  take the supremum over $y$ and use $|Q|=(r/\sqrt d)^d$, thereby arriving at
$$
 \sup_{y\in Q}\rho_x(y)
 \le C r^{-d}\int_Q\rho_x(z)\,dz.
$$
It follows that
$$
 \sum_Q\sup_{y\in Q}\rho_x(y)
 \le C r^{-d}\int_{\mathbb R^d}\rho_x(z)\,dz
 = C r^{-d}\|\rho\|_1<\infty.
$$
Finally, we estimate
\begin{align*}
 \int_s^{s+h}\int_{\mathbb R^d}|b_n(t,y)|\rho_x(y)\,dy\,dt
 &\le \sum_Q\left(\sup_{y\in Q}\rho_x(y)\right)
       \int_s^{s+h}\int_Q|b_n(t,y)|\,dy\,dt\\
 &\le C\big(M^{\alpha-1}+\|b_{\mathfrak b}\|_\infty\big)
       r^{d+1}\sum_Q\sup_{y\in Q}\rho_x(y)\\
 &\le C\big(M^{\alpha-1}+\|b_{\mathfrak b}\|_\infty\big)
       \|\rho\|_1 r.
\end{align*}
Since $r=h^{1/\alpha}$, this yields \eqref{weighted_drift_mass}.
\end{proof}

\subsection{Proof of Theorem \ref{thm1}} \

\medskip

\noindent \textit{Proof of} (\textit{ii}). The forward Feller propagator of \ref{feller_subsect} is our point of
departure. To construct the backward Feller propagator for the original drift
$b$, we apply \ref{feller_subsect} to the drifts
\[
 \widehat b(t,x):=b(T-t,x),\qquad
 \widehat b_n(t,x):=b_n(T-t,x),
\]
where the reflection in time preserves the Morrey class, and so \ref{feller_subsect} applies. Let us denote the
resulting forward propagators by $\widehat U$ and $\widehat U_n$, and
define, for $0\le s\le t\le T$,
\[
 P^{s,t}:=e^{\lambda(t-s)}\widehat U^{T-s,T-t},\qquad
 P_n^{s,t}:=e^{\lambda(t-s)}\widehat U_n^{T-s,T-t}.
\]
Here, including the exponential factor removes the constant term $\lambda$ from the equation.
In particular, we have
\[
 \sup_{0\le s\le t\le T}\|P_n^{s,t}g-P^{s,t}g\|_\infty\rightarrow0,
 \qquad g\in C_\infty(\mathbb R^d).
\]

\medskip

\noindent\emph{Preservation of probability by the limiting Feller propagator $P$.}
Let us fix a radial function
$\chi\in C_c^\infty(\mathbb R^d)$, $0\le\chi\le1$, with
$\chi=1$ on $B_1$ and $\chi=0$ outside $B_2$, and put
$\chi_R(y):=\chi(y/R)$, $R\ge1$. Then $\chi_R\uparrow1$,
and (recall that $A=(-\Delta)^{\frac{\alpha}{2}}$)
\[
 \|\nabla\chi_R\|_\infty\le C R^{-1},\qquad
 \|A\chi_R\|_\infty
 =R^{-\alpha}\|A\chi\|_\infty\le C R^{-\alpha}.
\]
Let $\mathbb E^n_{s,x}$ denote expectation for the SDE with smooth drift $b_n$. The corresponding backward Kolmogorov operator is
$-A-b_n(r,\cdot)\cdot\nabla$. Therefore, the It\^o's formula applied to $\chi_R$ takes the form
\[
 \chi_R(x)-P_n^{s,t}\chi_R(x)
 =\int_s^t P_n^{s,r}
       \big(A\chi_R+b_n(r,\cdot)\cdot\nabla\chi_R\big)(x)\,dr.
\]
Since $P_n$ is a Feller propagator, and therefore a positivity preserving $L^\infty$ contraction, we obtain
\[
 |\chi_R(x)-P_n^{s,t}\chi_R(x)|
 \le CT R^{-\alpha}
       +\mathbb E^n_{s,x}\int_s^t
          |b_n(r,\omega_r)|\,|\nabla\chi_R(\omega_r)|\,dr.
\]
To estimate the drift term, we apply \eqref{apr_bd}, proved in
\ref{smooth_krylov_subsect}, with $\mathsf f_n=b_n$ and
$h(r,y)=|\nabla\chi_R(y)|$, with terminal time $t$. It yields
\begin{align*}
 \mathbb E^n_{s,x}\int_s^t
       |b_n(r,\omega_r)|\,|\nabla\chi_R(\omega_r)|\,dr
 &\le K\left(\int_s^t\int_{\mathbb R^d}
       |b_n(r,y)|\,|\nabla\chi_R(y)|^p\rho_x(y)\,dy\,dr
       \right)^{1/p}\\
 &\le\frac{CK}{R}\left(\int_s^t\int_{\mathbb R^d}
       |b_n(r,y)|\rho_x(y)\,dy\,dr\right)^{1/p}
 \le\frac{C_T}{R},
\end{align*}
where, in order to estimate the last integral, we have \eqref{weighted_drift_mass}
splitting $[0,T]$ into finitely many intervals.
Therefore,
\[
 \sup_{0\le s\le t\le T}
 \|\chi_R-P_n^{s,t}\chi_R\|_\infty
 \le C_T\big(R^{-\alpha}+R^{-1}\big).
\]
For each fixed $R$, $\chi_R\in C_\infty(\mathbb R^d)$, so strong
convergence of the propagators, see above, gives the same inequality for $P^{s,t}$.

To establish the preservation of probability we first make sense of the expression $P^{s,t} 1$ since $1 \not\in C_\infty(\mathbb R^d)$. Fix $s,t,x$. Since $P^{s,t}$ is a Feller propagator, and so in particular a positivity preserving contraction on
$C_\infty(\mathbb R^d)$, the map $g\mapsto P^{s,t}g(x)$ is a
positive bounded linear functional of norm at most one.
By the Riesz representation theorem, there is a unique finite
positive Radon measure $p^{s,t}(x,dy)$ such that
\[
 P^{s,t}g(x)=\int_{\mathbb R^d}g(y)\,p^{s,t}(x,dy),
 \qquad g\in C_\infty(\mathbb R^d),
 \qquad p^{s,t}(x,\mathbb R^d)\le1.
\]
For $R>|x|$ we have $\chi_R(x)=1$, so
\[
 0\le1-P^{s,t}\chi_R(x)
 \le C_T\big(R^{-\alpha}+R^{-1}\big).
\]
We now apply the monotone convergence theorem with respect to the measure $p^{s,t}(x,dy)$.
Namely, as $\chi_R\uparrow1$ as $R \rightarrow \infty$, we have
\[
 p^{s,t}(x,\mathbb R^d)
 =\lim_{R\rightarrow\infty}\int_{\mathbb R^d}\chi_R(y)\,p^{s,t}(x,dy)
 =\lim_{R\rightarrow\infty}P^{s,t}\chi_R(x)=1,
\]
so we can now write
\[
 P^{s,t}1(x):=\int_{\mathbb R^d} p^{s,t}(x,dy)=1,
\]
i.e.\,we have preservation of probability.

\medskip

Now, by classical theory (see \cite[Theorem~2.22]{GC}), the Feller propagator $P^{s,t}$ yields a c\`adl\`ag strong Markov
family of probability measures $\{\mathbb P_{s,x}\}_{0\le s<T,\,x\in\mathbb R^d}$
such that \eqref{P} holds.

\medskip

\noindent \textit{Proof of} (\textit{i}). We first prove
Krylov bound \eqref{krylov_bd} for $\{\mathbb P_{s,x}\}_{x\in\mathbb R^d,\,s\ge0}$
by passing to the limit in the a priori Krylov bound \eqref{apr_bd}.

Let us first show that the strong convergence of propagators $P_n^{s,t} \rightarrow P^{s,t}$ on $C_\infty$ yields the
convergence of the time integrals needed for this passage to the limit.

First, note that, for fixed $s,t,x$, the convergence
$P_n^{s,t}g(x)\rightarrow P^{s,t}g(x)$ also holds for every
$g\in C_b(\mathbb R^d)$. Indeed, using the fact that $\chi_Rg$ belongs to $C_\infty$, we estimate
\[
 |P_n^{s,t}g(x)-P_n^{s,t}(\chi_Rg)(x)|
 \le\|g\|_\infty\bigl(1-P_n^{s,t}\chi_R(x)\bigr).
\]
The same inequality holds for $P^{s,t}$. For fixed $R$, all terms
involving $\chi_R$ converge as $n\rightarrow\infty$, while the remaining terms tend to zero as $R\rightarrow\infty$ since $P^{s,t}1(x)=1$.

Hence, for any $q\in C_b([s,T] \times \mathbb R^d)$,
Fubini's theorem and the dominated convergence theorem yield
\begin{equation}
\label{occupation_conv}
\begin{aligned}
 \mathbb E^n_{s,x}\int_s^Tq(r,\omega_r)\,dr
 &=\int_s^T P_n^{s,r}q(r,\cdot)(x)\,dr\\
 &\rightarrow\int_s^T P^{s,r}q(r,\cdot)(x)\,dr
 =\mathbb E_{s,x}\int_s^Tq(r,\omega_r)\,dr
\end{aligned}
\end{equation}
where the integrands are bounded by $\|q\|_\infty$.

\medskip

Before passing to the limit in order to arrive at the a posteriori Krylov bound, we need to make sure that the limiting process must spend zero
time in the subsets of space-time of Lebesgue measure zero. 

\medskip

\noindent\emph{Occupation measure of the limiting process.}
Define the occupation
measure of the limiting process by
\[
 \eta_{s,x}(B):=\mathbb E_{s,x}\int_s^T
                   \mathbf1_B(r,\omega_r)\,dr,
 \qquad \text{ for Borel measurable } B\subset[s,T] \times \mathbb R^d.
\]
Taking $|\mathsf f|=1$ in the a priori Krylov estimate and applying
\eqref{occupation_conv} to $|q|$, we obtain
\[
 \left|\int_{[s,T] \times \mathbb R^d} q\,d\eta_{s,x}\right|
 \le\int_{[s,T] \times \mathbb R^d}|q|\,d\eta_{s,x}
 \le K\|q\|_{L^p_{\rho_x}([s,T] \times \mathbb R^d)},
 \qquad q\in C_c([s,T] \times \mathbb R^d).
\]
Therefore, the integration against $\eta_{s,x}$ extends to a bounded
functional on $L^p_{\rho_x}([s,T] \times \mathbb R^d)$.
There exists a unique $k_{s,x}\in L^{p'}_{\rho_x}([s,T] \times \mathbb R^d)$ such that
\[
 \eta_{s,x}(dr\,dy)=k_{s,x}(r,y)\rho_x(y)\,dr\,dy
\]
It follows that, if $B$ has Lebesgue measure zero, then
$\mathbb E_{s,x}\int_s^T\mathbf1_B(r,\omega_r)\,dr=0$.

\medskip

We now fix $m$ and apply \eqref{apr_bd} with the test vector field
$\mathsf f_m$ (the drift remains $b_n$).
Since $|\mathsf f_m h|$ is bounded and continuous,
\eqref{occupation_conv} gives
\[
 \mathbb E_{s,x}\int_s^T|\mathsf f_m h|(r,\omega_r)\,dr
 =\lim_n\mathbb E^n_{s,x}\int_s^T|\mathsf f_m h|(r,\omega_r)\,dr
 \le K\||\mathsf f_m|^{1/p}h\|_{L^p_{\rho_x}([s,T] \times \mathbb R^d)},
\]
where $K$ is independent of $m,s,x$.
It remains to let $m\rightarrow\infty$. Passing to a subsequence, if necessary, we have
$\mathsf f_m\rightarrow\mathsf f$ a.e.\,and in
$L^1_{\rho_x}([s,T] \times \mathbb R^d)$ for each fixed $x$.
The occupation measure ignores the exceptional set on which
pointwise convergence may not hold. Fatou's lemma therefore allows us to
pass to the limit on the left-hand side and gives the sought estimate 
\begin{equation}
\label{kr_bd2}
 \mathbb E_{s,x}\int_s^T|\mathsf f(r,\omega_r)h(r,\omega_r)|\,dr
 \le K\||\mathsf f|^{1/p}h\|_{L^p_{\rho_x}([s,T] \times \mathbb R^d)}.
\end{equation}
The constant is the same for all $s,x$, so we may now take the
supremum over $x$ on both sides. This proves \eqref{krylov_bd}
and its uniform version in Theorem~\ref{thm1}.

\medskip

We now proceed to the proof of the properties (a)-(c) in assertion (\textit{i}).
Without loss of generality, $s=0$; (a) and (b) follow right away from the construction of $\mathbb P_{0,x}$ and from the Krylov-type bound \eqref{kr_bd2} (take $\mathsf{f}=b$, $h=1$). It remains to verify (c), i.e.\,to show that
$$
Z_t(\omega):=\omega_t-x+\int_0^t b(s,\omega_s)ds, \quad 0\le t\le T,
$$
is an $\alpha$-stable process under $\mathbb P_{0,x}$.
Arguing as \cite{P,PP,CW}, we need to show that
\begin{equation}
\label{moment}
\mathbb E_{0,x} \biggl[e^{i \xi \cdot (\omega_t-x+\int_0^t b(s,\omega_s)ds)} \biggr]=e^{-t|\xi|^\alpha}, \quad \forall\,x,\xi \in \mathbb R^d, \; 0\le t\le T,
\end{equation}
where the right-hand side is the characteristic function of the isotropic $\alpha$-stable process. We will use this below to identify $Z$ as an isotropic $\alpha$-stable process, 
	and thus to prove (c). Specifically, the proof of \eqref{moment} yields a more general result: for all $0 \leq r \leq t \leq T$, 
$$
\mathbb E_{r,x}\exp\left(
 i\xi\cdot\left(
 \omega_t-x+\int_r^t b(u,\omega_u)\,du
 \right)
\right)= e^{-(t-r)|\xi|^\alpha}.
$$
Therefore,
$$
\mathbb E_{0,x}\left[
 e^{i\xi\cdot (Z_t-Z_r)}
 \,\middle|\, \mathcal F_r
\right] = e^{-(t-r)|\xi|^\alpha}.
$$
The right-hand side is deterministic, so $Z_t-Z_r$ is independent of $\mathcal F_r$, and has the isotropic $\alpha$-stable law with characteristic
function $e^{-(t-r)|\xi|^\alpha}$. Iterating this identity, we obtain for $0=t_0<t_1<\cdots<t_n\le T$:
$$
\mathbb E_{0,x}\exp\left(
 i\sum_{k=1}^n \xi_k\cdot (Z_{t_k}-Z_{t_{k-1}})
\right)
=
\prod_{k=1}^n
e^{-(t_k-t_{k-1})|\xi_k|^\alpha}.
$$
Thus the increments of $\{Z_t\}$ are independent and stationary. Since  $\{Z_t\}$ has
c\`adl\`ag paths and $Z_0=0$, it is an isotropic
$\alpha$-stable process.

We establish \eqref{moment} by first showing that, for all $r \in [0,t]$, $x \in \mathbb R^d$, $\xi \in \mathbb R^d$, both functions
$$
\nu(r,t,x,\xi):=\mathbb E_{r,x} \biggl[e^{i \xi \cdot (\omega_t+\int_r^t b(s,\omega_s)ds)} \biggr]
$$	
and
$$
\tilde{\nu}(r,t,x,\xi):=e^{i \xi \cdot x-(t-r)|\xi|^\alpha}
$$	
are bounded solutions of the same integral equation, and then showing that this integral equation can have only one bounded solution, so \eqref{moment} follows.

\begin{claim}
\label{cl1}
$\nu(r,t,x,\xi)$ satisfies the integral equation
\begin{align}
\label{eq_nu}
\nu(r,t,x,\xi) = \mathbb{E}_{r,x}[ e^{i \xi \cdot \omega_t}] 
+ i \int_r^t \mathbb{E}_{r,x}[ \xi \cdot b(u,\omega_u)  \nu(u,t,\omega_u,\xi)] \, du.
\end{align}
\end{claim}
\begin{proof}[Proof of Claim \ref{cl1}]
Let $r \leq s \leq t$. We write
$$
\int_r^t \xi \cdot b(u,\omega_u)\, du 
= \int_r^s \xi \cdot b(u,\omega_u)\, du + \int_s^t \xi \cdot b(u,\omega_u)\, du
$$
and obtain, using the Markov property,
\begin{align}
\nu(r,t,x,\xi) 
&= \mathbb{E}_{r,x}\Big[ e^{i \int_r^s \xi \cdot b(u,\omega_u)\,du}
e^{i \int_s^t \xi \cdot b(u,\omega_u)\,du} e^{i \xi \cdot \omega_t} \big]\notag\\
&= \mathbb{E}_{r,x} \Big[ e^{i \int_r^s \xi \cdot b(u,\omega_u)\,du}
\mathbb{E}_{r,x}\Big[e^{i \int_s^t \xi \cdot b(u,\omega_u)\,du} e^{i\xi \cdot \omega_t} \,\Big|\, \mathcal{F}_s \Big] \Big]\notag\\
&= \mathbb{E}_{r,x} \Big[ e^{i \int_r^s \xi \cdot b(u,\omega_u)\,du}
\mathbb{E}_{s,\omega_s}\Big[e^{i \int_s^t \xi \cdot b(u,\omega_u)\,du} e^{i\xi \cdot \omega_t} \Big] \Big]\notag\\
&= \mathbb{E}_{r,x}\Big[e^{i \int_r^s \xi \cdot b(u,\omega_u)\,du}  \nu(s,t,\omega_s,\xi)\Big]\label{1}.
\end{align}
Note that
\begin{align*}
e^{
i\int_r^s \xi \cdot  b(u,\omega_u)\,du}
= 1 + i \int_r^s \xi \cdot b(u,\omega_u)\, 
e^{i\int_u^s \xi \cdot b(l,\omega_l)\,dl}
\, du
\end{align*}
(integrate the derivative of the left-hand side).
Therefore, 
\begin{align*}
&\nu(r,t,x,\xi) \\
&\quad= \mathbb{E}_{r,x}[\, \nu(s,t,\omega_s,\xi)\,] 
+ i \int_r^s \mathbb{E}_{r,x}\Big[ \nu(s,t,\omega_s,\xi) \, \xi \cdot b(u,\omega_u) e^{i  \int_u^s \xi \cdot b(l,\omega_l)\, dl} \Big] du.
\end{align*}
Using the Markov property again and applying \eqref{1} in  the other direction, we have
\begin{align*}
&\mathbb{E}_{r,x}\Big[ \nu(s,t,\omega_s,\xi) \, \xi \cdot b(u,\omega_u) e^{i \int_u^s  \xi \cdot b(l,\omega_l)\, dl} \Big]\\
&\quad =
\mathbb{E}_{r,x}\Big[ \mathbb{E}_{r,x}\Big[ \nu(s,t,\omega_s,\xi) \, \xi \cdot b(u,\omega_u)\, 
e^{i \int_u^s \xi \cdot  b(l,\omega_l)\, dl} \,\Big|\, \mathcal{F}_u \Big]\Big]\\
&\quad = \mathbb{E}_{r,x}\Big[ \xi \cdot b(u,\omega_u) 
\mathbb{E}_{u,\omega_u}\Big[ \nu(s,t,\omega_s,\xi) \,e^{i \int_u^s \xi \cdot  b(l,\omega_l)\, dl}\Big] \,\Big]\\
&\quad =
\mathbb{E}_{r,x}\big[ \xi \cdot b(u,\omega_u) \nu(u,t,\omega_u,\xi) \big].
\end{align*}
Thus,
$$
\nu(r,t,x,\xi) = \mathbb{E}_{r,x}[ \nu(s,t,\omega_s,\xi)] 
+ i \int_r^s \mathbb{E}_{r,x}[ \xi \cdot b(u,\omega_u) \nu (u,t,\omega_u,\xi)] \, du.
$$
Note that
$$
\nu(t,t,x,\xi) = \mathbb{E}_{t,x}[ e^{i \xi \cdot \omega_t}] = e^{i \xi \cdot x}.
$$
Taking $s=t$, we have
$$
\nu(r,t,x,\xi) = \mathbb{E}_{r,x}[ e^{i \xi \cdot \omega_t}] 
+ i \int_r^t \mathbb{E}_{r,x}[ \xi \cdot b(u,\omega_u) \nu(u,t,\omega_u,\xi)] \, du.
$$
This completes the proof of Claim \ref{cl1}.
\end{proof}

\begin{claim}
\label{cl2}	
The function 
$
\tilde{\nu}(r,t,x,\xi):=e^{i \xi \cdot x-(t-r)|\xi|^\alpha}
$ is a bounded solution of \eqref{eq_nu}.		
	
\end{claim}
\begin{proof}

Fix $0\le r<t\le T$, $x\in\mathbb R^d$, and
$\xi\in\mathbb R^d$. We first obtain the required identity for the
smooth drift $b_n$, and then pass to the limit in $n$. For $r\le s\le t$, write
\[
 \tilde{\nu}(s,y) \equiv \tilde\nu(s,t,y,\xi)
       =e^{i\xi\cdot y-(t-s)|\xi|^\alpha}.
\]
Since $A e^{i\xi\cdot y}=|\xi|^\alpha e^{i\xi\cdot y}$, we have
\[
 (\partial_s-A)\tilde{\nu}=0,\qquad \nabla \tilde{\nu}=i\xi \tilde{\nu},
 \qquad \tilde{\nu}(t,y)=e^{i\xi\cdot y},\qquad |\tilde{\nu}|\le1.
\]
Under $\mathbb P^n_{r,x}$ the coordinate process solves
$d\omega_s=-b_n(s,\omega_s)\,ds+dZ^n_s$, where $Z^n$ is the isotropic $\alpha$-stable process of the regularized SDE. The corresponding Kolmogorov operator is
$-A-b_n\cdot\nabla$, so, applying It\^o's formula to the real and
imaginary parts of $\tilde{\nu}$ and taking expectations yields
\begin{equation}
\label{n_id}
 \tilde\nu(r,t,x,\xi)
 =\mathbb E^n_{r,x}e^{i\xi\cdot \omega_t}
  +i\mathbb E^n_{r,x}\int_r^t
       \xi\cdot b_n(s,\omega_s)\,\tilde{\nu}(s,\omega_s)\,ds.
\end{equation}

We now justify the passage to the limit $n \rightarrow \infty$ in each term of
\eqref{n_id}. The convergence of the propagators, extended from
$C_\infty$ to $C_b$  earlier in the proof, gives, with $h_\xi(y)=e^{i\xi\cdot y}$,
\[
 \mathbb E^n_{r,x}e^{i\xi\cdot \omega_t}
 =P_n^{r,t}h_\xi(x)
 \rightarrow
 P^{r,t}h_\xi(x)
 =\mathbb E_{r,x}e^{i\xi\cdot \omega_t}.
\]
For the integral term, we write
\[
 I_n:=\mathbb E^n_{r,x}\int_r^t
       \xi\cdot b_n(s,\omega_s)\tilde{\nu}(s,\omega_s)\,ds,
 \qquad
 I:=\mathbb E_{r,x}\int_r^t
       \xi\cdot b(s,\omega_s)\tilde{\nu}(s,\omega_s)\,ds.
\]
The second integral is finite by \eqref{kr_bd2} with
$\mathsf f=b$ and $h=1$. For a fixed  $m$, let us add
and subtract the terms containing $b_m$:
\begin{align*}
 I_n-I
 ={}&\mathbb E^n_{r,x}\int_r^t
       \xi\cdot(b_n-b_m)(s,\omega_s)\tilde{\nu}(s,\omega_s)\,ds\\
 &+\left(\mathbb E^n_{r,x}\int_r^t q_m(s,\omega_s)\,ds
         -\mathbb E_{r,x}\int_r^t q_m(s,\omega_s)\,ds\right)\\
 &+\mathbb E_{r,x}\int_r^t
       \xi\cdot(b_m-b)(s,\omega_s)\tilde{\nu}(s,\omega_s)\,ds,
 \qquad q_m(s,y):=\xi\cdot b_m(s,y)\tilde{\nu}(s,y).
\end{align*}
For fixed $m$, $q_m$ is bounded and continuous, so the middle
line tends to zero as $n\rightarrow\infty$.

To estimate the first and last terms, we use the estimates on $[r,t]$. The ``smooth'' estimate
\eqref{apr_bd}, applied with the test vector  field $b_n-b_m$, and the
limiting estimate \eqref{kr_bd2}, applied with $b_m-b$, yield
\begin{align*}
 \left|\mathbb E^n_{r,x}\int_r^t
       \xi\cdot(b_n-b_m)\tilde{\nu}\,du\right|
 &\le K|\xi|\|b_n-b_m\|_{L^1_{\rho_x}([r,t] \times \mathbb R^d)}^{1/p},\\
 \left|\mathbb E_{r,x}\int_r^t
       \xi\cdot(b_m-b)\tilde{\nu}\,du\right|
 &\le K|\xi|\|b_m-b\|_{L^1_{\rho_x}([r,t] \times \mathbb R^d)}^{1/p}.
\end{align*}
Here $K$ is independent of $n,m,r,x$.

For each fixed $x$, we have
$b_n\rightarrow b$ in $L^1_{\rho_x}([r,t] \times \mathbb R^d)$. Indeed, we have this convergence locally in $L^1$. At the same time, the integrals of $|b_n|$, $|b|$  over unit cubes are uniformly bounded, see the proof of \eqref{weighted_drift_mass}. Summing these bounds
against the suprema of $\rho_x$ outside a large ball allows us to ``cut tails'', i.e.\,make the
weighted tails uniformly small since $\theta>d$. This, and the local convergence in $L^1$ yield the convergence in $L_{\rho_x}^1$. 

Therefore, first letting
$n\rightarrow\infty$ with $m$ fixed gives
\[
 \limsup_{n\rightarrow\infty}|I_n-I|
 \le 2K|\xi|\|b-b_m\|_{L^1_{\rho_x}([r,t] \times \mathbb R^d)}^{1/p}.
\]
Taking $m\rightarrow\infty$ now yields $I_n\rightarrow I$. 

Passing to the limit
in \eqref{n_id}, we obtain
\[
 \tilde\nu(r,t,x,\xi)
 =\mathbb E_{r,x}e^{i\xi\cdot \omega_t}
  +i\int_r^t\mathbb E_{r,x}\big[
       \xi\cdot b(u,\omega_u)\tilde\nu(u,t,\omega_u,\xi)\big]\,du,
\]
as claimed.
\end{proof}

Let 
$$
\mu(r,t,x,\xi) 
:=\nu(r,t,x,\xi) -\tilde{\nu}(r,t,x,\xi).
$$
Then, comparing the definitions of $\nu$ and $\tilde{\nu}$, we have
\begin{align}\label{eq4}
\mu(r,t,x,\xi) 
= i \int_r^t \mathbb{E}_{r,x}\big[ \xi \cdot b(u,\omega_u)\, 
\mu(u,t,\omega_u,\xi) \big] du, \quad 0 \leq r \leq t, \quad x \in \mathbb R^d,
\end{align}

To conclude our proof that $Z_t(\omega)$ is an isotropic $\alpha$-stable process under $\mathbb P_{0,x}$ it remains to establish the following claim.

\begin{claim}
\label{cl3}	
$\mu$
is identically zero.
\end{claim}

\begin{proof}
We first show that, for fixed $\xi\in\mathbb{R}^d$, there exists $a>0$ sufficiently small such that on the interval $r\in[t-a,t]\subset[0,t]$, the only bounded solution to the homogeneous equation \eqref{eq4} is $\mu\equiv 0$.

First, let us note that, applying Krylov bound \eqref{kr_bd2} on $[r,v]$ with $\mathsf f=b$ and
$h=1$, and then using integral estimate \eqref{weighted_drift_mass}, yields
\begin{equation}\label{uniform_drift_occupation}
 \sup_{x\in\mathbb R^d}\mathbb E_{r,x}
       \int_r^v|b(u,\omega_u)|\,du
 \le C_T\eta(v-r),
 \qquad 0\le r\le v\le T,\quad v-r\le 1,
\end{equation}
where
$$
\eta(v-r):=(v-r)^{1/(\alpha p)}.
$$
The constant $C_T$ is independent of $r,v,x$. 

In what follows, by the $L^\infty$ norm we mean the $\sup$-norm (not ${\rm esssup}$, i.e.\,we do not ignore sets of measure zero).
Define the operator
$$
\mu \mapsto \Gamma \mu(r,t,x,\xi) := i \int_r^t \mathbb{E}_{r,x}\big[\xi \cdot b(u,\omega_u)\, \mu(u,t,\omega_u,\xi)\big]\, du, \quad (r,x)\in[t-a,t] \times \mathbb R^d,
$$
acting on functions $
\mu(\cdot,t,\cdot,\xi)
 \in L^\infty([t-a,t] \times \mathbb R^d)$. The constant $a$ will be chosen sufficiently small so that $\Gamma$ is a contraction on  $L^\infty([t-a,t] \times \mathbb R^d)$. 
We have
\begin{align*}
&\sup_{(r,x)\in[t-a,t] \times \mathbb R^d} |\Gamma \mu(r,t,x,\xi)|\\
&\quad \leq 
    \sup_{(r,x)\in[t-a,t] \times \mathbb R^d} |\xi| \int_r^t \mathbb{E}_{r,x}[|b(u,\omega_u)|\, |\mu(u,t,\omega_u,\xi)| ] \, du\\
&\quad \leq 
    \sup_{(r,x)\in[t-a,t] \times \mathbb R^d} |\xi| \int_r^t \mathbb{E}_{r,x}|b(u,\omega_u)|  du\, \|\mu(\cdot,t,\cdot,\xi)\|_{L^\infty([t-a,t] \times \mathbb R^d)}.
\end{align*}
By \eqref{uniform_drift_occupation},
\[
 \sup_{(r,x)\in[t-a,t] \times \mathbb R^d}
 \mathbb E_{r,x}\int_r^t|b(u,\omega_u)|\,du\le\eta(a).
\]
For the given $\xi$, choose an integer $N\ge2$ sufficiently large
and set $a=t/N$ so that $a\le 1$ and $|\xi|\eta(a)<1$. This is possible since
$\eta(a)\rightarrow0$. The same choice of $a$ works for every subinterval of $[0,t]$
of length at most $a$. We have therefore shown that
\[
 \|\Gamma\mu\|_{L^\infty([t-a,t] \times \mathbb R^d)}
 \le |\xi|\eta(a)\,
      \|\mu\|_{L^\infty([t-a,t] \times \mathbb R^d)}.
\]
Therefore, $\Gamma$ is a contraction on $L^{\infty}([t-a,t] \times \mathbb R^d)$. We conclude that $\mu$, as a solution of \eqref{eq4}, must be identically equal to zero, for all $(r,x)\in[t-a,t] \times \mathbb R^d$. 

We now extend this conclusion to the whole interval $[0,t]$. Indeed, for $r\in [t-2a, t-a]$, we have
\begin{align*} \mu(r,t,x,\xi) 
&= i \int_r^t \mathbb{E}_{r,x}\big[\xi \cdot b(u,\omega_u)\, \mu(u,t,\omega_u,\xi)\big]\, du\\
&=i \int_r^{t-a} \mathbb{E}_{r,x}\big[\xi \cdot b(u,\omega_u)\, \mu(u,t,\omega_u,\xi)\big]\, du
+i \int_{t-a}^{t} \mathbb{E}_{r,x}\big[\xi \cdot b(u,\omega_u)\, \mu(u,t,\omega_u,\xi)\big]\, du\\
&=i \int_r^{t-a} \mathbb{E}_{r,x}\big[\xi \cdot b(u,\omega_u)\, \mu(u,t,\omega_u,\xi)\big]\, du,
\end{align*}
where we have used that $\mu=0$ on $[t-a,t] \times \mathbb R^d$ established in the last step.
Applying \eqref{uniform_drift_occupation} once again, we obtain
\begin{align*}
&\|\mu(\cdot,t,\cdot,\xi)\|_{L^\infty ([t-2a,t-a] \times \mathbb R^d)} \leq 
    \sup_{(r,x)\in[t-2a,t-a] \times \mathbb R^d} |\xi| \int_r^{t-a} \mathbb{E}_{r,x}|b(u,\omega_u)|  du\, \|\mu(\cdot,t,\cdot,\xi)\|_{L^\infty ([t-2a,t-a] \times \mathbb R^d)}\notag\\
& \quad \le |\xi|\eta(a) \|\mu(\cdot,t,\cdot,\xi)\|_{L^\infty ([t-2a,t-a] \times \mathbb R^d)},
\end{align*}
where $|\xi|\eta(a)<1$. It follows immediately that $\mu=0$ for all $(r,x)\in[t-2a,t-a] \times \mathbb R^d$.
Repeating this for the $N$ intervals of length $a=t/N$, we obtain that $\mu$ is identically equal to zero for all $(r,x)\in[0,t] \times \mathbb R^d$, as claimed.
\end{proof}

Claim \ref{cl3} gives us
$$
\nu(r,t,x,\xi)=\tilde{\nu}(r,t,x,\xi),
$$
which concludes the proof of (c) $\Rightarrow$ (\textit{i}).

\medskip

\noindent \textit{Proof of} (\textit{iii}). Without loss of generality, we assume that the initial time $s=0$.
Suppose that there exist two weak solutions $\mathbb P^1_x$, $\mathbb P^2_x$ to \eqref{sde1} that satisfy, for $p>d+1$ from (\textit{i}),
\begin{align}
\label{kr_est1}
\mathbb E_{x}^i\int_0^T |h(t,\omega_t)|dt \leq K\|h\|_{L^p([0,T] \times \mathbb R^d)}
\end{align}
and
\begin{align}
\label{kr_est2}
\mathbb E_{x}^i\int_0^T |b(t,\omega_t)h(t,\omega_t)|dt \leq K\||b|^{\frac{1}{p}}h\|_{L^p([0,T] \times \mathbb R^d)}
\end{align}
for all $h \in (C_b \cap L^p)(\mathbb{R}^{1+d})$,
with a constant $K$ independent of $h$ ($i=1,2$).
Here we use the unweighted bounds, as in the statement of (\textit{iii}). The weak solution constructed in (\textit{i}) satisfies these bounds since $\rho_x\leq 1$. We may take the larger of the constants for the two solutions.

We will also use \eqref{kr_est1}, \eqref{kr_est2} for Borel measurable functions $h$ provided the right-hand side is finite. In detail, the measures
$$
\mu_i(B):=\mathbb E_x^i\int_0^T\mathbf{1}_B(t,\omega_t)\,dt,
\qquad
\nu_i(B):=\mathbb E_x^i\int_0^T|b(t,\omega_t)|\mathbf{1}_B(t,\omega_t)\,dt
$$
are finite; for $\nu_i$ this follows from (b). Then \eqref{kr_est1} makes the integration against $\mu_i$ (first, against smooth functions) a bounded functional on $L^p(dt\,dy)$. Similarly, \eqref{kr_est2} makes integration against $\nu_i$ a bounded functional on $L^p(|b|dt\,dy)$. Smooth functions of compact support are dense in these spaces, because $|b|$ is locally integrable. By the $L^p$ duality the two measures therefore have densities in the corresponding $L^{p'}$ spaces, with norms at most $K$. H\"older's inequality now gives the same bounds for all Borel measurable functions with finite right-hand side. In particular, indicator functions occurring below are admissible.

We will identify the conditional expectations under each $\mathbb P_x^i$, needed to show that the laws coincide, with the Feller propagator $P^{r,t}$ constructed in (\textit{ii}). Fix a deterministic time $0\leq r<T$ and a bounded $\mathcal F_r$-measurable random variable $H$.
For $F\in C_c^\infty(\mathbb R^{1+d})$, let $u_n \in C([0,T],C_\infty(\mathbb R^d))$ be the classical solution to the terminal-value problem
\begin{equation}
\label{eq_F}
(-\partial + A + b_n \cdot \nabla)u_n=F, \quad u_n(T,\cdot)=0,
\end{equation}
where, recall, the bounded smooth regularization $b_n$, $n=1,2,\dots$, of $b$ is defined via \eqref{b_n}.
We have, by the properties of mollifiers,
$$
b_n \rightarrow b \quad \text{ in } L^1_{\loc}(\mathbb R^{1+d},\mathbb R^d).
$$
Put
$$
\tau_R^r:=\inf\{t\in[r,T]:|\omega_t|\geq R\}\wedge T,
\qquad \inf\varnothing:=\infty,
$$
i.e.\,the first exit time from $B_R=B_R(0)$ after time $r$, stopped at $T$.
We now apply It\^o's formula (see Section \ref{notations_sect}) between $r$ and $\tau_R^r$, multiply by $H$, and take expectations:
\begin{align}
\mathbb E_x^i[H u_n(\tau_R^r,\omega_{\tau_R^r})]
&=\mathbb E_x^i[H u_n(r,\omega_r)]
 -\mathbb E_x^i\left[H\int_r^{\tau_R^r}F(t,\omega_t)\,dt\right]\notag\\
&\quad-\mathbb E_x^i\left[H\int_r^{\tau_R^r}
          [(b-b_n)\cdot\nabla u_n](t,\omega_t)\,dt\right],
\label{e_i}
\end{align}
where we have used 
$\partial u_n-Au_n-b\cdot\nabla u_n=-F-(b-b_n)\cdot\nabla u_n$ by \eqref{eq_F}.

Our goal is to show that the last term tends to zero. Applying \eqref{kr_est2} and \eqref{kr_est1}, we obtain
\begin{align*}
&\mathbb E_x^i\int_r^{\tau_R^r}
       \big|(b-b_n)\cdot\nabla u_n\big|(t,\omega_t)\,dt\\
&\quad\leq
K\|\mathbf{1}_{B_R}(1-\mathbf{1}_n)
          b^{\frac1p}\cdot\nabla u_n\|_{L^p([0,T] \times \mathbb R^d)}
+K\|\mathbf{1}_{B_R}(\mathbf{1}_n b-b_n)
          \cdot\nabla u_n\|_{L^p([0,T] \times \mathbb R^d)}=:I_n+J_n.
\end{align*}
For the first term, apply \eqref{kr_est2} with
$$
h=\mathbf{1}_{[0,T]\times B_R}(1-\mathbf{1}_n)
       \frac{b}{|b|}\cdot\nabla u_n,
$$
where $b/|b|$ is set equal to zero when $b=0$. For the second term use \eqref{kr_est1}. Before $\tau_R^r$ the path lies in $B_R$, and the single exit time does not affect the time integrals (we have enlarged the time interval to $[0,T]$ to apply the aforementioned bounds). Thus, the absolute value of the last term in \eqref{e_i} is at most $\|H\|_\infty(I_n+J_n)$.

\medskip

1) Let us deal with $I_n$.
We use the convergence in the Sobolev space from \ref{dualsubsect} at the same exponent $p$ as in \eqref{kr_est1}, \eqref{kr_est2}. Put
$$
\beta:=\frac1\alpha+\frac{\alpha-1}{\alpha p},\qquad
e_\lambda(t):=e^{\lambda(t-T)},\qquad
f:=\mathbf{1}_{[0,T]}e_\lambda F.
$$
On $[0,T]$, the function $\widetilde u_n=e_\lambda u_n$ satisfies
$$
(\lambda-\partial+A+b_n\cdot\nabla)\widetilde u_n=f,
\qquad \widetilde u_n(T,\cdot)=0.
$$
In order to apply \ref{dualsubsect}, we extend $\widetilde u_n$ to the whole time axis as the solution of this equation with right-hand side $f$. It vanishes for $t\geq T$. After reversing time, the convergence asserted in \ref{dualsubsect} yields
$$
v_n:=(\lambda-\partial+A)^\beta\widetilde u_n\rightarrow v
       \quad\text{in }L^p(\mathbb R^{1+d}),
\qquad
\sup_n\|v_n\|_p\leq C\|F\|_p.
$$
Here and below $\lambda$ is fixed sufficiently large. The constant $C$ is independent of $n$.

The operator
$$
R_p(b):=b^{\frac1p}\cdot\nabla (\lambda-\partial+A)^{-\beta}
$$
is bounded on $L^p$, by \ref{dualsubsect} with time reversed. Therefore,
\begin{align*}
w_n&:=\mathbf{1}_{[0,T]}b^{\frac1p}\cdot\nabla u_n
     =\mathbf{1}_{[0,T]}e_{-\lambda}R_p(b)v_n\rightarrow
w:=\mathbf{1}_{[0,T]}e_{-\lambda}R_p(b)v
       \quad\text{in }L^p(\mathbb R^{1+d}).
\end{align*}
The multiplier $e_{-\lambda}$ is bounded by $e^{\lambda T}$ on $[0,T]$.
Let us note that the operator $R_p(b)$ in this formula contains the original drift $b$, while $v_n$ comes from the equation with regularized drift $b_n$.
Since $\mathbf{1}_n\rightarrow1$ almost everywhere, we obtain
$$
I_n\leq K\|w_n-w\|_p+K\|(1-\mathbf{1}_n)w\|_p
       \rightarrow 0.
$$
The first term tends to zero by the strong convergence just proved. The second term tends to zero by the dominated convergence theorem since $|w|^p$ has finite integral.

\medskip

2) Next, we deal with $J_n$. This term is easier since $\mathbf{1}_n b-b_n$ is simply the difference between the cutoff of $b$ and the mollification of this cutoff.
We need a little more integrability of $\nabla u_n$. Choose $p_1>p$ sufficiently close to $p$ that
$$
\left(1+\frac d\alpha\right)
\left(\frac1p-\frac1{p_1}\right)<\frac{\alpha-1}{\alpha p}.
$$
Such a choice is possible because $\alpha>1$. The parabolic Riesz transform
$\nabla (\lambda-\partial+A)^{-1/\alpha}$ is bounded on $L^p$. Applying the Sobolev embedding \eqref{sob_emb0}, \eqref{sob_emb} to
$$
\nabla\widetilde u_n =(\lambda-\partial+A)^{-\frac{\alpha-1}{\alpha p}}
\big(\nabla (\lambda-\partial+A)^{-1/\alpha}v_n\big),
$$
we obtain
$$
\|\nabla u_n\|_{L^{p_1}([0,T] \times \mathbb R^d)} \leq C e^{\lambda T}\|v_n\|_p
\leq C e^{\lambda T}\|F\|_p.
$$
We now fix $q$ by $1/p=1/p_1+1/q$, and fix a sequence $a_n\downarrow0$.
The cutoff $\mathbf{1}_n b$ is bounded and compactly supported. According to the definition of \eqref{b_n}, we are free to select the mollifier radii $\varepsilon_n\downarrow0$ sufficiently rapidly, so that
$$
\|\mathbf{1}_n b-b_n\|_q\leq a_n.
$$
The radii can be chosen while preserving the uniform Morrey and operator bounds imposed in \eqref{b_n}.
H\"older's inequality and the previous gradient bound now give
$$
J_n\leq K\|\mathbf{1}_n b-b_n\|_q \|\nabla u_n\|_{L^{p_1}([0,T] \times \mathbb R^d)}
\leq C K a_n e^{\lambda T}\|F\|_p \rightarrow 0.
$$

\medskip

We conclude from 1) and 2) that
$$
\left|\mathbb E_x^i\left[H\int_r^{\tau_R^r}
       [(b-b_n)\cdot\nabla u_n](t,\omega_t)\,dt\right]\right|
\leq \|H\|_\infty(I_n+J_n)\rightarrow0.
$$
It remains to note, using again \ref{dualsubsect}, that solutions $u_n$ converge uniformly to a function $u\in C([0,T],C_\infty(\mathbb R^d))$, with $u(T,\cdot)=0$.

\medskip

Therefore, we can pass to the limit in \eqref{e_i}, first in $n \rightarrow \infty$ with $R$ fixed, and then we take $R\rightarrow\infty$. Every c\`adl\`ag path is bounded on $[0,T]$, so $\tau_R^r=T$ for all sufficiently large $R$, for each such path. Since $H$, $u$ and $F$ are bounded, the dominated convergence theorem applies to the terminal term and to the integral of $F$. We obtain
\begin{equation}
\label{uniq_cond_occ}
\mathbb E_x^i\left[H\int_r^T F(t,\omega_t)\,dt\right]
       =\mathbb E_x^i[H u(r,\omega_r)],
\end{equation}
where, in turn, by the strong convergence of propagators,
\begin{equation}
\label{uniq_u_repr}
u(r,y)=\int_r^T P^{r,t}F(t,\cdot)(y)\,dt.
\end{equation}

Next, let us take $F(t,y)=\psi(t)g(y)$, where $\psi\in C_c^\infty(]r,T[)$ and $g\in C_c^\infty(\mathbb R^d)$. Combining \eqref{uniq_cond_occ} and \eqref{uniq_u_repr}, and applying Fubini's theorem, we obtain
$$
\int_r^T\psi(t)\left(
\mathbb E_x^i[H g(\omega_t)]
-\mathbb E_x^i[H P^{r,t}g(\omega_r)]\right)\,dt=0.
$$
Since $\psi$ is arbitrary, the expression in brackets is zero for a.e.\,$t\in]r,T[$. Moreover, it is zero for every $r\leq t\leq T$. Indeed, the SDE expresses $\omega_t$ as the sum of a continuous drift integral and a stochastically continuous isotropic $\alpha$-stable process. Hence $\omega_{t_k}\rightarrow\omega_t$ in probability as $t_k\rightarrow t$, including at $t=T$. Since $g$ and $H$ are bounded, $t \mapsto \mathbb E_x^i[H g(\omega_t)]$ is continuous. The other expectation is continuous by the strong continuity of $P^{r,t}$ on $C_\infty$. Thus, vanishing a.e.\,implies, by continuity, vanishing everywhere.

Since $H$ is arbitrary, we conclude
\begin{equation}
\label{uniq_cond_trans}
\mathbb E_x^i[g(\omega_t)\mid\mathcal F_r]
       =P^{r,t}g(\omega_r),\qquad 0\leq r\leq t\leq T,
\end{equation}
for all  $g\in C_c^\infty(\mathbb R^d)$. Furthermore, we have this identity for all $g\in C_\infty(\mathbb R^d)$ by density.

\medskip
\noindent\emph{Finite-dimensional distributions coincide.}
Let $0<t_1<\cdots<t_m\leq T$, and let $g_1,\dots,g_m\in C_\infty(\mathbb R^d)$. For $m\geq2$, apply \eqref{uniq_cond_trans} with $r=t_{m-1}$, $t=t_m$, and multiply by $\prod_{j=1}^{m-1}g_j(\omega_{t_j})$. Taking expectations gives
\begin{align*}
\mathbb E_x^i\prod_{j=1}^m g_j(\omega_{t_j})
=\mathbb E_x^i\left[
\left(\prod_{j=1}^{m-2}g_j(\omega_{t_j})\right)
\big(g_{m-1}P^{t_{m-1},t_m}g_m\big)(\omega_{t_{m-1}})
\right].
\end{align*}
The Feller property gives $P^{t_{m-1},t_m}g_m\in C_\infty$, and therefore $g_{m-1}P^{t_{m-1},t_m}g_m\in C_\infty$. We can consequently apply the same argument again. Repeating it, and finally using $\omega_0=x$, yields
$$
\mathbb E_x^i\prod_{j=1}^m g_j(\omega_{t_j})
=P^{0,t_1}\big(g_1P^{t_1,t_2}(g_2\cdots
      P^{t_{m-1},t_m}g_m)\cdots\big)(x).
$$
The right-hand side does not depend on $i$. Finite linear combinations of the product functions $g_1(y_1)\cdots g_m(y_m)$ are uniformly dense in $C_\infty((\mathbb R^d)^m)$ by the Stone--Weierstrass theorem. Hence the two joint probability measures agree on all functions in this space, and therefore coincide. This yields equality of all finite-dimensional distributions. 

\medskip

The coordinate maps generate the Borel $\sigma$-field of $D([0,T],\mathbb R^d)$, so these distributions determine the law on the canonical path space. Hence $\mathbb P_x^1=\mathbb P_x^2$, and their completions also coincide. Taking one of these laws to be the solution constructed in (\textit{i}) ends the proof of (\textit{iii}). 

This ends the proof of Theorem \ref{thm1}. \hfill \qed

\bigskip

\section{Proof of Theorem \ref{thm2}}

Throughout this section, we write
$
A:=(-\Delta)^{\frac{\alpha}{2}}.
$
Assertion (\textit{i}) follows directly from Proposition \ref{prop1} and
\begin{equation}
\label{conv_GH}
	G_p(b_n)\rightarrow G_p(b),\quad
	H_p(b_n)\rightarrow H_p(b)
	\quad 
	\text{strongly in}\ L^{p}(\mathbb{R}^{1+d})
	\end{equation}
	which, in turn, is a consequence of our choice of approximation $\{b_n\}$ (essentially monotone) and the Dominated Convergence Theorem.

	  Let us prove (\textit{ii}).  \textit{Existence.}
	The fact that $\eta$ constructed in (\textit{i}) is a weak solution follows right away from the convergence 
	$$
	\eta_n \rightarrow \eta \quad \text{ in } \mathbb{W}_\alpha^{(1-\frac{1}{\alpha})\frac{2}{p},p}(\mathbb{R}^{1+d}),
	$$
	from (\textit{i}), i.e.
	\begin{align}\label{conv3}
		(\lambda+\partial+A)^{(1-\frac{1}{\alpha})\frac{1}{p}}\eta_n
		\rightarrow
		(\lambda+\partial+A)^{(1-\frac{1}{\alpha})\frac{1}{p}}\eta
		\qquad\text{in }L^p(\mathbb{R}^{1+d}),
	\end{align}
	and the convergence \eqref{conv_GH} which we supplement with $G_p^\ast(b_n) \rightarrow G_p^\ast(b)$ in $L^{p'}(\mathbb R^{1+d})$. (Having an explicit representation from (\textit{i}) for the candidate for a weak solution, i.e.\,$\eta$, greatly simplifies the proof of weak existence.)

	\textit{Uniqueness.} First, we note that, given any weak solution $\eta$, in view of the inclusions $H_p \in \mathcal B(L^p)$, $G_p^\ast \in \mathcal B(L^{p'})$ established in Theorem \ref{thm2}, if the right-hand side $f$ satisfies \eqref{f_cond}, then the identity \eqref{weak_def3_} extends to 
			    \begin{align}
		&\langle
		(\lambda+\partial_t+A)^{(1-\frac{1}{\alpha})\frac{1}{p}} \eta,
		(\lambda-\partial_t+A)^{1-(1-\frac{1}{\alpha})\frac{1}{p}} \varphi
		\rangle \notag \\
		&\quad+
		\langle
		H_p	(\lambda+\partial_t+A)^{(1-\frac{1}{\alpha})\frac{1}{p}}  \eta, 
		-G^\ast_p
		(\lambda-\partial_t+A)^{1-(1-\frac{1}{\alpha})\frac{1}{p}}   \varphi
		\rangle\notag\\
		&\quad \quad =
		\langle
		(\lambda+\partial_t+A)^{-1+(1-\frac{1}{\alpha})\frac{1}{p}} f,
		(\lambda-\partial_t +A)^{1-(1-\frac{1}{\alpha})\frac{1}{p}} \varphi
		\rangle \label{weak_def3}
	\end{align}		
		for all test functions 
	$\varphi \in L^{p'}(\mathbb R^{1+d})$ such that 
	$$
	(\lambda-\partial_t +A)^{1-(1-\frac{1}{\alpha})\frac{1}{p}}\varphi \in L^{p'}(\mathbb{R}^{1+d}).
	$$
	Now, let $\eta$ be the weak solution constructed earlier, and let $\eta_\star \in \mathbb{W}_\alpha^{(1-\frac{1}{\alpha})\frac{2}{p},p}(\mathbb{R}^{1+d})$ be another weak solution. Then, by the previous remark, identity \eqref{weak_def3} holds for $\eta_\star$ as well. Therefore, the difference $\eta-\eta_\star$ satisfies
		\begin{align}
		&\langle
		(\lambda+\partial+A)^{(1-\frac{1}{\alpha})\frac{1}{p}} (\eta-\eta_\star),
		(\lambda-\partial+A)^{1-(1-\frac{1}{\alpha})\frac{1}{p}} \varphi
		\rangle \notag \\
		&\quad+
		\langle
		H_p	(\lambda+\partial+A)^{(1-\frac{1}{\alpha})\frac{1}{p}}  (\eta-\eta_\star), 
		-G_p^\ast
		(\lambda-\partial+A)^{1-(1-\frac{1}{\alpha})\frac{1}{p}}   \varphi
		\rangle
		=
		0 \label{weak_diff}
	\end{align}
	for all test functions
	$\varphi \in L^{p'}(\mathbb R^{1+d})$ such that 
	$$
	(\lambda-\partial+
	A
	)^{1-(1-\frac{1}{\alpha})\frac{1}{p}}\varphi \in L^{p'}(\mathbb{R}^{1+d}).
	$$
	Now, we choose in \eqref{weak_diff}
	$$
	\varphi:= (\lambda-\partial +A)^{-1+(1-\frac{1}{\alpha})\frac{1}{p}}
	[(\lambda+\partial +A)^{(1-\frac{1}{\alpha})\frac{1}{p}}(\eta-\eta_\star)|(\lambda+\partial +A)^{(1-\frac{1}{\alpha})\frac{1}{p}}(\eta-\eta_\star)|^{p-2}],
	$$
	where the inclusion $(\lambda-\partial+A)^{1-(1-\frac{1}{\alpha})\frac{1}{p}}\varphi \in L^{p'}(\mathbb{R}^{1+d})$ follows from $\eta, \eta_\star \in \mathbb{W}_\alpha^{(1-\frac{1}{\alpha})\frac{2}{p},p}(\mathbb{R}^{1+d})$,  thus defined $\varphi$ is an admissible test function.
We arrive at
		\begin{align*}
		&\langle
		(\lambda+\partial+A)^{(1-\frac{1}{\alpha})\frac{1}{p}} (\eta-\eta_\star),
		(\lambda+\partial +A)^{(1-\frac{1}{\alpha})\frac{1}{p}}(\eta-\eta_\star)|(\lambda+\partial +A)^{(1-\frac{1}{\alpha})\frac{1}{p}}(\eta-\eta_\star)|^{p-2}
		\rangle\\
		&\quad+
		\langle
		H_p	(\lambda+\partial+A)^{(1-\frac{1}{\alpha})\frac{1}{p}}  (\eta-\eta_\star), 
		-G_p^\ast
		(\lambda+\partial +A)^{(1-\frac{1}{\alpha})\frac{1}{p}}(\eta-\eta_\star)|(\lambda+\partial +A)^{(1-\frac{1}{\alpha})\frac{1}{p}}(\eta-\eta_\star)|^{p-2}
		\rangle \\
		&\quad =
		0.
	\end{align*}
	Observe that the first term coincides with $\|\eta-\eta_\star\|_{\mathbb{W}_\alpha^{(1-\frac{1}{\alpha})\frac{2}{p},p}(\mathbb{R}^{1+d})}^p$, while the absolute value of the second term is bounded from above by $\|H_p\|_{p \rightarrow p} \|G^\ast_p\|_{p' \rightarrow p'}\|\eta-\eta_\star\|^p_{\mathbb{W}_\alpha^{(1-\frac{1}{\alpha})\frac{2}{p},p}}$. 
		 Hence
	 $$
	 (1-\|H_p\|_{p \rightarrow p} \|G^\ast_p\|_{p' \rightarrow p'})
	 \|\eta-\eta_\star
	 \|_{\mathbb{W}_\alpha^{(1-\frac{1}{\alpha})\frac{2}{p},p}}^p\le 0.
	 $$
	By our assumption on the Morrey norm of $b$, 
	$$
	\|H_p\|_{p \rightarrow p} \|G^\ast_p\|_{p' \rightarrow p'} \equiv \|H_p\|_{p \rightarrow p} \|G_p\|_{p \rightarrow p}<1,
	$$
	 which implies $\eta=\eta_\star$. This ends the proof of assertion (\textit{ii}).
	 
	 \medskip
	 
(\textit{iii}) is a consequence of (\textit{ii}) and $\delta_{t=0}\delta_x \in \mathbb{W}_\alpha^{-\frac{2}{\alpha}+(-1+\frac{1}{\alpha})\frac{2}{p'},p}(\mathbb{R}^{1+d})$.
In detail, by the parabolic Sobolev embedding property \eqref{sob_emb0} \eqref{sob_emb}, we have
	$$
	\mathbb{W}_\alpha^{\frac{2}{\alpha}-(-1+\frac{1}{\alpha})\frac{2}{p'},p'}
	(\mathbb{R}^{1+d})
	\hookrightarrow
	C_\infty (\mathbb{R}^{1+d})
	\quad\text{provided}\ 
	p<\frac{d+1}{d}.
	$$
	Hence,
	$$
	\delta_{t=0}\delta_x\in (C_\infty (\mathbb{R}^{1+d}))'
	\hookrightarrow (\mathbb{W}_\alpha^{\frac{2}{\alpha}-(-1+\frac{1}{\alpha})\frac{2}{p'},p'}
	(\mathbb{R}^{1+d}))'
	=\mathbb{W}_\alpha^{-\frac{2}{\alpha}+(-1+\frac{1}{\alpha})\frac{2}{p'},p}
	(\mathbb{R}^{1+d}).
	$$
	\hfill \qed

\bigskip

\section{Proof of Theorem \ref{thm3}}
(\textit{i})
Set
$
A:=\sum_{i=1}^N (-\Delta)^{\frac{\alpha}{2}}_{x^i}.
$
In order to apply the argument of \cite{KM}, we need to establish the following:

a) Pointwise bound
 $$
 |\nabla_x (\lambda+A)^{-1}(x-y)| \leq c_{d,N,\alpha}(\lambda+A)^{-1+\frac{1}{\alpha}}(x-y)
 $$
 for all $x,y \in \mathbb R^{dN}$, $x \neq y$. Indeed, recall that, for $z \in \mathbb R^d$,
 $$
 e^{-t(-\Delta)^{\frac{\alpha}{2}}_{\mathbb R^d}}(z) \approx t\bigl(|z|^{-d-\alpha} \wedge t^{-\frac{d+\alpha}{\alpha}}\bigr)
$$ 
and, using \cite[proof of Lemma 5]{BJ},
\begin{align*}
|\nabla_z e^{-t(-\Delta)^{\frac{\alpha}{2}}_{\mathbb R^d}}(z)| & \leq K_1 t\bigl( |z|^{-d-\alpha-1} \wedge t^{-\frac{d+\alpha+1}{\alpha}}\bigr) \\
& (\text{use the previous estimate}) \\
& \leq K_2 t^{-\frac{1}{\alpha}}e^{-t(-\Delta)^{\frac{\alpha}{2}}_{\mathbb R^d}}(z).
\end{align*}
Therefore, taking into account the factorization of $e^{-tA}$, for all $x \in \mathbb R^{dN}$,
\begin{align*}
|\nabla_{x^i} e^{-t A}(x)| & \leq |\nabla_{x^i }e^{-t(-\Delta)_{x^i}^{\frac{\alpha}{2}}}(x^i)|\prod_{j \neq i}e^{-t(-\Delta)_{x^j}^{\frac{\alpha}{2}}}(x^j) \\
& \leq K_2 t^{-\frac{1}{\alpha}}e^{-t A}(x).
\end{align*}
Hence 
$$
|\nabla_x e^{-tA}(x)| \leq C t^{-\frac{1}{\alpha}}e^{-t A}(x).
$$
Multiplying the previous inequality by $e^{-\lambda t}$ and integrating in $t$ from $0$ to $\infty$, we arrive at the sought bound using $(\lambda+A)^{-\frac{\gamma}{2}}(x)=\frac{1}{\Gamma(\frac{\gamma}{2})} \int_0^\infty e^{-\lambda t}t^{\frac{\gamma}{2}-1} e^{-tA}(x)dt$ ($0<\gamma \leq 2$).

b) The many-particle drift $b=(b^1,\dots,b^N)$,
$$
b_i(x)=\frac{1}{N}\sum_{j=1, j \neq i}^N K(x^i-x^j),
$$
is weakly form-bounded on $\mathbb R^{dN}$ with respect to $A$. This follows by first noting that $K_{ij}(x):=K(x^i-x^j)$ satisfies
$$
\||K_{ij}|^{\frac{1}{2}}(\lambda+(-\Delta)_{x^i}^{\frac{\alpha}{2}})^{-\frac{1}{2}\frac{\alpha-1}{\alpha}}\|_{L^2(\mathbb R^{dN}) \rightarrow L^2(\mathbb R^{dN})} \leq \sqrt{\delta}
$$
and therefore
$
\||K_{ij}|^{\frac{1}{2}}(\lambda+A)^{-\frac{1}{2}\frac{\alpha-1}{\alpha}}\|_{L^2(\mathbb R^{dN}) \rightarrow L^2(\mathbb R^{dN})} \leq \sqrt{\delta},
$
so it remains to add up these inequalities, taking into account the normalization by $\frac{1}{N}$, in order to arrive at
$$
\||b|^{\frac{1}{2}}(\lambda+A)^{-\frac{1}{2}\frac{\alpha-1}{\alpha}}\|_{L^2(\mathbb R^{dN}) \rightarrow L^2(\mathbb R^{dN})} \leq \sqrt{N-1}\sqrt{\delta}.
$$

Armed with a), b) and using the fact that $A$ is a symmetric Markov generator, we obtain the estimate
\begin{equation*}
\|b^{\frac{1}{p}} \cdot \nabla (\lambda+A)^{-1}|b|^{\frac{1}{p'}}\|_{L^p(\mathbb R^{dN}) \rightarrow L^p(\mathbb R^{dN})} \leq c_{d,\alpha,p}(N-1)\delta
\end{equation*} 
and repeat the argument of \cite{KM}, i.e.\,postulate the Neumann series for $A +b \cdot \nabla$, show that it determines the resolvent of a Feller generator, and then use the argument of Podolynny-Portenko and Chen-Wang -- i.e.\,the one that we also employ in the present paper -- to show that the Feller semigroup indeed delivers weak solutions to the particle system.

\medskip

(\textit{ii}) It suffices for us to show that the drift $K \star w$ satisfies $\||K \star w|^{\frac{1}{\alpha-1}}\|_{E_{1+\epsilon}} \leq \||K|^{\frac{1}{\alpha-1}}\|_{M_{1+\epsilon}}~\big(\leq c_{d,\alpha,p,\epsilon}\big)$ and then apply the regularity result of Theorem \ref{thm2} to the equation
$$
\lambda w_\lambda+	\partial_t w_\lambda +(-\Delta)^{\frac{\alpha}{2}} w_\lambda - {\rm div\,}\big((K \star w)w_\lambda\big)=\delta_{t=0}\delta_x.
$$
 Indeed, 
 \begin{align*}
\||K \star w|^{\frac{1}{\alpha-1}}\|_{E_{1+\epsilon}} & =\sup_{(t,x) \in \mathbb{R}^{1+d}, r>0} r \biggl(\frac{1}{|C_r^\alpha|}\int_{C_r^\alpha(t,x)}|K \star w|^{\frac{1+\epsilon}{\alpha-1}}  \biggr)^{\frac{1}{1+\epsilon}} \\
& = \sup_{(t,x) \in \mathbb{R}^{1+d}, r>0} r \biggl(\frac{1}{|C_r^\alpha|}\int_t^{t+r^\alpha}\|K \star w\|^{\frac{1+\epsilon}{\alpha-1}}_{L^{\frac{1+\epsilon}{\alpha-1}}(B_r(x))} ds\biggr)^{\frac{1}{1+\epsilon}} \\
& = \sup_{(t,x) \in \mathbb{R}^{1+d}, r>0} r \biggl(\frac{1}{|C_r^\alpha|}\int_t^{t+r^\alpha}\big\|\int_{\mathbb R^d}K(\cdot-y)w(s,y)dy\big\|^{\frac{1+\epsilon}{\alpha-1}}_{L^{\frac{1+\epsilon}{\alpha-1}}(B_r(x))} ds \biggr)^{\frac{1}{1+\epsilon}} \\
& \leq \sup_{(t,x) \in \mathbb{R}^{1+d}, r>0} r \biggl(\frac{1}{|C_r^\alpha|}\int_t^{t+r^\alpha}\biggl(\int_{\mathbb R^d}\big\|K(\cdot-y)\big\|_{L^{\frac{1+\epsilon}{\alpha-1}}(B_r(x))} w(s,y)dy\biggr)^{\frac{1+\epsilon}{\alpha-1}} ds \biggr)^{\frac{1}{1+\epsilon}} \\
& \le \sup_{(t,z) \in \mathbb{R}^{1+d}, r>0} r \biggl(\frac{1}{|C_r^\alpha|}\|K\|^{\frac{1+\epsilon}{\alpha-1}}_{L^{\frac{1+\epsilon}{\alpha-1}}(B_r(z))}  \int_t^{t+r^\alpha}\biggl(\int_{\mathbb R^d}w(s,y)dy\biggr)^{\frac{1+\epsilon}{\alpha-1}} ds \biggr)^{\frac{1}{1+\epsilon}} \\
& \leq \sup_{z \in \mathbb R^{d}, r>0} r \biggl(\frac{1}{|B_r|}\|K\|^{\frac{1+\epsilon}{\alpha-1}}_{L^{\frac{1+\epsilon}{\alpha-1}}(B_r(z))}\biggr)^{\frac{1}{1+\epsilon}}=\||K|^{\frac{1}{\alpha-1}}\|_{M_{1+\epsilon}},
\end{align*}
where at the last step we have used $\int_{\mathbb R^d}w(s,y)dy \leq 1$. (Above we have, in fact, reproved for the reader's convenience a special case of Young's inequality for Morrey spaces \cite{BT}.)
\hfill \qed

\bigskip

\appendix

\section{Counterexample to weak existence}

\label{counter_app}

Below we provide a detailed justification of the counterexample to weak existence (Example \ref{counter_ex}). In fact, it only remains for us to prove identity \eqref{ito_cor} assuming that there exists a weak solution $(X_t,Z_t)$ to SDE \eqref{sde1} with drift $b(x)=\kappa x|x|^{-\alpha}$, $\kappa>0$, ($b(0)=0$), departing from the origin at time $t=0$.

\begin{lemma}
\eqref{ito_cor} holds.
\end{lemma}

\begin{proof}
Set
$$
f(x)=|x|^\beta,\qquad \beta\in(1,\alpha).
$$
Since $f$ is not $C^2$ at the origin, we regularize it by
$$
f_\varepsilon(x):=(|x|^2+\varepsilon^2)^{\frac{\beta}{2}},\qquad 0<\varepsilon\le 1.
$$
Then $f_\varepsilon\in C^2(\mathbb R^d)$. Let
$$
\tau_R:=\inf\{t\ge0:\ |X_t|\ge R\},\qquad R>0
$$
and
$$
\sigma_m
:=
\inf\left\{
u\ge0:
\int_0^u g(X_s)\,ds\ge m
\right\},
\qquad m\ge1,
$$
where
$$
g(x)=
\begin{cases}
	|x|^{1-\alpha}, \qquad &x\ne 0,\\
	0, &x= 0.\\
\end{cases}
$$
Since $|b(x)|=\kappa g(x)$, the definition of a weak solution gives, for every $t>0$,
$$
\int_0^t g(X_s)\,ds
=\frac{1}{\kappa}\int_0^t |b(X_s)|\,ds
<\infty
\qquad\text{a.s.}
$$
It follows that $\sigma_m\uparrow\infty$ almost surely. 

Applying It\^o's formula \eqref{ito} to
$
\chi_n f_\varepsilon,
$
where $\chi_n\in C_c^\infty(\mathbb R^d)$ satisfies
$
0\le \chi_n\le1$,
$\chi_n=1$ on  $B_n(0),$
(so that $\chi_n f_\varepsilon$ is bounded and has bounded derivatives), 
 and then letting \(n\rightarrow\infty\), we obtain
\begin{equation}
\label{temp_id}
\mathbb E f_\varepsilon(X_{t\wedge\tau_R\wedge\sigma_m})
=
f_\varepsilon(0)
+
\mathbb E\int_0^{t\wedge\tau_R\wedge\sigma_m}
\Big(A f_\varepsilon(X_s)-b(X_s)\cdot \nabla f_\varepsilon(X_s)\Big)\,ds,
\end{equation}
where, throughout this appendix,
$$
A:=-(-\Delta)^{\frac{\alpha}{2}}.
$$

We now justify the passage to the limit $\varepsilon\downarrow 0$ in \eqref{temp_id}, in order to obtain the sought identity \eqref{ito_cor}. 

\medskip

Step 1. Convergence of the left-hand side of \eqref{temp_id}. 
Using the SDE, we find
\[
X_{t\wedge\tau_R\wedge\sigma_m}
=
-\int_0^{t\wedge\tau_R\wedge\sigma_m} b(X_s)\,ds+Z_{t\wedge\tau_R\wedge\sigma_m}.
\]
Moreover, by the definition of \(\sigma_m\),
\[
\int_0^{t\wedge\tau_R\wedge\sigma_m} |b(X_s)|\,ds
=
\kappa\int_0^{t\wedge\tau_R\wedge\sigma_m} g(X_s)\,ds
\le \kappa m.
\]
Hence
\begin{align}\label{X_s_bound}
|X_{t\wedge\tau_R\wedge\sigma_m}|
\le
\kappa m+\sup_{0\le s\le t}|Z_s|.
\end{align}
Since \(\beta<\alpha\), we have
\begin{align}\label{Z_s_exp}
\mathbb E\sup_{0\le s\le t}|Z_s|^\beta<\infty.
\end{align}
Thus \(f_\varepsilon(X_{t\wedge\tau_R\wedge\sigma_m})\) is dominated by an integrable random variable independent of \(\varepsilon\), and therefore
\[
\mathbb E f_\varepsilon(X_{t\wedge\tau_R\wedge\sigma_m})
\rightarrow
\mathbb E |X_{t\wedge\tau_R\wedge\sigma_m}|^\beta.
\]

Step 2. Convergence of the right-hand side of \eqref{temp_id}. 
Clearly, $f_\varepsilon(0)\rightarrow0$. We next pass to the limit in the drift term and the fractional Laplacian term:

1) For the drift term, we have
\[
b(x)\cdot\nabla f_\varepsilon(x)
=\kappa\beta (|x|^2+\varepsilon^2)^{\beta/2-1}|x|^{2-\alpha}.
\]
Since \(\beta<2\) and \(b(0)=g(0)=0\), for every \(x\in\mathbb R^d\),
\[
0\le b(x)\cdot\nabla f_\varepsilon(x)
\le \kappa\beta |x|^{\beta-1}g(x),
\]
and
\[
b(x)\cdot\nabla f_\varepsilon(x)
\longrightarrow \kappa\beta |x|^{\beta-1}g(x)
\qquad\text{as }\varepsilon\downarrow0.
\]
On \(\{s< \tau_R\}\), we have \(|X_s|<R\), and hence
\begin{align}\label{b_grad_f}
0\le b(X_s)\cdot\nabla f_\varepsilon(X_s)
\le \kappa\beta R^{\beta-1}g(X_s).
\end{align}
Moreover, by the definition of \(\sigma_m\),
\[
\mathbb E\int_0^{t\wedge\tau_R\wedge\sigma_m} g(X_s)\,ds\le m.
\]
Thus the dominated convergence theorem yields
\[
\mathbb E\int_0^{t\wedge\tau_R\wedge\sigma_m} b(X_s)\cdot\nabla f_\varepsilon(X_s)\,ds
\longrightarrow
\kappa\beta\,\mathbb E\int_0^{t\wedge\tau_R\wedge\sigma_m} |X_s|^{\beta-1}g(X_s)\,ds.
\]
By the zero occupation time property  proved in Step 4, the last expression equals
\[
\kappa\beta\,\mathbb E\int_0^{t\wedge\tau_R\wedge\sigma_m} |X_s|^{\beta-\alpha}\,ds.
\]

2) For the fractional Laplacian term, we first show that 
\begin{align}
\label{fraclplacian_limit}
A f_\varepsilon(x)\rightarrow A|x|^\beta
=
c_{d,\alpha,\beta}|x|^{\beta-\alpha},
\qquad x\neq0.
\end{align}
Indeed, recall the L\'evy representation for the fractional Laplacian,
$$
A f(x)
=
\int_{\mathbb R^d}
\Big(f(x+z)-f(x)-\mathbf 1_{\{|z|\le1\}}\nabla f(x)\cdot z\Big)\nu(dz)
\quad \text{with}\  \nu(dz)=c_{d,\alpha}|z|^{-d-\alpha}\,dz.
$$
Define
$$
F_\varepsilon(x,z)
:=
f_\varepsilon(x+z)-f_\varepsilon(x)-\mathbf 1_{\{|z|\le1\}}\nabla f_\varepsilon(x)\cdot z.
$$
Fix $x\neq0$. We estimate \(F_\varepsilon(x,z)\)
separately on three regions:

For $|z|\le\min\{|x|/2,1\}$, we have $|x+z|\ge \frac{|x|}{2}$, and so
\begin{align}\label{F_bound1}
|F_\varepsilon(x,z)|
\le C|z|^2
\sup_{0\le\theta\le1}
|D^2f_\varepsilon(x+\theta z)|
\le C |x|^{\beta-2} |z|^2.
\end{align}

For $\min\{|x|/2,1\}<|z|\le1$, using \(|\nabla f_\varepsilon(x)|\le C|x|^{\beta-1}\)
and the Mean value theorem, we have
\[
|f_\varepsilon(x+z)-f_\varepsilon(x)|
\le C(|z|^\beta+|x|^{\beta-1}|z|),
\]
and then 
\begin{equation}
\label{F_bound12}
|F_\varepsilon(x,z)|
\le C(|z|^\beta+|x|^{\beta-1}|z|).
\end{equation}

For $|z|>1$, since
\begin{align}
|f_\varepsilon(x)|\le C(1+|x|^\beta),
\end{align}
we have
\begin{align}\label{F_bound2}
|F_\varepsilon(x,z)|
\le C(|x|^\beta+|z|^\beta).
\end{align}
Since $\alpha<2$ and $\beta<\alpha$, these bounds are
$\nu$-integrable on their corresponding regions. Therefore, by the Dominated convergence theorem,
we conclude \eqref{fraclplacian_limit}.

Next, we show that
\begin{align}\label{rhd_term2}
\mathbb E\int_0^{t\wedge\tau_R\wedge\sigma_m} A f_\varepsilon(X_s)\,ds
\rightarrow
c_{d,\alpha,\beta}\mathbb E\int_0^{t\wedge\tau_R\wedge\sigma_m}|X_s|^{\beta-\alpha}\,ds.
\end{align}
Using \eqref{F_bound1}-\eqref{F_bound2}, we have
\begin{align*}
	|A f_\varepsilon(x)|
	&\le \left(
	\int_{|z|\le \min\{|x|/2,1\}} +\int_{\min\{|x|/2,1\}<|z|\le1} +\int_{|z|>1}
	\right) |F_{\varepsilon}(x,z)| \nu(dz)\\
	&\le C (|x|^{\beta-\alpha}+|x|^\beta+1),
\end{align*}
while on $\{0<|x|\le R\}$ we have
\[
|A f_\varepsilon(x)|
\le C_R\big(1+|x|^{\beta-\alpha}\big).
\]
By the zero occupation time property proved in Step 4,
almost surely, for Lebesgue-almost every $0\le s<\tau_R$,
\[
|Af_\varepsilon(X_s)|
\le C_R(1+|X_s|^{\beta-\alpha})
\le C_R(1+R^{\beta-1}|X_s|^{1-\alpha}).
\]
Moreover, using the same property and the definition of
$\sigma_m$, we have
\[
\int_0^{t\wedge\tau_R\wedge\sigma_m}
|X_s|^{1-\alpha}\,ds
=
\int_0^{t\wedge\tau_R\wedge\sigma_m}
g(X_s)\,ds
\le m
\qquad\text{a.s.}
\]
Hence,
\[
\int_0^{t\wedge\tau_R\wedge\sigma_m}
|Af_\varepsilon(X_s)|\,ds
\le C_R(t+R^{\beta-1}m)
\qquad\text{a.s.}
\]
Together with \eqref{fraclplacian_limit}, the dominated
convergence theorem yields \eqref{rhd_term2}.

\medskip

Step 3.~Armed with the convergence results of Steps 1 and 2, we now pass to the limit $\varepsilon \downarrow 0$ in \eqref{temp_id} to obtain \eqref{ito_cor}.

Step 4.~ 
We now prove that, for every $t>0$,
$$
\int_0^t \mathbf{1}_{\{X_s=0\}}\,ds=0
\qquad\text{a.s.}
$$

By the L\'evy representation, the symmetry of $\nu$,
and the convexity of $f_\varepsilon$, we have
$$
Af_\varepsilon(x)
=\frac12\int_{\mathbb R^d}
\bigl(f_\varepsilon(x+z)+f_\varepsilon(x-z)-2f_\varepsilon(x)\bigr)\nu(dz)
\ge0.
$$
Moreover, changing variables $z=\varepsilon y$ yields
\begin{align*}
	Af_\varepsilon(0)
	=c_{d,\alpha}\int_{\mathbb R^d}
	\frac{(|z|^2+\varepsilon^2)^{\beta/2}-\varepsilon^\beta}
	{|z|^{d+\alpha}}\,dz
	=\varepsilon^{\beta-\alpha}
	c_{d,\alpha}\int_{\mathbb R^d}
	\frac{(1+|y|^2)^{\beta/2}-1}{|y|^{d+\alpha}}\,dy
	=a_{d,\alpha,\beta} \varepsilon^{\beta-\alpha}.
\end{align*}
Consequently,
$$
Af_\varepsilon(x)
\ge
a_{d,\alpha,\beta}\varepsilon^{\beta-\alpha}
\mathbf1_{\{x=0\}}.
$$
By \eqref{X_s_bound} and \eqref{Z_s_exp}, we have
\begin{align*}
	\mathbb E f_\varepsilon(X_{t\wedge\tau_R\wedge\sigma_m})-\varepsilon^\beta
	&\le\mathbb E|X_{t\wedge\tau_R\wedge\sigma_m}|^\beta\\
	&\le 2^{\beta-1}
	\big((\kappa m)^\beta+
	\mathbb E\sup_{u\le t}|Z_u|^\beta\big)
	=:C_{t,m}<\infty,
\end{align*}
where $C_{t,m}<\infty$ is independent of $\varepsilon$ and $R$.
Also, by \eqref{b_grad_f}, we obtain
\begin{align*}
	\mathbb E\int_0^{t\wedge\tau_R\wedge\sigma_m} b(X_s)\cdot\nabla f_\varepsilon(X_s)\,ds
	\le\kappa\beta R^{\beta-1}
	\mathbb E\int_0^{t\wedge\tau_R\wedge\sigma_m} g(X_s)\,ds
	\le\kappa\beta R^{\beta-1}m.
\end{align*}
Combining these bounds with \eqref{temp_id}, we obtain
\begin{align*}
	a_{d,\alpha,\beta}\varepsilon^{\beta-\alpha}
	\mathbb E\int_0^{t\wedge\tau_R\wedge\sigma_m}\mathbf1_{\{X_s=0\}}\,ds
	&\le \mathbb E\int_0^{t\wedge\tau_R\wedge\sigma_m} Af_\varepsilon(X_s)\,ds\\
	&=\mathbb E f_\varepsilon(X_{t\wedge\tau_R\wedge\sigma_m})-\varepsilon^\beta
	+\mathbb E\int_0^{t\wedge\tau_R\wedge\sigma_m} b(X_s)\cdot\nabla f_\varepsilon(X_s)\,ds\\
	&\le C_{t,m}+\kappa\beta R^{\beta-1}m.
\end{align*}

Therefore,
$$
\mathbb E\int_0^{t\wedge\tau_R\wedge\sigma_m}\mathbf1_{\{X_s=0\}}\,ds
\le C_{t,m,R}\varepsilon^{\alpha-\beta}
\to 0 \qquad (\varepsilon\to0).
$$
Then letting $R,m\uparrow\infty$, we have
$$
\int_0^t\mathbf1_{\{X_s=0\}}\,ds=0
\quad\text{a.s.}
$$

\end{proof}

\bigskip

\makeatletter
\let\savedtocwrite\@tocwrite
\let\@tocwrite\@gobbletwo
\section*{AI usage }
\let\@tocwrite\savedtocwrite
\makeatother

This paper was written in 2025 and early 2026 without the use of AI. To prepare the present revised version, we used OpenAI ChatGPT Astra rather lightly, to find possible remaining errors after we have taken into account referees' comments. It did find a few remaining typos and minor errors, which we have corrected.

\end{document}